\documentclass[10pt]{article}
\usepackage[utf8]{inputenc} 
\usepackage[T1]{fontenc}
\usepackage{lmodern}

\usepackage{amsthm,amsmath,amsfonts,amssymb,stmaryrd,bbm}
\usepackage{mathtools}
\usepackage[shortlabels]{enumitem}
\usepackage{mathrsfs} 
\usepackage{cases}

\usepackage{tocloft}

\usepackage{subcaption} 

\usepackage{graphicx}
\usepackage[dvipsnames]{xcolor}
\usepackage{caption,tikz}

\usepackage[english]{babel}
\usepackage[colorlinks=true, linkcolor=blue, urlcolor=black, citecolor=blue,pdfstartview=FitH]{hyperref}

\newcommand{\ba}{{\bf{a}}}
\newcommand{\bb}{{\bf{b}}}
\newcommand{\bc}{{\bf{c}}}
\newcommand{\bl}{{\bm{\lambda}}}
\newcommand{\bmu}{{\bm{\mu}}}
\newcommand{\bx}{{\bf{x}}}

\numberwithin{equation}{section}

\let\originalleft\left
\let\originalright\right
\renewcommand{\left}{\mathopen{}\mathclose\bgroup\originalleft}
\renewcommand{\right}{\aftergroup\egroup\originalright}

\newlength{\bibitemsep}
\newlength{\bibparskip}
\let\oldthebibliography\thebibliography
\renewcommand\thebibliography[1]{\oldthebibliography{#1}
  \setlength{\parskip}{\bibitemsep}
  \setlength{\itemsep}{\bibparskip}}

\DeclareMathOperator{\re}{Re}
\DeclareMathOperator{\im}{Im}
\DeclareMathOperator{\Tr}{Tr}
\DeclareMathOperator{\Var}{Var}

\DeclareMathOperator{\OO}{O}
\DeclareMathOperator{\oo}{o}

\DeclareMathOperator{\diag}{diag}

\newcommand{\ii}{\mathrm{i}}

\newcommand{\diff}{\mathop{}\mathopen{}\mathrm{d}}
\newcommand{\rd}{\mathop{}\mathopen{}\mathrm{d}}

\usepackage{bm}

\newcommand{\CC}{\mathbb{C}}
\theoremstyle{plain} 
\newtheorem{theorem}{Theorem}[section]
\newtheorem{lemma}[theorem]{Lemma}
\newtheorem{corollary}[theorem]{Corollary}
\newtheorem{proposition}[theorem]{Proposition}

\newtheorem{definition}[theorem]{Definition}
\newtheorem{remark}[theorem]{Remark}

\renewcommand{\epsilon}{\varepsilon}
\newcommand{\e}{{\varepsilon}}
\renewcommand{\leq}{\leqslant}
\renewcommand{\geq}{\geqslant}

\renewcommand{\P}{\mathbb{P}}
\newcommand{\E}{\mathbb{E}}
\newcommand{\R}{\mathbb{R}}
\newcommand{\C}{\mathbb{C}}
\newcommand{\N}{\mathbb{N}}

\newcommand{\1}{\mathbbm{1}}

\newcommand{\cN}{\mathcal{N}}

\newcommand{\pa}[1]{\left({#1}\right)}

\newcommand{\abs}[1]{\lvert #1 \rvert}
\newcommand{\absa}[1]{\left\lvert #1 \right\rvert}
\newcommand{\norm}[1]{\lVert #1 \rVert}

\newcommand{\vertiii}[1]{{\left\vert\kern-0.25ex\left\vert\kern-0.25ex\left\vert #1 
    \right\vert\kern-0.25ex\right\vert\kern-0.25ex\right\vert}}

\theoremstyle{plain} 

\makeatletter
\renewcommand{\subsection}{\@startsection
{subsection}
{2}
{0mm}
{-\baselineskip}
{0 \baselineskip}
{\normalfont\bf\itshape}} 
\makeatother

\makeatletter
\renewcommand{\subsubsection}{\@startsection
{subsubsection}
{3}
{0mm}
{-\baselineskip}
{0 \baselineskip}
{\normalfont\bf\itshape}} 
\makeatother

\def\author#1{\par
    {\centering{\authorfont#1}\par\vspace*{0.05in}}
}

\def\titlefont{\fontsize{13}{15}\bfseries\boldmath\selectfont\centering{}}
\def\authorfont{\fontsize{13}{15}}
\def\abstractfont{\fontsize{8}{10}}

\let\affiliationfont\rhfont

\def\address#1{\par
    {\centering{\affiliationfont#1\par}}\par\vspace*{11pt}
}

\def\title#1{
    \thispagestyle{plain}
    \vspace*{-14pt}
    \vskip 79pt
    {\centering{\titlefont #1\par}}%
    \vskip 1em
}

\def\body{
\setcounter{footnote}{0}
\def\thefootnote{\alph{footnote}}
\def\@makefnmark{{$^{\rm \@thefnmark}$}}
}

\renewenvironment{abstract}{\par%
    \vspace*{6pt}\noindent 
    \abstractfont
    \noindent\leftskip10pt\rightskip10pt
}{%
  \par}

\usepackage{tocloft}

\makeatletter
\renewcommand{\section}{\@startsection
{section}
{1}
{-3mm}
{-2\baselineskip}
{1\baselineskip}
{\normalfont\large\scshape\centering}} 
\makeatother

\makeatletter
\newcommand{\mynameis}[1]{#1\renewcommand{\@currentlabel}{#1}}
\makeatother

\begin{document}

~\vspace{-1cm}

\title{The $\log$-Characteristic Polynomial of Generalized Wigner Matrices is Log-Correlated}

\vspace{0.7cm}
\noindent

\noindent\makebox[\textwidth][c]{
 \begin{minipage}[c]{0.31\textwidth}
 \author{Krishnan Mody}
\address{
Courant Institute\\    km2718@nyu.edu}
 \end{minipage}
 }
\vspace{0.1cm}
\noindent
 
\begin{abstract}
	We prove that in the limit of large dimension, the distribution of the logarithm of the characteristic polynomial of a generalized Wigner matrix converges to a log-correlated field.
	In particular, this shows that the limiting joint fluctuations of the eigenvalues are also log-correlated. 
	Our argument mirrors that of \cite{BouMod2019}, which is in turn based on the three-step argument of \cite{ErdPecRmSchYau2010,ErdSchYau2011Uni}, 
	but applies to a wider class of models, and at the edge of the spectrum. 
	We rely on (i) the results in the Gaussian cases, special cases of the results in \cite{BouModPai2021}, 
	(ii) the local laws of \cite{ErdYauYin2012}
	(iii) the observable \cite{Bou2020} introduced and its analysis of the stochastic advection equation this observable satisfies, and
	(iv) the argument for a central limit theorem on mesoscopic scales in \cite{LanLopSos2021}.
	For the proof, we also establish a Wegner estimate and local
	law down to the microscopic scale, both at the edge of the spectrum. 
\end{abstract}

{\renewcommand{\contentsname}{}
	\hypersetup{linkcolor=black}
	\tableofcontents
}

\vspace{0.1cm}

\section{Introduction}

We are interested in this article in the log-characteristic polynomial of generalized Wigner matrices. We prove central limit theorems
for both the real and imaginary parts of this object, or in other words for the electric potential and the joint fluctuations of the eigenvalues.
In both cases, we show that appropriately centered and normalized, the log
characteristic polynomial is log-correlated in the limit of large dimension.

To give context to our result, we give a partial survey of the literature regarding fluctuations of the log-characteristic polynomial of a random matrix.
We first highlight those results which showed that the log-characteristic polynomial converges in law, in the sense of distributions, to a log-correlated field,
and then later results which showed convergence in law in a point-wise sense.

Interest in the log-characteristic polynomial goes back to at least \cite{KeaSna2000}, where the authors considered the log-characteristic polynomial 
of Haar distributed unitary matrices (or the circular unitary ensemble, CUE) as a model for log of the $\zeta$-function on the critical axis, and proved a central limit theorem
for this object analogous to Selberg's for the $\zeta$-function \cite{Sel1946}. 
Then in \cite{HugKeaOCo2001}, the authors proved
that for any $\e > 0$, the real and imaginary parts of the log characteristic polynomial, evaluated on the unit circle, converge in law in $H_0^{-\e} \times H_0^{-\e}$ to
a log-correlated field, the Gaussian free field on the unit circle. 

Results in the distributional sense, i.e. results for the log characteristic polynomial integrated against a smooth test function,
also hold for other random matrix models.
In the case of the GUE (Gaussian unitary ensemble), \cite{FyoKhoSim2016} proved distributional convergence of the log-characteristic polynomial to a random Fourier series.
For a class of general $\beta$-ensembles with eigenvalues $\lambda_1, \dots, \lambda_N$ and a sufficiently smooth function $f$, \cite{Johansson1998} proved 
the convergence in distribution as $N \to \infty$,
	\begin{equation}
	\label{eq:clt_smooth_f}
		\frac{\sum_{k=1}^N f\pa{\lambda_k} -\E \pa{\sum_{k=1}^N f\pa{\lambda_k} } }{ \sigma(f)} \to \mathscr{N}(0,1),
	\end{equation}
with an explicit formula for $\sigma^2(f)$, given essentially by $\norm{f}_{H^{1/2}}$.
This result applies in particular to the log-characteristic polynomial integrated against a smooth test function.
A similar result to \eqref{eq:clt_smooth_f} holds in the Wigner case \cite{LytPas2009}, and also applies to the log-characteristic polynomial in the sense of distributions.

Concerning pointwise results, we first mention some central limit theorems valid for the log-characteristic polynomial evaluated at one point.
In \cite{TaoVu2012}, Tao and Vu proved a central limit theorem for the log-determinant
of a Wigner matrix with first 4 moments matching the Gaussian ones. The first step of their argument proves the result for Gaussian ensembles with general $\beta$, 
based on the tridiagonal model of \cite{DumEde2002}. Study of the log-characteristic polynomial for this model is significantly harder than for the log-determinant,
and \cite{AugButZei2020} proved a central limit theorem for the log-characteristic polynomial evaluated in the bulk of the spectrum,
for a general class of tridiagonal models that includes this case. In \cite{LamPaq20202}, the authors proved a similar result for the G$\beta$E at the edge.
And recently, \cite{LanLopSos2021} proved a central limit theorem for the fluctuations of a single eigenvalue in the bulk for Wigner-type matrices, a more general
model than the generalized Wigner matrices we consider in this paper.

Returning to the log-correlated picture, Bourgade \cite{Bou2010} showed that the real and imaginary parts of the log-characteristic
polynomial of a CUE distributed matrix are log-correlated in the limit of large dimension. 
(In the same paper Bourgade shows that log of the $\zeta$-function is also log-correlated near the critical axis.)
In the case of the GUE (Gaussian unitary ensemble) and GOE (Gaussian orthogonal ensemble), 
\cite{Gus2005, Oro2010} showed that the field of appropriately centered and rescaled eigenvalues, or the imaginary part of the log characteristic polynomial,
is also log-correlated in the limit of large dimension. 

Generalizing the results for the GOE and GUE, \cite{BouModPai2021} proves that both the real and imaginary parts of the log-characteristic polynomial are log-correlated in the large dimension limit,
for a general class of one dimensional $\beta$-ensembles.
Based on the result \cite{BouModPai2021} in the GOE and GUE cases, 
\cite{BouMod2019} proves the log-correlated picture holds for the real and imaginary parts of the log-characteristic polynomial in the bulk,
for a (standard) Wigner matrix.
In the more general setting of generalized Wigner matrices, using \cite{Gus2005, Oro2010} as a point of comparison,
\cite{Bou2020} shows that the imaginary part of the log-characteristic-polynomial is log-correlated at the edge of the spectrum. 
In this paper, we build on arguments from \cite{Bou2020, BouMod2019} to cover essentially the entire spectrum, and the real case.

The argument in \cite{BouMod2019} uses a Wegner estimate and local law on the optimal scale as input to a three-step-strategy \cite{ErdPecRmSchYau2010,ErdSchYau2011Uni}
approach to convert the relevant question into a question about fluctuations of linear statistics of a function on an almost macroscopic scale.
For Wigner matrices in the bulk, the Wegner estimate, the local law on the optimal scale, and central limit theorems sufficiently strong to close the argument
were available in the literature. For generalized Wigner matrices at the edge, we use arguments from \cite{Bou2020} to prove the Wegner estimate
and local law valid on the optimal scale. For a central limit theorem to close the argument, we follow the strategy of \cite{LanLopSos2021}, which applies to
Wigner-Type matrices in the bulk. For generalized Wigner matrices, the calculation there simplifies and allows us to reach the edge, and therefore
to conclude our proof.

We mention that our discussion here ignores that the maximum of the log-characteristic polynomial, and the characteristic polynomial itself are also objects of recent interest. 
See for example \cite{PaqZei2022} for an exciting recent result concerning the maximum of the log-characteristic polynomial for the C$\beta$E, and a discussion of the 
history of this problem. (See also \cite{BouModPai2021} for an analogous conjecture in the case of $\beta$-ensembles and Wigner matrices).
And for a recent result regarding the characteristic polynomial of the CUE and its relationship to the Gaussian Multiplicative Chaos, see \cite{BouFal2022} and the discussion therein.
Finally, for a broad survey on log-correlated fields, including beyond random matrix theory, see \cite{Arg2017}.

\subsection{Model Assumptions. }

In this paper, we write $a \sim b$ to indicate that there exists $C > 0$ such that $C^{-1} b < a < C b$. For complex
valued functions, $a \sim b$ means $\re a \sim \re b$ and $\im a \sim \im b$.
\begin{definition}
	A generalized Wigner matrix $H=H(N)$ is a symmetric or Hermitian $N \times N$ matrix whose upper-triangular elements
	$H_{ij} = \overline{H}_{ji}$, $i \leq j$, are independent random variables with mean zero and variances 
	$\sigma_{ij}^2 = \E(\abs{H_{ij}}^2)$ that satisfy $\sigma_{ij}^2 \sim N^{-1}$ for all $i, j \in \llbracket 1, N \rrbracket$,
	and 
	\[
		\sum_{j=1}^N \sigma_{ij}^2 = 1 \text{ for every } 1 \leq i \leq N.
	\]
	In the Hermitian case, we further assume $\Var \re(H_{ij}) \sim \Var \im(H_{ij})$ and that
	$\re(H_{ij})$ and $\im(H_{ij})$ are independent. 
	Finally, we assume uniform sub-exponential decay of the entries, namely that there exists a constant $\vartheta > 0$ such that for all $i,j \in \llbracket 1, N \rrbracket$, and 
	$x > 0$,
	\[
		\P\pa{
			\absa{H_{ij}} > x\sigma_{ij} 
		}
		\leq
		\vartheta^{-1}\exp\pa{-x^\vartheta}.
	\]
\end{definition}

For the sake of brevity, we will often write Wigner matrix in place of generalized Wigner matrix.
It is likely that in what follows one can replace the sub-exponential decay condition with the existence of finitely many moments, however
we do not pursue this direction. We also highlight the special case $H_{ij} \sim \mathscr{N}(0, (1+\delta_{ij})N^{-1})$, i.i.d. up to
the symmetry condition,
(respectively $H_{ij} \sim \mathscr{N}(0,(1+\delta_{ij})/2 N^{-1}) + \ii \mathscr{N}(0,(1+\delta_{ij})/2 N^{-1})$ for $i < j$ and
$H_{ii} \sim \mathscr{N}(0,N^{-1})$) which recovers the Gaussian Orthogonal Ensemble, or GOE (respectively Gaussian Unitary Ensemble, or GUE).
This case is integrable, and the base case for the arguments in this paper.

\subsection{Main Results. }

In order to state our results, we introduce the log-characteristic polynomial of a generalized Wigner matrix, centered (at first order, as we will see). To do so,
we define $\log$ using the principal branch, and extend it to the negative real numbers by continuity from above,
namely $\log(re^{\ii \theta}) = \log(r) + \ii \theta$ for any $r > 0$ and $\theta \in (-\pi, \pi]$.
For $z \in \CC$ and $\bl = (\lambda_1, \dots, \lambda_N)$, let
\begin{equation}
	\label{def:L_N}
	L_N(\bl, z) = \sum_{j=1}^N \log\pa{z-\lambda_j} - N \int \log\pa{z-x} \diff \varrho(x),
\end{equation}
where 
	\[
		\varrho(x) = \frac{1}{2\pi} \sqrt{(4-x^2)_+}
	\]
denotes the limiting spectral density for (generalized) Wigner matrices. 
Further, for $z = E + \ii \eta$, set
\begin{equation} \label{def:l(E)_and_kappa(E)}
    \kappa(z) = \abs{z+2} \wedge \abs{z-2}
	\qquad \text{and} \qquad
    \ell(E) =
    \begin{cases} 
    N^{-1} \kappa(E)^{-1/2} & \text{if } E \in [-2+N^{-2/3}, 2-N^{-2/3}] \\
    N^{-2/3} & \text{otherwise}.
    \end{cases} 
\end{equation}
Here $\ell(E)$ denotes the typical spacing between eigenvalues near $E$, and $\kappa(E)$ is the distance from $E$ to the edge of the spectrum.
Our first main result says that $L_N(\bl, z)$ is log-correlated as $N \to \infty$. 

\begin{theorem}
	\label{thm:log_corr_field}
	For fixed $\e, c > 0$, define the subset of the spectrum, 
	\begin{equation}
	\label{def:G_e,c}
		\mathcal{G}_{\e, c} = 
		\{ -2 \leq E \leq 2 \}
		\cap \{
				\{ \kappa(E) < N^{-\e} \} \cup \{ \kappa(E) \geq c \}
			\},
	\end{equation}
	and for fixed $m \geq 1$, let $(E_1,\dots,E_m)_{N \geq 1} \in \mathcal{G}_{\e, c}^m$, possibly depending on $N$. 
	Let
	\[
		\delta_i=\frac{1}{4} \pa{ \frac{2}{\beta}-1} \log \pa{\kappa(E_i) \vee N^{-2/3}}, \quad i =1, \dots, m
	\]
	and assume that the limits,
	\[
		a_{ij} = \lim_{N \to \infty} 
		\frac{\log (\abs{E_i-E_j} \vee \ell(E_i))}{-\log N} 
		\qquad \text{and} \qquad
		b_{ij} = \lim_{N \to \infty} 
		\frac{\log \Bigl( 
		    \frac{\abs{E_i-E_j} \vee \ell(E_i)}
		    {\kappa(E_i)}
		    \wedge 1
		    \Bigr)}{-\log N},
	\]
	exist. Denote $\ba=(a_{ij})_{1 \leq i,j \leq m}$ and $\bb =(b_{ij})_{1 \leq i,j \leq m}$,
	and write $L_N(E) = L_N(E, \bl)$, where $\bl=(\lambda_1, \dots, \lambda_N)$ is distributed as the eigenvalues of a 
	generalized Wigner matrix.
	Then, we have convergence in distribution
	\[
		\sqrt{\frac{\beta}{\log N}}	
		\pa{\re L_N(E_1) - \delta_1, \dots, \re L_N(E_m) - \delta_m,
		\im L_N(E_1), \dots, \im L_N(E_m})
		\xrightarrow[N\to\infty]{}
		\mathscr{N} \left( 0, \left( \begin{matrix}
		\ba & 0 \\ 0 & \bb
		\end{matrix} \right)\right).
	\] 
\end{theorem}

Theorem \ref{thm:log_corr_field} implies a central theorem for joint eigenvalue fluctuations near the edge. For its statement, and for frequent
use in this paper, define $\gamma_k$ via 
	\begin{equation}
	\label{eq:def_quantiles}
		\frac{k}{N} = \int_{-\infty}^{\gamma_k} \rd \varrho(x), 
		\quad 
		1 \leq k \leq N.
	\end{equation}
Further, define the centered and normalized eigenvalues, 
	\[
			Y_N(n) = \pi N \sqrt{\frac{\beta}{\log N}}\varrho(\gamma_n) (\lambda_n - \gamma_n),
			\quad
			1 \leq n \leq N,
		\]
taking $\beta = 1$ for real Wigner matrices, and $\beta = 2$ for complex ones.
With our definition of $\log$, for $x \in \R^*$, $\im \log(x) = \pi \1_{x < 0}$
and therefore $\frac{1}{\pi} \sum_{k=1}^N \im \log\pa{E-\lambda_k}$ counts the number of eigenvalues $\lambda_k$ larger than $E$.
Arguing as in \cite[Lemma B.2]{BouMod2019}, we therefore recover a joint central limit theorem for the $Y_N$'s, see \cite[Theorem 1.6]{Bou2020}.

\begin{corollary}
\label{corr:eig_fluct}
Fix $\alpha, \delta \in(0,1)$.
For fixed $m\geq 1$, let $n_1, \dots, n_m \in \llbracket 1,N^\delta \rrbracket \cup \llbracket \alpha N, (1-\alpha) N \rrbracket$ where each $n_i$ may depend on $N$.
Assume that for every $1 \leq i,j \leq m$, the limit
	\[
		c_{ij} = \lim_{N \to \infty} 
		\frac{\log \Bigl( 
		    \frac{\abs{\gamma_{n_i}-\gamma_{n_j}} \vee \ell(\gamma_{n_i})}
		    {\kappa(\gamma_{n_i})}
		    \wedge 1
		    \Bigr)}{-\log N}
	\]
exists, and denote  $\bc = (c_{ij})_{1\leq i, j \leq m}$. Then, we have the convergence in distribution,
   	\[
        		\pa{Y_N(n_1), \dots, Y_N(n_m)} \xrightarrow[N\to\infty]{} \mathscr{N}(0,\bc).
    	\]
\end{corollary}

Our proof of Theorem \ref{thm:log_corr_field} follows a similar strategy to \cite{BouMod2019, BouModPai2021}, and 
relies on the so-called Wegner estimate and local law on an optimal scale, which we state next. These results may be of independent interest.

\begin{theorem}
	\label{thm:wegner_estimate}
	Fix $\e,c > 0$.
	For any $E \in \mathcal{G}_{\e, c}$ as in \eqref{def:G_e,c}, let $I = [E-\delta_N\ell(E),E+\delta_N \ell(E)]$ with any $\delta_N \to 0$ as $N \to \infty$, and
	let $\cN(I) = \abs{\{ k : \lambda_k \in I \}}$. Then uniformly in $E$, $\P(\cN(I) > 0) \to 0$ as $N \to \infty$.
\end{theorem}

In other words, the above Wegner estimate says that with probability going to one, there will be no eigenvalues in any
interval of sub-microscopic size around $E$. Note that in the bulk of the spectrum the Wegner estimate follows from \cite{BouErdYauYin2016}, see
Corollaries 2.6 and 2.7 in \cite{BouMod2019}. 
We mention also that for any $k \geq 1$, \cite[Appendix B]{BenLop2022} proves so-called overcrowding estimates, informally, bounds of the form $\P(\mathcal{N}(I) \geq k) \leq N^{-ck}$.
While these bounds are quantitative, we cannot use them for our purpose where we need to take $k = 0$.\\

\begin{remark}
	We expect that
	Theorem \ref{thm:log_corr_field} holds for $E_i \in [-2,2]$, $1 \leq i \leq m$,
	Corollary \ref{corr:eig_fluct} holds for $n_i \in \llbracket 1, N \rrbracket$, $1 \leq i \leq m$,
	and the Wegner estimate holds for $E \in [-2,2]$. Our restrictions in these results 
	arise because we use 
	\cite[Theorem 2.2]{BouErdYauYin2016} and \cite[Theorem 3.1 and Lemma 3.4]{Bou2020} 
	as inputs to our argument, and as stated, these hold for $E \in [-2+\kappa, 2-\kappa]$, any $\kappa > 0$.
	However, we expect the arguments that give these results extend to encompass
	$E \in [-2+N^{-c}, 2-N^{-c}]$ for some small, fixed $c > 0$. 
\end{remark}

For the statement of the optimal local law, we introduce the notation
	\begin{equation}
	\label{eq:stiltjes_transform}
		s(z) = \frac{1}{N} \sum_{k=1}^N \frac{1}{\lambda_k - z},
		\quad
		m(z) = \int \frac{\rd \varrho(x)}{x-z},
	\end{equation}
valid for any Wigner matrix with eigenvalues $\bl$. When $\lambda_k=\lambda_k(t)$, we may write $s_t(z)$ to emphasize that $s$ depends on $t$. Let
	\begin{equation}
	\label{def:D_eps}
		\mathcal{D}_\e = \left\{
			z = E + \ii \eta,\, E \in [-2,2],\, \kappa(E) \leq N^{-\e},\, N^{-1} \leq \eta \leq 1
		\right\}.
	\end{equation}

\begin{theorem}
\label{thm:local_law}
	For any $\e > 0$ and $p \geq 1$, there exists a constant $C_p$ depending on $p$ such that for any $N \geq 1$,
		\[
			\sup_{z \in \mathcal{D}_\e} \E\pa{
				\absa{
					s(z) - m(z)
				}^{2p}
			}
			\leq
			\frac{C_p}{(N \eta )^{2p}}.
		\]
\end{theorem}

\begin{remark}
	Our proof of Theorem \ref{thm:log_corr_field} does not use the full strength of Theorem \ref{thm:local_law}, and indeed
	in the bulk we rely on the weaker estimate Proposition \ref{prop:weak_local_law_bulk}. 
	In \cite{LanLopSos2021}, the authors used a different method to prove
	a central limit theorem for fluctuations of individual eigenvalues in the bulk.
	Such an argument does not control the characteristic polynomial, which requires control over all eigenvalues 
	simultaneously. However, the result of \cite{LanLopSos2021} is valid for the more general class of Wigner-Type matrices.
\end{remark}

Finally, our proof of Theorem \ref{thm:log_corr_field} relies on a central limit theorem at almost macroscopic scale. 
In Appendix \ref{sec:mesoscopic_CLT}, we explain how to modify the proof of \cite[Theorem 5.21]{LanLopSos2021} 
to prove a central limit theorem for linear statistics of a function on a mesoscopic scale, with support at the edge.
In fact \cite[Theorem 5.21]{LanLopSos2021} applies to the more general "Wigner-Type" class of models. 
However, at the edge, the limiting variance is challenging to analyze in this generality. For an expression sufficiently tractable
for our purpose, we restrict ourselves to the generalized Wigner case.
For details, see Propositions \ref{prop:char_function} and \ref{prop:var_GW}.

\begin{remark}
	After the initial posting of this paper, Riabov proved a central limit theorem for Wigner-Type matrices,
	valid at all types of singularities of the limiting spectrum \cite{Ria2023b}. We believe that with this result in place of Appendix \ref{sec:mesoscopic_CLT}, 
	one can use the techniques in this paper to prove the log-correlated picture in the Wigner-Type case, at least up to distance $N^{-1/5}$ from a regular edge. 
\end{remark}

\subsection{Outline of the paper. }
Before outlining the proof of Theorem \ref{thm:log_corr_field}, we first introduce the relevant notation. 
First, we will often say an event $A$, possibly depending on $N$, holds with overwhelming probability. By this we mean that
for any $D > 0$, there exists $N_0(D)$ such that for any $N > N_0$, we have
	\begin{equation*}
		\P(A) > 1 - N^{-D}.
	\end{equation*}
Next, we recall Dyson's Brownian Motion following \cite{Bou2020}. 
Let $\bmu(0)$ denote GOE distributed eigenvalues, and $\bl(0)$ denote Wigner distributed eigenvalues.
For $\nu \in [0,1]$, we define $x_k^{(\nu)}(t)$ as the unique strong solution (see \cite{McK1969}) to 
	\begin{equation}
	\label{eq:DBM}
		\rd x_k^{(\nu)} = \sqrt{\frac{2}{\beta N}} \rd B_k
		+ \pa{  \frac{1}{N} \sum_{\ell \neq k} \frac{1}{x_k^{(\nu)} - x_\ell^{(\nu)}} -  \frac{x_k^{(\nu)}}{2} }\rd t,
		\quad 
		x_k^{(\nu)}(0) = \nu \mu_k(0) + (1-\nu)\lambda_k(0).
	\end{equation}
In the remainder of this paper,
we will often suppress the dependence of quantities on $t$ and $\nu$ to ease notation.
Further, let
	\begin{equation}
	\label{eq:def_uk_f}
		\mathfrak{u}_k = \mathfrak{u}_k^{(\nu)}(t) = e^{ t/2} \frac{\rd}{\rd \nu} x_k^{(\nu)}(t), 
		\quad 
		f_t(z) = e^{-t/2} \sum_{k=1}^N \frac{\mathfrak{u}_k}{x_k^{(\nu)}(t) - z},
	\end{equation}
and note that 
	\begin{equation}
	\label{eq:pde_uk}
		\rd \mathfrak{u}_k = \frac{1}{N} \sum_{\ell \neq k} \frac{\mathfrak{u}_\ell - \mathfrak{u}_k}{(x_k-x_\ell)^2}.
	\end{equation}
We define $\mathfrak{v}_k$ as the solution to \eqref{eq:pde_uk} with initial condition $\mathfrak{v}_k(0) = \abs{\mathfrak{u}_k(0)}$, $k = 1, \dots, N$,
and introduce
	\begin{equation}
	\label{def:ft_tilde}
		\tilde{f}_t(z) = e^{-t/2}  \sum_{k=1}^N \frac{ \mathfrak{v}_k(t)}{x_k^{(\nu)}(t) - z}.
	\end{equation}
With the above notation, we now outline our proof of Theorem \ref{thm:log_corr_field} near the edge. 
Let $\bmu(t)$ and $\bl(t)$ be the result of running the dynamics \eqref{eq:DBM} with initial conditions $\bmu(0)$ and $\bl(0)$, respectively.
By Theorem \ref{thm:coupled_distance} (\cite[Theorem 2.8]{Bou2020}),  
for $0 < t < 1$ and any $\e > 0$, we have
		\begin{equation*}
			\abs{ \lambda_k(t) - \mu_k(t) } \leq \frac{N^\e}{Nt}, \quad 1 \leq k \leq N,
		\end{equation*}
with overwhelming probability. This means that by choosing an appropriately large $t$, near the edge (as in Theorem \ref{thm:log_corr_field}), 
$\abs{\lambda_k(t) - \mu_k(t)} \ll \ell(\gamma_k)$, meaning the gap between $\lambda_k(t)$ and $\mu_k(t)$ becomes sub-microscopic. 
Since the dynamics \eqref{thm:local_law} preserve the distribution of GOE eigenvalues, 
from the GOE case, we may infer a local law and Wegner estimate for the eigenvalues $\bl(t)$. 
In order to transfer these estimates to $\bl(0)$, we use four-moment matching arguments, 
to conclude Theorems \ref{thm:wegner_estimate} and \ref{thm:local_law}. \\

With Theorems \ref{thm:wegner_estimate} and \ref{thm:local_law}, we can prove Theorem \ref{thm:log_corr_field} following \cite{BouMod2019}.
Indeed, first we show that in order to prove Theorem \ref{thm:log_corr_field}, we can prove 
its statement with $L_N(E+\ii \eta(E))$, $\eta(E) = \exp( (\log N)^{1/4} ) \ell(E)$, in place of $L_N(E)$. 
For the precise statement, see Proposition \ref{prop:smoothing}. Extending \cite[Proposition 2.11]{Bou2020} to the edge,
we show further that for any $0<\e <1/100$, say, with $z_t$ the characteristic of a relevant advection equation,
	\[
		L_N(\bl(\tau), z) - L_N(\bmu(\tau), z) \approx L_N(\bl(0), z_\tau) - L_N(\bmu(0), z_\tau), \quad \tau = N^{-\e},
	\]
see Proposition \ref{prop:coupling}. 
To establish Theorem \ref{thm:log_corr_field} for $L_N(\bl(\tau), z)$, we want to estimate its variance, and 
since $\im z_\tau$ is at an almost macroscopic level, we may apply the results of Appendix \ref{sec:mesoscopic_CLT} 
to do so. We finally conclude the proof of Theorem \ref{thm:log_corr_field} at the edge with a four-moment matching argument,
see Section \ref{sec:end_of_proof}. Our argument in the bulk follows the same outline, with 
Proposition \ref{prop:weak_local_law_bulk} as a substitute for Theorem \ref{thm:local_law}. \\

\noindent{\bf Acknowledgement.} We thank Paul Bourgade for many helpful discussions during the completion of this work.\\

\section{Preliminaries} 
\label{section:stochastic_advection_equation}

In this section, we collect results which we will use in an essential way in the remainder of the paper. 
First, we recall the local law and rigidity from \cite{ErdYauYin2012}. Fix a large constant $C>0$ and let
	\begin{equation}
		\varphi = e^{C (\log\log N)^2}.
	\end{equation}
\begin{theorem}[Local Law]
	Let $H$ be a generalized Wigner matrix, $G(z) = (H-z)^{-1}$, and $m(z)$ be as in \eqref{eq:stiltjes_transform}. Then uniformly for $z$ in any compact set, 
	there exists $C > 0$ such that for any $D > 0$, 
	\begin{equation}
	\label{eq:local_law_entrywise_GW}
		\P\pa{
			\norm{G(z) - m(z) {\rm Id}}_\infty \geq \varphi^C \Psi(z)
		} \leq N^{-D},
		\text{ where } 
		\Psi(z) = \sqrt{  \frac{\varrho(z) }{N\eta} }+ \frac{1}{N\eta}.
	\end{equation}	
	Furthermore, with $\mathcal{S} = \{z = E + \ii \eta, \,\abs{E} \leq 5, \,\varphi/N \leq \eta \leq 10\}$, for any $D > 0$, 
	\begin{equation}
	\label{eq:local_law_original}
		\P\pa{
			\bigcup_{z \in \mathcal{S}} \left\{
				\abs{s(z) - m(z)} \geq \frac{\varphi}{N \im(z)}
			\right\}
		} \leq N^{-D}.
	\end{equation}
\end{theorem}

\begin{corollary}[Rigidity]
	\label{cor:rigidity}
	Let $\bl$ denote the eigenvalues of a generalized Wigner matrix, and
	let $\gamma_k$, $k = 1, \dots, N$ be as in \eqref{eq:def_quantiles}. 
	For $N$ sufficiently large and with $\ell(\cdot)$ as in \eqref{def:l(E)_and_kappa(E)}, 
	\[
		\P\pa{
			\max_{1 \leq k \leq N}\, \abs{ \lambda_k - \gamma_k } \geq \varphi^{1/2} \ell(\gamma_k)
		} \leq N^{-D}.
	\]
\end{corollary}

In fact, rigidity holds along the dynamics \eqref{eq:DBM}, see \cite[(2.7)]{Bou2020}. 

\begin{lemma}
\label{lem:rigidity}
	The event
	\begin{equation}
	\label{eq:rigidity_set}
		\mathcal{A} = \left\{
			\absa{x_k^{(\nu)}(t) - \gamma_k} < \varphi^{1/2} N^{-2/3} (\hat{k})^{-1/3} \text{ for all }
			0 \leq t \leq 1, \, k \in \llbracket 1, N \rrbracket, \, 0 \leq \nu \leq 1
		\right\}
	\end{equation}
	holds with overwhelming probability.
\end{lemma}

The following is \cite[Theorem 2.8]{Bou2020}. It is the fundamental input to our proofs of Theorems \ref{thm:wegner_estimate} and 
\ref{thm:local_law}.

\begin{theorem}
\label{thm:coupled_distance}
	For any $D > 0$ and $\e > 0$ there exists $N_0 = N_0(D, \e)$ such that for any $N > N_0$,
		\[
			\P\pa{
				\abs{ \mu_k(t) - \lambda_k(t)    } \geq \frac{N^\e}{Nt} \text{ for all } k \in \llbracket 1, N \rrbracket
				\text{ and } 0 < t < 1
			} \leq N^{-D}.
		\]
\end{theorem}

In the following two sections, we prove Theorems \ref{thm:wegner_estimate} and \ref{thm:local_law}.

\section{Wegner Estimate at the Edge}

In this section, we prove Theorem \ref{thm:wegner_estimate} at the edge. In the bulk of the spectrum the Wegner estimate follows from \cite{BouErdYauYin2016}, see
Corollaries 2.6 and 2.7 in \cite{BouMod2019}. 
As in Section \ref{sec:local_law}, we follow the three step strategy \cite{ErdPecRmSchYau2010,ErdSchYau2011Uni}.
In Proposition \ref{prop:dynamic_step}
we show that for $t$ sufficiently large, Theorem \ref{thm:wegner_estimate} holds at the edge for $\bl(t)$ which evolve according to the dynamics \eqref{eq:DBM}.
Then in Proposition \ref{prop:4_moment_matching}, we prove an almost 4 moment matching theorem from which we conclude.

\begin{proposition} \label{prop:dynamic_step} 
	Let $\bl(t)$ be the result of running the dynamics \eqref{eq:DBM} with initial condition $\bl(0)$, the eigenvalues of a Wigner matrix.
	Fix $\e > 0$, let $E \in [2-N^{-\e}, 2]$, $I = [E-\delta_N\ell(E),E+\delta_N \ell(E)]$ for any $\delta_N \to 0$ as $N \to \infty$, and
	let $\cN(I) = \abs{\{ k : \lambda_k(t) \in I \}}$. Then for  $t \geq N^{-\e/4}$,
	uniformly in $E$, $\P(\cN(I) > 0) \to 0$ as $N \to \infty$.
\end{proposition}
Proposition \ref{prop:dynamic_step} is a consequence of the Theorem \ref{thm:coupled_distance}, and the fact that the Wegner estimate holds in the Gaussian case.
\begin{proof}[Proof of Proposition \ref{prop:dynamic_step}]
	By Proposition 3.1 in \cite{BouModPai2021}, the statement of the Proposition holds for $\bmu(0)$, and
	since the distribution of $\bmu(t)$ is invariant along the dynamics \eqref{eq:DBM}, we have
		\begin{equation}
		\label{eq:beta_case}
			\lim_{N \to \infty} \P\pa{ \absa{\left\{ k : \mu_k(t) \in I \right\}} > 0} = 0.
		\end{equation}
	Write
		\begin{multline}
			\P\pa{ \absa{\left\{ k : \lambda_k(t) \in I \right\}} > 0} =
				\P\pa{ \absa{\left\{ k : \lambda_k(t) \in I \right\}} > 0 \text{ and } \absa{\left\{ k : \mu_k(t) \in I \right\}} > 0} \\
				+  \P\pa{ \absa{\left\{ k : \lambda_k(t) \in I \right\}} > 0 \text{ and } \absa{\left\{ k : \mu_k(t) \in I \right\}} = 0}.
		\end{multline}
	Then
		\[
			\P\pa{ \absa{\left\{ k : \lambda_k(t) \in I \right\}} > 0 \text{ and } \absa{\left\{ k : \mu_k(t) \in I \right\}} > 0}
			\leq \P\pa{ \absa{\left\{ k : \mu_k(t) \in I \right\}} > 0}, 
		\]
	and by Theorem \ref{thm:coupled_distance}, for any $\gamma > 0$ we have
		\[
			\P\pa{ \absa{\left\{ k : \lambda_k(t) \in I \right\}} > 0 \text{ and } \absa{\left\{ k : \mu_k(t) \in I \right\}} = 0} \leq
			\P\pa{
				\exists k \text{ s.t. } \absa{\lambda_k(t) - \mu_k(t)} > \frac{N^{\gamma}}{Nt}
			} < N^{-D}
		\]
	for any $D$ and $N$ sufficiently large. Choosing $t = \frac{N^\gamma}{N\ell(2-N^{-\e})\delta_N}$, $\gamma = \e/8$, and $\delta_N \geq N^{-\e/8}$ so that $t < 1$, we have
		\[
			\lim_{N \to \infty} \P\pa{ \absa{\left\{ k : \lambda_k(t) \in I \right\}} > 0} = 0,
		\]
	which establishes the proposition for $\delta_N \geq N^{-\e/8}$.
	Finally, if $I' \subset I$, then  
		\[
			\lim_{N \to \infty} \P\pa{ \absa{\left\{ k : \lambda_k(t) \in I' \right\}} > 0} 
			\leq
			\lim_{N \to \infty} \P\pa{ \absa{\left\{ k : \lambda_k(t) \in I \right\}} > 0} 
		\]
	establishes the proposition for any $\delta_N \to 0$.
\end{proof}

We now prove the following (almost) four-moment matching proposition.
\begin{proposition}
\label{prop:4_moment_matching}
	Let $H^v, H^w$ be two Wigner matrices with eigenvalues $\bl$ and $\bmu$ respectively. Assume that 
	for all $1 \leq i \leq j \leq N$, and $1 \leq k \leq 3$, $\E^v\pa{h_{ij}^k} = \E^w\pa{h_{ij}^k}$, and further that there
	exists $t = t(N)$ such that
		\begin{equation}
		\label{eq:4_moment_condition}
			\absa{\E^v\pa{h_{ij}^4} - \E^w\pa{h_{ij}^4}} \leq t \text{ for } i \leq j.
		\end{equation}
	Fix $\e > 0$ and let $\rho = \rho(E) = \delta_N \ell(E)$ for some $\delta_N \to 0$ as $N \to \infty$.  
	Let $f$ a smooth compactly supported function such that
	with $E \in [2-N^{-\e}, 2]$, $f = 1$ on $[E-\rho, E + \rho]$, $f = 0$ on $[E -2\rho, E + 2 \rho]$,
	and $\norm{f^{(k)}}_\infty < \rho^{-k}$, $k=1,2$.
	Then uniformly in $E$,
		\begin{equation}
		\label{eqn:4moment_to_prove}
				\lim_{N \to \infty} 
				\absa{
					\E\pa{ \Tr f\pa{H^v}} - \E\pa{\Tr \pa{H^w}}
				}
			= 0.
		\end{equation}
\end{proposition}
\begin{proof}
	Fix a bijective ordering map of the index set of the independent matrix entries, 
	$\phi: \{ (i,j) : 1 \leq i \leq j \leq N\} \to \llbracket 1, \gamma(N) \rrbracket$, with $\gamma(N) = N(N+1)/2$.
	Let $H_\gamma$ be the Wigner matrix whose matrix elements $h_{ij}$ follow the $v$-distribution for $\phi(i,j) \leq \gamma$,
	and the $w$-distribution otherwise, so that $H^v = H_0$ and $H^w = H_{\gamma(N)}$. Then it is sufficient to prove that
		\begin{equation}
		\label{eq:to_prove_4_moment}
			\absa{
				\E\pa{\Tr f(H_\gamma)} - \E\pa{\Tr f(H_{\gamma-1})}
			}
			\leq 
			\frac{\varphi^C}{N^2} \frac{t \ell(E) \kappa(z)^{1/2}}{\rho(E)} = \frac{\varphi^C\kappa(z)^{1/2}t}{N^2\delta_N}. 
		\end{equation}
	To do this, let $\chi$ be a fixed, smooth, symmetric function such that
	$\chi(x)=1$ if $\abs{x} < \ell(E)$, $\chi(x) = 0$ if $\abs{x} > 2 \ell(E)$, and $\norm{\chi'(x)}_\infty \leq 1/\ell(E)$. 
	By the Helffer-Sj\"ostrand formula \cite{HelfferSjostrand1989}, with $z = x + \ii y$, we have
		\[
			\sum_{k=1}^N f\pa{\lambda_k} = \int_\C g(z) \Tr\pa{ \frac{1}{H-z} } \, \rd z, 
			\text{ where }
			g(z) = \frac{1}{\pi} \pa{
				\ii y f''(x) \chi(y) + \ii \pa{
						f(x) + \ii y f'(x)
					}\chi'(y)
			}.
		\]
	Let 
		\[
			\Xi(H) = \re\pa{ \int_{\abs{y} > N^{-1}} g(z) \Tr\pa{ \frac{1}{H-z} } \, \rd z }
			= \int_{\abs{y} > N^{-1}} g(z) \Tr\im\pa{ \frac{1}{H-z} } \, \rd z,
		\]
	and note that for $\abs{y} < 1/N$, we have $y^2 \abs{\lambda - (x + \ii y)}^{-2} \leq N^{-2} \abs{\lambda - (x + \ii / N)}^{-2}$.
	It follows that
		\[
			\absa{\Tr f(H) - \Xi(H)}
			\leq
			\absa{
				\int_{ \substack{ \abs{y} < 1/N  \\ x \in [E-2\rho, E + 2\rho ]} } 
					\sum_{k=1}^N \frac{y^2 \absa{ f''(x) } }{ (x-\lambda_k)^2 + y^2 }
					\, \rd y \rd x
			}
			\leq
			\absa{
				\frac{1}{N^3 \rho^2} \int_{E-2\rho}^{E+2\rho}
					\sum_{k=1}^N \frac{1}{ (x-\lambda_k)^2 + N^{-2} }
					\, \rd x
			}.
		\]
	Since $\int_{\R} \abs{\lambda_i - (x+ \ii / N)}^{-2} \, \rd x \leq CN$, we have
		\[
			\int_{E-2\rho}^{E+2\rho} \frac{1}{\rho^2 N^3} \sum_{k:\gamma_k \in (E - \varphi^C \rho, E + \varphi^C \rho)} 
				\frac{ \rd x}{ \abs{\lambda_i - (x+ \ii / N)}^2} 
			= \OO\pa{
				\frac{\varphi^C}{(N\rho)^2}
			}.
		\]
	And by Corollary \ref{cor:rigidity}, with overwhelming probability, we have $\abs{\lambda_i - \gamma_i} \ll \abs{\gamma_i - E}$.
	Therefore,
		\[
			\int_{E-2\rho}^{E+2\rho} \frac{1}{\rho^2 N^3} \sum_{k:\gamma_k \notin (E - \varphi^C \rho, E + \varphi^C \rho)} 
				\frac{ \rd x}{ \abs{\lambda_i - (x+ \ii / N)}^2} 
			\leq \frac{1}{\rho^2 N^3} \sum_{k:\gamma_k \notin (E - \varphi^C \rho, E + \varphi^C \rho)} \frac{\varphi^C \rho}{ \abs{E - \gamma_i}^2   }
			= \OO\pa{ \frac{\varphi^C}{ (N\rho)^2  }  },
		\]
	with overwhelming probability. It follows that in order to prove the proposition, it is sufficient to prove
	\eqref{eqn:4moment_to_prove} with $\Xi(H)$ in place of $\Tr f(H)$. \\
	
	 Let $V = V_{ij} E_{ij} + V_{ji} E_{ji}$, $W = W_{ij} E_{ij} + W_{ji} E_{ji}$, and
		\[
			H_{\gamma-1} = Q + \frac{V}{\sqrt{N}}, \quad H_\gamma = Q + \frac{W}{\sqrt{N}},
		\]
	where $Q$ coincides with $H_{\gamma-1}$ and $H_\gamma$ except at the entries $(i, j)$ and $(j, i)$ where $Q_{ij} = 0$.
	Let
		\[
			R = \frac{1}{Q-z}, \quad S = \frac{1}{H_\gamma-z}, \quad \hat{R} = \frac{1}{N} \Tr R,
			\quad \hat{R}_v^{(m)} = \frac{(-1)^m}{N} \Tr(RV)^mR,
			\quad \Omega_v = - \frac{1}{N} \Tr(RV)^5 S,
		\]
	and note that by the resolvent expansion, we have
		\[
			\frac{1}{N} \Tr S = \hat{R} + \xi_v, \text{ where }
			\xi_v = \sum_{m=1}^4 N^{-m/2} \hat{R}_v^{(m)} + N^{-5/2} \Omega_v.
		\]
	Therefore,
		\begin{equation}
		\label{eq:total}
			\E\pa{ \Xi(H_\gamma) - \Xi(H_{\gamma-1}) } = \int g(z) \pa{ \sum_{m=1}^4
				N^{-\frac{m}{2} + 1} \pa{
					\hat{R}_v^{(m)} - \hat{R}_w^{(m)}
				}
				+ N^{-\frac{3}{2}} \pa{\Omega_v - \Omega_w}
			} \, \rd^2z. 
		\end{equation}
	By assumption, the first three moments associated to $v,w$ match, so terms corresponding to $m=1,2,3$ in \eqref{eq:total} contribute identically 0.
	To analyze the term $m=4$, we write
		\[
			\Tr(RV)^4R = 
			\sum_{k=1}^N \sum_{i_1, \dots, i_8}
				R_{k i_1} V_{i_1 i_2} R_{i_2 i_3} V_{i_3 i_4} R_{i_4 i_5} V_{i_5 i_6} R_{i_6 i_7} V_{i_7 i_8} R_{i_8 k},
		\]
	where the sums over $i_1, \dots, i_8$ are over the set $\{i, j\}$, the indices corresponding to the only non-zero entries of $V,W$.
	Writing the same expansion with $W$ in place of $V$, taking the difference and applying \eqref{eq:4_moment_condition}, we bound
	the term corresponding to $m=4$ in \eqref{eq:total} by	
		\[
			\int \abs{g(z)} N^{-\frac{4}{2}} t \sum_{k=1}^N\sum_{i_1, \dots, i_8} \absa{R_{ki_1}R_{i_2i_3}R_{i_4 i_5} R_{i_6 i_7} R_{i_8 k}}.
		\]
	We now apply \eqref{eq:isotropic_local_law} to $S(z)$ and $R(z)$, and distinguish two cases. First, by \eqref{eq:local_law_entrywise_GW}, 
	since $\abs{m(z)} = \OO(1)$, with overwhelming probability, we have
		\[
			\int \abs{g(z)} N^{-\frac{4}{2}} t \sum_{k \in \{i, j\}} \sum_{i_1, \dots, i_8} \absa{R_{ki_1}R_{i_2i_3}R_{i_4 i_5} R_{i_6 i_7} R_{i_8 k}}
			\leq N^{-\frac{4}{2}} t  \int \abs{g(z)} \frac{\varphi^C}{Ny} \, \rd^2 z 
			= \OO\pa{
				\frac{\varphi^C}{N^2} \frac{t}{N\rho}
			}.
		\]
	In the remaining case, using \eqref{eq:local_law_entrywise_GW} and $\im m(z) \leq \kappa(z)^{1/2}$ we have
		\[
			\sum_{k \notin \{i, j\}} \sum_{i_1, \dots, i_8} \absa{R_{ki_1}R_{i_2i_3}R_{i_4 i_5} R_{i_6 i_7} R_{i_8 k}} 
			\leq N \biggr( \max_i \abs{R_{ii}}^3 \biggr)  \biggr( \max_{i \neq j} \abs{R_{ij}}^2 \biggr)
			\leq \frac{\varphi^C}{y} \pa{ \kappa(z)^{1/2} + \frac{1}{Ny}},
		\]
	with overwhelming probability. Therefore,
		\[
			\int \abs{g(z)} N^{-\frac{4}{2}} t \sum_{k \notin \{i, j\}} \sum_{i_1, \dots, i_8} \absa{R_{ki_1}R_{i_2i_3}R_{i_4 i_5} R_{i_6 i_7} R_{i_8 k}}
			= \OO\pa{
				\frac{\varphi^C}{N^2} \frac{t}{N\rho}
			},
		\]
	with overwhelming probability, which establishes \eqref{eq:to_prove_4_moment}, and completes the proof.
\end{proof}
\begin{proof}[Proof of Theorem \ref{thm:wegner_estimate}]
Choose $t \in [N^{-\e/4}, 1)$
and let $\tilde{H}_t = e^{-t/2}\tilde{H}_0 + (1+e^{-t})^{1/2}U$, where $\tilde{H}_0$ is any initial Wigner matrix, and $U$ is an independent GOE/GUE matrix. For any fixed $\delta >0$ and $f$ as in Proposition \ref{prop:4_moment_matching}, we have
	\begin{align*}
		\lim_{N \to \infty} \delta \P \pa{ \absa{\left\{ k : \lambda_k \in I \right\}}  > \delta } 
		&\leq 
		\lim_{N \to \infty} \delta \P \pa{  \Tr f\pa{H} > \delta } \\
		&\leq 
		\lim_{N \to \infty} \E(\Tr f\pa{\tilde{H}_t}) + \absa{ \E(\Tr f \pa{H})  -  \E(\Tr f\pa{ \tilde{H}_t })    }.
		\end{align*}
By our choice of $t$, Proposition \ref{prop:dynamic_step} gives $\lim_{N \to \infty} \E(\Tr f\pa{\tilde{H}_t}) = 0$. 
Furthermore, by \cite[Lemma 3.4]{ErdYauYin2011}, we can choose $\tilde{H}_0$ such that $\tilde{H}_t$ and $H$ meet the moment 
matching requirements of Proposition \ref{prop:4_moment_matching}. Therefore,
Proposition \ref{prop:4_moment_matching} gives $\lim_{N\to\infty} \abs{ \E(\Tr f \pa{H})  -  \E(\Tr f (\tilde{H}_t  ) )   } = 0$, which concludes the proof.
\end{proof}

\section{Local Law}
\label{sec:local_law}

In the first part of this section, follow the three-step strategy \cite{ErdPecRmSchYau2010,ErdSchYau2011Uni} to prove Theorem \ref{thm:local_law}, 
an optimal local law, valid near the edge of the spectrum. The following lemma 
covers the dynamical step and
gives an optimal local law after running the Dyson's Brownian motion \eqref{eq:DBM} for a sufficiently long time.

\begin{lemma}
	\label{lem:coupling_stieltjes_transforms}
	Fix $\e > 0$ and let $\mathcal{D}_\e$ be as in \eqref{def:D_eps}.
	For any $t \geq N^{-\e/100}$, for any $p \geq 1$, and
	for any $z \in \mathcal{D}_\e$, writing $z = E + \ii \eta$, there exist constants $C_p$ depending on $p$ and $\delta>0$ such that 
		\begin{equation}
		\label{eq:local_law_after_dynamics}
			\E\pa{
				\absa{
					\frac{1}{N} \sum_{k=1}^N \frac{1}{\lambda_k(t) - z} - \frac{1}{N} \sum_{k=1}^N \frac{1}{\mu_k(t) - z}
				}^{2p}
			} \leq \frac{C_pN^{-2p\delta}}{(N\eta)^{2p}} .
		\end{equation}
\end{lemma}
\begin{proof}
	Suppose without loss of generality that $\kappa(E) = 2 - E$. 
	By Theorem \ref{thm:coupled_distance}, for any $\gamma > 0$, with overwhelming probability, we have
		\[
			\absa{
					\frac{1}{N} \sum_{k=1}^N \frac{1}{\lambda_k(t) - z} - \frac{1}{N} \sum_{k=1}^N \frac{1}{\mu_k(t) - z}
				}
			\leq
			\frac{N^\gamma}{N^2t} \sum_{k=1}^{N} \frac{1}{  \absa{\lambda_k(t) - z} \absa{\mu_k(t) - z}    }.
		\]
	Furthermore, on the event \eqref{eq:rigidity_set}, there exists $C > 0$ such that 
		\[
			\frac{N^\gamma}{N^2t} \sum_{k: \abs{\gamma_k-E} \geq \varphi \eta} \frac{1}{  \absa{\lambda_k(t) - z} \absa{\mu_k(t) - z}    }
			\leq
			\frac{CN^\gamma}{N^2t} \sum_{k: \abs{\gamma_k - E} \geq \varphi \eta} \frac{1}{  \absa{\gamma_k - z} ^2  }
			\leq
			\frac{CN^\gamma}{Nt} \int_{-2}^{E-\varphi \eta} \frac{\diff\varrho(x)}{\abs{x-z}^2},
		\]
	and
		\[
			\frac{N^\gamma}{Nt} \int_{-2}^{E-\varphi \eta} \frac{\diff\varrho(x)}{\abs{x-z}^2}
			\leq
			\frac{N^\gamma}{Nt} \int_{-2}^{E-\varphi\eta} \frac{\diff\varrho(x)}{\abs{x-E}^2}
			=
			\frac{N^\gamma}{Nt}\pa{
				\frac{\sqrt{2-x}}{E-x} + \frac{\tanh^{-1}\pa{  \sqrt{ \frac{2-x}{2-E} }  }}{\sqrt{2-E}}
			}\biggr|_{-2}^{E-\varphi\eta}
			=
			\OO\pa{
				\frac{N^{2\gamma- \frac{\e}{5}} }{N\eta t} 
			}.			
		\]
	For the remaining terms, we have
		\[
			\frac{N^\gamma}{N^2t} \sum_{k:\abs{\gamma_k-E} \leq \varphi \eta} \frac{1}{  \absa{\lambda_k(t) - z} \absa{\mu_k(t) - z}    }
			\leq
			\frac{N^{2\gamma} (N\eta) \kappa(E)^{1/2}}{ t(N\eta)^2}
			=
			\frac{N^{2\gamma} \kappa(E)^{1/2}}{ t N\eta }.
		\]
	Choosing $\gamma = \e/100$, say, gives
		\[
			\absa{
					\frac{1}{N} \sum_{k=1}^N \frac{1}{\lambda_k(t) - z} - \frac{1}{N} \sum_{k=1}^N \frac{1}{\mu_k(t) - z}
			}
			=
			\OO\pa{ \frac{N^{-\delta}}{N\eta} },
		\]
	with overwhelming probability. Now let $X = \absa{ \frac{1}{N} \sum_{k=1}^N \frac{1}{\lambda_k(t) - z} - \frac{1}{N} \sum_{k=1}^N \frac{1}{\mu_k(t) - z} }$,
	$\e_1 = \frac{N^{-\delta}}{N\eta}$, and note that $X < 100/\eta$ deterministically. Then for any $D > 0$ and $N \geq N_0(D)$, we have
			\[
				\E\pa{X^p} = p \int_0^{100/\eta} \lambda^{p-1} \P(X > \lambda) \rd \lambda
					\leq \int_0^{\e_1} p \lambda^{p-1} \rd \lambda + \int_{\e_1}^{100/\eta} p\lambda^{p-1} N^{-D} \rd \lambda
					\leq \e_1^p + N^{-D} \pa{\frac{100}{\eta}}^p.
			\]
	Choosing $D$ large enough, depending on $p$, we have \eqref{eq:local_law_after_dynamics} for $N \geq N_0(D) = N_0(p)$ and
	some $C_p > 0$.
	To cover $N \geq 1$, we note simply that $\abs{s(z)}$ and $\abs{m(z)}$ are both bounded deterministically and uniformly in $z \in \mathcal{D}_\e$.
	Therefore, we can increase $C_p$ to cover the finitely many $N \leq N_0(p)$ to conclude the proof.
\end{proof}

We now present a four moment matching argument following \cite{LanLopMar2020} which constitutes the last step in the three step strategy.
For any Wigner matrix $M$, define
	\begin{equation}
	\label{eq:def_T_p}
		T_p(M) = T_p(z,M) = (\im s(z) - \im m(z))^{p} + (\re s(z) - \re m(z))^p, \quad g(p) = \frac{(Cp)^{p/2}}{(N\eta)^p},
	\end{equation}
where $C>0$ is some universal constant. 

\begin{definition}
	For any $w \in [0,1]$, $H$ a Wigner matrix and indices $a,b \in \llbracket 1, N \rrbracket$, define
	\[
		\Theta_w^{(a,b)} H = \begin{cases}
			h_{ij} & (i,j) \notin \{(a,b),(b,a)\} \\
			w h_{ij} & (i,j) \in \{(a,b),(b,a)\}.
		\end{cases}
	\]
\end{definition}

We begin with the following Lemma which follows from the local law \eqref{eq:local_law_entrywise_GW}.

\begin{lemma}
\label{lem:sc_estimates}
	Let $H$ be a Wigner matrix and write $\partial_{ab} = \partial/\partial_{H_{ab}}$. 
	There exists an event $\mathcal{G}$ and constants $C, c> 0$ such that for 
	all $z = E + \ii \eta \in \mathcal{D}_\e$, 
	\begin{equation}
	\label{eq:derivative_bound_good_set}
		\sup_{a,b,c,d \in \llbracket 1, N \rrbracket} \sup_{j \in \llbracket 1, 5 \rrbracket} \sup_{w \in [0,1]}
			\1_{\mathcal{G}} \absa{
				\partial_{cd}^j  T_1\pa{z,\Theta_{w}^{(a,b)}H}
			} \leq \frac{C \varphi^C}{N\eta},
	\end{equation}
	and
	\begin{equation}
	\label{eq:derivatives_size_of_bad_set}
		\P\pa{ \mathcal{G}^c } \leq C \exp\pa{  -c(\log N)^{c\log\log N}  }.
	\end{equation}
	Furthermore, we have the deterministic bound 
	\begin{equation}
	\label{eq:derivative_bound_bad_set}
		\sup_{a,b,c,d \in \llbracket 1, N \rrbracket} \sup_{j \in \llbracket 1, 5 \rrbracket} \sup_{w \in [0,1]}
			\1_{\mathcal{G}} \absa{
				\partial_{cd}^j  T_1\pa{z,\Theta_{w}^{(a,b)}H}
			} \leq CN^C.
	\end{equation}
\end{lemma}
\begin{proof}
	Recall that $G(z) = (H-z)^{-1}$ satisfies (see for example \cite{BenKno2016})
	\begin{enumerate}
		\item (Ward identity) For any $i \in \llbracket 1, N \rrbracket$,
			\[
				\sum_{1 \leq j \leq N} \absa{G_{ij}}^2 = \frac{\im G_{ii}}{\eta},
			\]
		\item For $i,j,k,\ell \in \llbracket 1, N \rrbracket$,
			\[
				\partial_{kl}G_{ij} = -\pa{
					G_{ik}G_{\ell j} + G_{i\ell}G_{kj}
				}(1 + \delta_{kl})^{-1}.
			\]
	\end{enumerate}
	Therefore
	\[
		\partial_{ab} s(z) =  \frac{1}{N} \sum_{i=1}^N \partial_{ba} G_{ii} = -\frac{1}{N} \sum_{i=1}^N G_{ia}G_{bi},
	\]
	and applying the Ward identity and the local law, we have
	\[
		\absa{ \partial_{ab} s(z) } \leq \frac{1}{N}\sum_{i=1}^N \absa{G_{ia}G_{bi}}
		\leq \frac{C}{N} \sum_{i=1}^N \pa{ \absa{G_{ia}}^2 + \absa{G_{bi}}^2  }
		\leq \frac{C}{N\eta}\pa{
			\absa{G_{aa}} + \absa{G_{bb}}.
		}
	\]
	Similarly, for the higher derivatives, we have
	\[
		\absa{\partial_{ab}^j s(z)} \leq \frac{C}{N\eta} \pa{
			\absa{G_{aa}} + \absa{G_{bb}} + \absa{G}_{ab}
		}^j.
	\]
	On the set such that \eqref{eq:local_law_entrywise_GW} holds, this concludes the proof when $w=1$.
	For $w \in [0,1)$, one can use a resolvent expansion to show that the same result holds \cite[(4.54)]{LanLopMar2020}. 
\end{proof}

\begin{proposition}
	\label{prop:4_mom_matching_high_moments}
	Let $H^{v}, H^{w}$ be two Wigner matrices with eigenvalues $\bl$ and $\bmu$ respectively. Assume that 
	for all $1 \leq i \leq j \leq N$, and $1 \leq k \leq 3$, $\E^v\pa{h_{ij}^k} = \E^w\pa{h_{ij}^k}$, and further that for
	some fixed $\delta > 0$,
		\begin{equation}
		\label{eq:4_moment_condition_2}
			\absa{\E^v\pa{h_{ij}^4} - \E^w\pa{h_{ij}^4}} \leq N^{-2-\delta} \text{ for } i \leq j.
		\end{equation}
	For $z \in \mathcal{D}_\e$ and fixed $p \geq 1$, if $T_{2p}(z,W) \leq Cg(2p)$, then
	there exists $C'>0$ such that
	$T_{2p}(z,V) \leq C'g(2p)$.
\end{proposition}
\begin{proof}
	Fix a bijective ordering map of the index set of the independent matrix entries, 
	$\phi: \{ (i,j) : 1 \leq i \leq j \leq N\} \to \llbracket 1, \gamma(N) \rrbracket$, with $\gamma(N) = N(N+1)/2$.
	For $i \leq j$, define $H_\gamma$ by
	\[
		h_{ij}^\gamma = \begin{cases}
			h_{ij} = H^v_{ij} & \phi(i,j) \leq \gamma \\
			r_{ij} = H^w_{ij} & \phi(i,j) > \gamma,
		\end{cases}
	\]
	and note that $H^w = H_0$ and $H^v = H_{\gamma(N)}$.
	Writing,
		\[
			\E\pa{
				T_{2p}(z,V)
			}
			-
			\E\pa{
				T_{2p}(z,W)		
			}
			=
			\sum_{\gamma=0}^{\gamma(N)-1} 
			\E\pa{
				T_{2p}(z,H_{\gamma+1})
			}
			-
			\E\pa{
				T_{2p}(z,H_{\gamma})				
			},
		\]
	we see it is sufficient to prove that for any $\gamma$,
		\begin{equation}
		\label{eq:goal_4_mom_high_moments}
			\E\pa{
				T_{2p}(z,H_{\gamma+1})
			}
			-
			\E\pa{
				T_{2p}(z,H_{\gamma})				
			}
			\leq
			g(2p).
		\end{equation}
	Now fix $(i,j)$ so that in the following we write $\partial$ for $\partial_{ij}$. For any $m \geq 1$, Taylor
	expanding $T_m(H_\gamma)$ in the $(i,j)$ entry, we have
	\begin{multline}
	\label{eq:taylor_gamma}
		T_m\pa{H_\gamma} - T_m\pa{\Theta_0^{(i,j)}H_\gamma} =
		\partial T_m\pa{\Theta_0^{(i,j)}H_\gamma} h_{ij} + \frac{1}{2} \partial^2 T_m\pa{\Theta_0^{(i,j)}H_\gamma} h_{ij}^2
		+ \frac{1}{3!} \partial^3 T_m\pa{\Theta_0^{(i,j)}H_\gamma} h_{ij}^3 \\
		+ \frac{1}{4!} \partial^4 T_m\pa{\Theta_0^{(i,j)}H_\gamma} h_{ij}^4
		+ \frac{1}{5!} \partial^5 T_m\pa{\Theta_{w_1(\gamma)}^{(i,j)}H_\gamma} h_{ij}^5,
	\end{multline}
	where $w_1(\gamma) \in [0,1]$ is a random variable depending on $h_{ij}$. Similarly, we have 
	\begin{multline}
	\label{eq:taylor_gamma-1}
		T_m\pa{H_{\gamma-1}} - T_m\pa{\Theta_0^{(i,j)}H_\gamma} =
		\partial T_m\pa{\Theta_0^{(i,j)}H_\gamma} r_{ij} + \frac{1}{2} \partial^2 T_m\pa{\Theta_0^{(i,j)}H_\gamma} r_{ij}^2
		+ \frac{1}{3!} \partial^3 T_m\pa{\Theta_0^{(i,j)}H_\gamma} r_{ij}^3 \\
		+ \frac{1}{4!} \partial^4 T_m\pa{\Theta_0^{(i,j)}H_\gamma} r_{ij}^4
		+ \frac{1}{5!} \partial^5 T_m\pa{\Theta_{w_2(\gamma)}^{(i,j)}H_\gamma} r_{ij}^5,
	\end{multline}
	where $w_2(\gamma) \in [0,1]$ a random variable depending on $r_{ij}$. 
	Subtracting \eqref{eq:taylor_gamma-1} from \eqref{eq:taylor_gamma}, taking expectation, and using
	$\Theta_0^{(i,j)}H_\gamma$ is independent from $h_{ij}, r_{ij}$ and 
	$\E(h_{ij}^k) = \E(r_{ij}^k)$ for $k=1,2,3$, we have
	\begin{multline}
	\label{eq:4th_5th_order_terms}
		\E\left[ T_m\pa{H_\gamma}\right] - \E\left[ T_m\pa{H_{\gamma-1}}\right] 
		= \frac{1}{4!}\left[
			\partial^4 T_m\pa{\Theta_0^{(i,j)} H_\gamma  } h_{ij}^4
		\right]
		- \frac{1}{4!}\left[
			\partial^4 T_m\pa{\Theta_0^{(i,j)} H_\gamma  } r_{ij}^4 
		\right] \\
		+
		\frac{1}{5!}\left[
			\partial^5 T_m\pa{\Theta_{w_1(\gamma)}^{(i,j)} H_\gamma  } h_{ij}^5
		\right]
		- \frac{1}{5!}\left[
			\partial^4 T_m\pa{\Theta_{w_2(\gamma)}^{(i,j)} H_\gamma  } r_{ij}^5
		\right].
	\end{multline}
	We now proceed by induction. The induction hypothesis at step $m \in \N$ is that
	for all $0 \leq n \leq m$, $w \in [0,1]$ and $(a,b) \in \llbracket 1, N \rrbracket^2$,
	for some $C > 0$,
	\begin{equation}
	\label{eq:induction_hypothesis}
		\E\left[
			T_{2n}\pa{ \Theta_w^{(a,b)} H_\gamma   }
		\right]
		\leq
		3C g(2n).
	\end{equation}
	The induction hypothesis is trivial when $m=0$.
	Assume the induction hypothesis at step $m-1$. We will prove it for step $m$.
	Again by the independence of $h_{ij}$ and $r_{ij}$ from $\Theta_0^{(i,j)}H_\gamma$, we have
	\[
		\E\left[
			\partial^4 T_m\pa{\Theta_0^{(i,j)} H_\gamma  } h_{ij}^4
		\right]
		- \E\left[
			\partial^4 T_m\pa{\Theta_0^{(i,j)} H_\gamma  } r_{ij}^4 
		\right]
		=
		\E \left[
			\partial^4 T_m\pa{\Theta_0^{(i,j)} H_\gamma  } 
		\right] \E\left[
			h_{ij}^4 - r_{ij}^4
		\right],
	\]
	and we have $\E\left[ h_{ij}^4 - r_{ij}^4 \right] \leq t$ by assumption. For the remaining factor, we have
	\begin{multline*}
		\partial^4 T_{2m} =
		2m T_{2m-1}T^{(4)} + 3(2m)(2m-1)T_{2m-2}\pa{T^{(2)}}^2 + 2m(2m-1)(2m-2)(2m-3)T_{2m-4}\pa{T'}^4 \\
		+ 4(2m)(2m-1)T_{2m-2}T'T^{(3)} + 6(2m)(2m-1)(2m-2)T_{2m-3}\pa{T'}^2T^{(2)}.
	\end{multline*}
	To bound these terms, note that by \eqref{eq:local_law_original}, $\E(\abs{T^{2m-1}}) \leq \varphi \E(T^{2m})$, 
	and by \eqref{eq:derivative_bound_good_set}, we have $g(p) \leq \varphi^C g(p+1)$ for some $C>0$.
	This gives
	that there exists $C > 0$ such that
	\begin{equation}
	\label{eq:bound_4th_deriv_good_set}
		\E\pa{
		\absa{
				\1_{\mathcal{G}} \partial^4 T_{2m}\pa{
					\Theta_0^{(i,j)} H_\gamma
				}
		}
		}
		\leq
		C \varphi^{(1+C)100} g(2m),
	\end{equation}
	and by \eqref{eq:derivatives_size_of_bad_set} and \eqref{eq:derivative_bound_bad_set}, we have
	\begin{equation}
	\label{eq:bound_4th_deriv_bad_set}
		\E\pa{
			\absa{
				\1_{\mathcal{G}^c} \partial^4 T_{2m}\pa{
					\Theta_0^{(i,j)} H_\gamma
				}
			}
		}
		\leq
		CN^{-100}.
	\end{equation}
	In total, we have 
	\[
		\absa{
			\E\left[
			\partial^4 T_m\pa{\Theta_0^{(i,j)} H_\gamma  } h_{ij}^4
			\right]
			- \E\left[
				\partial^4 T_m\pa{\Theta_0^{(i,j)} H_\gamma  } r_{ij}^4 
			\right]
		}
		\leq
		C \varphi^{(1+C)100} N^{-2-\delta} g(2m),
	\]
	which for $N$ sufficiently large gives
	\[
		\absa{
			\E\left[
			\partial^4 T_m\pa{\Theta_0^{(i,j)} H_\gamma  } h_{ij}^4
			\right]
			- \E\left[
				\partial^4 T_m\pa{\Theta_0^{(i,j)} H_\gamma  } r_{ij}^4 
			\right]
		}
		\leq
		\frac{1}{2} N^{-2} g(2m).
	\]
	To control the fifth order terms in \eqref{eq:4th_5th_order_terms}, we use the (uniform) sub-exponential
	decay of $h_{ij}$ and $r_{ij}$. Namely, let $\mathcal{D} = \mathcal{D}(\delta)$ be the event such that 
	$\sup_{i,j} \abs{r_{ij}} + \abs{h_{ij}} \leq N^{-1/2 + \delta}$. Then there exist constants $D(\delta), d(\delta)$
	such that
	\[
		\P\pa{\mathcal{D}^c} \leq D(\delta)\exp\pa{
			-d(\delta) (\log N)^{d(\delta) \log \log N}
		}.
	\]
	It follows that
	\begin{align*}
		\absa{
			\E\pa{
				\partial^5 T_{2m}\pa{
					\Theta^{(i,j)}_{w_1(\gamma)} 
				} h_{ij}^5
			}
		}
		&\leq
		\absa{
			\E\pa{ \1_\mathcal{D}
				\partial^5 T_{2m}\pa{
					\Theta^{(i,j)}_{w_1(\gamma)} 
				} h_{ij}^5
			} 
		}
		+
		\absa{ 
			\E\pa{\1_{\mathcal{D}^c}
				\partial^5 T_{2m}\pa{
					\Theta^{(i,j)}_{w_1(\gamma)} 
				} h_{ij}^5
			}
		} \\
		& \leq
		CN^{-5/2+5\delta}\pa{
			\E\pa{
				\absa{
					\partial^5 T_{2m}\pa{
						\Theta_{w_1(\gamma)}^{(i,j)} H_\gamma
					}
				}
			} + 1
		}.
	\end{align*}
	The same bound holds with $r_{ij}$ in place of $h_{ij}$, and repeating the argument
	we gave to bound the fourth order terms gives that for $N$ sufficiently large,
	\begin{equation}
	\label{eq:5th_order_total}
		\absa{
			\E\pa{
				\partial^5 T_{2m}\pa{
					\Theta^{(i,j)}_{w_1(\gamma)} 
				} h_{ij}^5
			}
		}
		+
		\absa{
			\E\pa{
				\partial^5 T_{2m}\pa{
					\Theta^{(i,j)}_{w_1(\gamma)} 
				} r_{ij}^5
			}
		}
		\leq
		\frac{1}{2} N^{-2} g(2m).
	\end{equation}
	In summary, so far we have that conditional on the induction hypothesis \eqref{eq:induction_hypothesis} at
	level $m-1$, for $N$ large enough,
	\[
		\absa{
			\E\pa{T_{2m}\pa{H_\gamma}} - \E\pa{T_{2m}\pa{H_{\gamma-1}}}
		}	
		\leq
		N^{-2}g(2m).
	\]
	Therefore, for any $\gamma$ we have
	\[
		\absa{
			\E\pa{T_{2m}(W)} - \E\pa{T_{2m}\pa{H_\gamma}}
		}
		\leq
		g(2m),
	\]
	and since $\E\pa{T_{2m}(W)} \leq Cg(2m)$ by assumption, we have
	\[
		\E\pa{T_{2m}\pa{H_\gamma}} \leq 2Cg(2m).
	\]
	This verifies the induction hypothesis at level $m$ when $w=1$. To extend to other values of $w$, 
	expand again,
	\begin{multline*}
		T_{2m}\pa{H_\gamma} - T_{2m}\pa{\Theta_{w}^{(i,j)}H_\gamma}
		=
		\partial T_{2m}\pa{\Theta_{0}^{(i,j)}H_\gamma} h_{ij}
		+ \frac{1}{2} \partial^2 T_{2m}\pa{\Theta_{0}^{(i,j)}H_\gamma} h_{ij}^2
		+ \frac{1}{3!} \partial^3 T_{2m}\pa{\Theta_{0}^{(i,j)}H_\gamma} h_{ij}^3 \\
		+ \frac{1}{4!} \partial^4 T_{2m} T_{2m}\pa{\Theta_{0}^{(i,j)}H_\gamma} h_{ij}^4
		+ \frac{1}{5!} \partial^5 \pa{\Theta_{\tau(w)}^{(i,j)}H_\gamma} h_{ij}^5,
	\end{multline*}
	for some random variable $\tau(w) \in [0,1]$. Taking expectation, we can bound
	the terms up to order four by $g(2m)/2$, arguing as in \eqref{eq:bound_4th_deriv_good_set}
	and \eqref{eq:bound_4th_deriv_bad_set}. To control the last term, we can argue has in \eqref{eq:5th_order_total}
	to again get $g(2m)/2$ as a bound. In total we have
	\[
		\sup_{w \in [0,1]} \sup_{i,j \in \llbracket 1, N \rrbracket}
		\E\left[
			T_{2m}\pa{
				\Theta_w^{(i,j)} H_\gamma
			}
		\right]
		\leq
		3Cg(2m).
	\]
	This concludes the induction step, and therefore the proof.
\end{proof}

\begin{proof}[Proof of Theorem \ref{thm:local_law}]
	Let $H$ be a Wigner matrix with eigenvalues $\bl$.
	Choose $t=N^{-\e/100}$ and let $H_t = e^{-t/2}H_0 + (1+e^{-t})^{1/2}U$, where $H_0$ is an initial Wigner matrix, 
	and $U$ is an independent standard GOE matrix. 
	By \cite[Lemma 3.4]{ErdYauYin2011}, we can choose $H_0=\tilde{H}_0$ such that $H_t$ 
	and $H$ meet the moment matching requirements of Proposition \ref{prop:4_mom_matching_high_moments}. 
	Let $s_{G_t}(z) = N^{-1} \sum_{k=1}^N (\mu_k(t) - z)^{-1}$.
	By Lemma \ref{lem:coupling_stieltjes_transforms}, for some $\delta > 0$, we have,
		\[
			\E\pa{
				\absa{
					s_{H_t}(z) - s_{G_t}(z)
				}^{2p}
			} \leq \frac{C_pN^{-2p\delta}}{(N\eta)^{2p}},
		\]
	and since the eigenvalue distribution of $G_t$ is invariant under the flow \eqref{eq:DBM}, 
	Theorem 1.1 in \cite{BouModPai2021} gives $\E(\abs{s_{G_t}(z) - m(z)}^{2p}) \leq (Cp)^{p/2}(N\eta)^{-2p}$.
	Applying Proposition \ref{prop:4_mom_matching_high_moments} completes the proof for $H$. 
\end{proof}

In the remainder of this section, we prove Proposition \ref{prop:weak_local_law_bulk}. In the proof of Theorem \ref{thm:log_corr_field}, it will serve as a substitute for 
Theorem \ref{thm:local_law}.

\begin{proposition}
\label{prop:weak_local_law_bulk}
	Fix $\kappa > 0$ and let $\mathcal{D} = \{ z = E + \ii \eta \text{ s.t. } \abs{E} < 2-\kappa, (\log N)^{1/4}/N \leq \eta \leq e^{(\log N)^{1/4}}/N \}$.
	There exists $C > 0$ fixed such that with $z \in \mathcal{D}$ and $u = (\log \log N)^C$, 
	\begin{equation}
	\label{eq:local_law_prob_bulk}
		\lim_{N\to\infty}\P\pa{ \abs{s(z) - m(z)} \geq \frac{u}{N\eta} } = 0.
	\end{equation}
\end{proposition}

We begin with the relaxation step, and take as an input, the following theorem which describes eigenvalue gaps in the bulk at time $t$.
\begin{theorem}[Theorem 3.1 and Lemma 3.4 in \cite{Bou2020}]
\label{thm:bulk_gaps}
	Let
	\[
		\bar{\mathfrak{u}}_k(t) = \frac{1}{N \im m\pa{\gamma_k^t}} \sum_{j=1}^N \im \pa{\frac{1}{\gamma_j - \gamma_k^t}}
			\pa{
				\lambda_j(0) - \mu_j(0)
			},
	\]
	and choose $\alpha, \e > 0$ fixed and arbitrarily small. For any $D>0$, there exist $C, N_0$ such that for any $N\geq N_0$, $\varphi^C/N < t < 1$
	and $k \in \llbracket \alpha N, (1-\alpha)N \rrbracket$, we have
		\[
			\P\pa{
				\absa{
					\lambda_k(t) - \mu_k(t) - \bar{\mathfrak{u}}_k(t) > \frac{N^\e}{N^2t}
				}
			}
			\leq
			N^{-D}.
		\]
	Moreover, for any $(k,\ell) \in \llbracket \alpha N, (1-\alpha) N \rrbracket^2$, 
	\[
		\absa{\bar{\mathfrak{u}}_k(t) - \bar{\mathfrak{u}}_\ell(t)}
		\leq
		\frac{C\varphi \abs{k-\ell}}{N^2t},
	\]
	also with overwhelming probability.
\end{theorem}

We first show that we can assume that $\bar{\mathfrak{u}}_k(t) \leq (\log N)^{3/4}/N$, say. 
\begin{lemma}
\label{lem:u_k_bound}
	Let $t = N^{-1+\theta}$ for some $0 < \theta < 1$. Then for some $C > 0$,
	\begin{equation}
	\label{eq:u_k_variance_bound}
		\Var(\bar{\mathfrak{u}}_k(t)) \leq CN^{-2}(1-\theta) \log N.
	\end{equation}
	Consequently,
	\begin{equation}
	\label{eq:u_k_good_set}
		\lim_{N\to\infty}\P\pa{
			\bar{\mathfrak{u}}_k \leq N^{-1}(\log N)^{3/4} 
		} = 1.
	\end{equation}
\end{lemma}
\begin{proof}
	Let $H$ be a Wigner matrix with eigenvalues $\bl$, and $G$ a GOE matrix with eigenvalues $\bmu$. Then for any 
	$\alpha N < j < (1-\alpha)N$,
	by Taylor expanding, we have that
	\begin{multline*}
		\im\log\det(H-\gamma_j + \ii t) - \im \log\det(G-\gamma_j+\ii t) \\
		=
		\sum_{k=1}^N \im\pa{ \frac{1}{\lambda_k - \gamma_j^t} } \pa{\lambda_k - \mu_k}
			+ \OO\pa{
				 \sum_{k=1}^N \frac{\abs{\lambda_k-\gamma_j}}{(t^2 + (\lambda_k-\gamma_j)^2)^2} (\lambda_k-\mu_k)^2
			}.
	\end{multline*}
	On the set $\mathcal{A}$, we find
	\[
		\frac{1}{N} \sum_{k=1}^N \im\pa{ \frac{1}{\lambda_k - \gamma_j^t} } \pa{\lambda_k - \mu_k}
		= \frac{1}{N} (\im\log(H-\gamma_j + \ii t) - \im \log(G-\gamma_j+\ii t) ) + \OO\pa{ \frac{\varphi^{10}}{(Nt)^2}  },
	\]
	and furthermore that
	\[
		\bar{\mathfrak{u}}_k(t)
		= \frac{1}{N} (\im\log\det(H-\gamma_j + \ii t) - \im \log\det(G-\gamma_j+\ii t) ) + \OO\pa{ \frac{\varphi^{10}}{(Nt)^2}  }.
	\]
	To conclude \eqref{eq:u_k_variance_bound}, we use Proposition \ref{prop:var_GW} and \cite[Lemma 4.9]{BouModPai2021}
	to bound $\Var(\im\log\det(H-\gamma_j + \ii t) )$
	and $\Var( \im \log\det(G-\gamma_j+\ii t) )$, respectively. Finally, \eqref{eq:u_k_good_set} follows from Chebyshev's inequality.
\end{proof}

\begin{lemma}
\label{lem:convert_to_uniform}
Let $t = N^{-1+3\alpha}$ for some $0 < \alpha < 1$. Then there exists $\bar{\mathfrak{u}}$ such that
	\begin{equation}
	\label{eq:s_t_u_t}
		\lim_{N\to\infty}\P\pa{
			\absa{s_t(z) -\frac{1}{N} \sum_{k=1}^N \frac{1}{\mu_k(t)-z + \bar{\mathfrak{u}}} } \geq \frac{C}{N^\alpha}
		}
		= 0,
	\end{equation}
	and $\P(\bar{\mathfrak{u}} \geq N^{-1}(\log N)^{3/4}) \to 0$ as $N\to\infty$.
\end{lemma}
\begin{proof}
	Let $I = [E - N^{-1+\alpha}, E + N^{-1+\alpha}]$. We have
	\[	
		s_t(z) -\frac{1}{N} \sum_{k=1}^N \frac{1}{\gamma_k - z}
		= \pa{
		\frac{1}{N} \sum_{\lambda_k(t) \in I} \frac{1}{\lambda_k(t) - z} 
		-\frac{1}{N} \sum_{\gamma_k \in I} \frac{1}{\gamma_k - z}
		}
		+ 
		\pa{
		\frac{1}{N} \sum_{\lambda_k(t) \in I^c} \frac{1}{\lambda_k(t) - z}
		 -\frac{1}{N} \sum_{\gamma_k \in I^c} \frac{1}{\gamma_k - z}
		 }.
	\]	
	On the event $\mathcal{A}$, we have
	\begin{equation}
	\label{eq:rigidity_on_complement}
		\absa{
			\frac{1}{N} \sum_{\lambda_k(t) \in I^c} \frac{1}{\lambda_k(t) - z} - \frac{1}{N} \sum_{\gamma_k \in I^c} \frac{1}{\gamma_k - z} 
		}
		\leq
		\frac{1}{N} \sum_{\lambda_k(t) \in I^c} \frac{ \absa{\lambda_k(t) - \gamma_k}  }{ \absa{\gamma_k-z}^2 }
		= 
		\OO\pa{N^{-\alpha}},
	\end{equation}
	and by Theorem \ref{thm:bulk_gaps} and Lemma \ref{lem:u_k_bound}, 
	our choice of $t$ allows us to choose $\bar{\mathfrak{u}} = \mathfrak{u}_k$ 
	with $k$ such that $\lambda_k(t) \in I$, so that
	\[
		\frac{1}{N} \sum_{\lambda_k(t) \in I} \frac{1}{\lambda_k(t) - z} 
		-\frac{1}{N} \sum_{\gamma_k \in I} \frac{1}{\gamma_k - z}
		=
		\frac{1}{N} \sum_{\lambda_k(t) \in I} \frac{1}{\mu_k(t) - z + \bar{\mathfrak{u}}} 
		-\frac{1}{N} \sum_{\gamma_k \in I} \frac{1}{\gamma_k - z}
		+ \OO\pa{
			N^{-\alpha}
		}.
	\]
	To conclude the proof, re-write the first term on right hand side of the above as
	\[
		\frac{1}{N} \sum_{k=1}^N \frac{1}{\mu_k(t) - z + \bar{\mathfrak{u}}} 
		-\frac{1}{N} \sum_{k=1}^N \frac{1}{\gamma_k - z}
		+
		\pa{
		\frac{1}{N} \sum_{\gamma_k \in I^c} \frac{1}{\gamma_k - z}
		- \frac{1}{N} \sum_{\lambda_k(t) \in I^c} \frac{1}{\mu_k(t) - z + \bar{\mathfrak{u}}} 
		},
	\]
	and bound the term in parentheses as in \eqref{eq:rigidity_on_complement}.
\end{proof}

\begin{lemma}
\label{lem:est_at_t}
	Fix $\kappa > 0$ and let $\mathcal{D} = \{ z = E + \ii \eta \text{ s.t. } \abs{E} < 2-\kappa, (\log N)^{1/4}/N \leq \eta \leq e^{(\log N)^{1/4}}/N \}$.
	There exists $C > 0$ fixed such that with $u = (\log \log N)^C$ and $t = N^{-1+\alpha}$ with $0 < \alpha < 1$, for $z \in \mathcal{D}$,
	we have
	\begin{equation}
	\label{eq:local_law_prob_bulk_t}
		\lim_{N\to\infty}\P\pa{\abs{s_t(z) - m(z)} \geq \frac{u}{N\eta} } = 0.
	\end{equation}
\end{lemma}

\begin{proof}
Let 
	\[
		\mathcal{G} = \left\{ 
			\absa{s_t(z) -\frac{1}{N} \sum_{k=1}^N \frac{1}{\mu_k(t)-z + \mathfrak{u}} } \leq \frac{C}{N^\alpha}
			\text{ and }
			\mathfrak{u} \leq \frac{(\log N)^{3/4}}{N}
		\right\},
	\]
and recall that by Lemma \ref{lem:convert_to_uniform}, $\lim_{N\to\infty} \P(\mathcal{G}) = 1$. 
Furthermore, we have,
	\[
		\P\pa{
			\left\{
				\absa{s_t(z)-m(z)} \geq \frac{u}{N\eta}
			\right\} \cap \mathcal{G}
		}
		\leq
		\P\pa{
			\left\{
				\absa{
					\frac{1}{N} \sum_{k=1}^N \frac{1}{\mu_k(t)-z + \mathfrak{u} } - m(z) 
				}
				\geq \frac{u}{2N\eta} 
			\right\} \cap \mathcal{G}
		}.
	\]
For a given $z = E + \ii \eta \in \mathcal{D}$, 
let $z_{\ell, k} = E_\ell + \ii \eta_k$, $1\leq k, \ell \leq (\log N)^{10}$ where $E_1 = E - (\log N)^{3/4}/N$, 
$E_{\ell+1}-E_\ell = u\eta/(\log N)^{10}$, and define $\eta_k$ by $\eta_1 = (\log N)^{1/4}/N$,
	\[
		\eta_{k+1} = \eta_k + \frac{u\eta_k}{(\log N)^{10}}.
	\]
Let
	\[
		\mathcal{B} = 
		\bigcup_{k,\ell}
		\left\{
			\absa{
				\frac{1}{N} \sum_{j=1}^N \frac{1}{\mu_k(t) - z_{k, \ell}}- m(z_{k,\ell})
			} \geq \frac{u}{N\eta}
		\right\},
	\]
and note that by \cite[(3.18)]{BouModPai2021}, for some universal constants $c,C > 0$ and $\beta = 1,2$, we have
	\[
		\P_{\rm G\beta E}\pa{
			\abs{s(z) - m(z)} \geq \frac{u}{N\eta}
		}
		\leq
		C^{-cu^2}.
	\]
It follows that for $z \in \mathcal{D}$,
	\[
		\lim_{N \to \infty}\P\pa{
			\mathcal{B}
		}
		\leq
		\lim_{N\to\infty}
		(\log N)^{200} \,\P\pa{
			\absa{
				\frac{1}{N} \sum_{j=1}^N \frac{1}{\mu_k(t) - z}- m(z)
			} \geq \frac{u}{N\eta}
		} 
		= 0.
	\]
To conclude the proof, we show that
	\[
		\P(\mathcal{B}^c)
		\leq
		\P\pa{
			\left\{
				\absa{
					\frac{1}{N} \sum_{k=1}^N \frac{1}{\mu_k(t)-z + \mathfrak{u} } - m(z) 
				}
				\leq \frac{u}{2N\eta} 
			\right\} \cap \mathcal{G}
		}.
	\]
To do so, we note that by \cite[Lemma 3.8]{BouModPai2021}, for any $z \in \mathcal{D}$, we have
	\begin{equation}
	\label{eq:s_lip}
		\abs{s'(z)} \leq \frac{ (\log N)^{5} }{N\eta^2}
	\end{equation}
with overwhelming probability.
For $w \in \mathcal{D}$, choose $z_{k,\ell}$ such that $\abs{w-z_{k, \ell}} \leq u\eta/(\log N)^{10}$.
Then
	\[
		\abs{
			s(w) - m(w)
		}	
		\leq
		\abs{
			s(w) - s(z_{k,\ell})
		}
		+ 
		\abs{
			s(z_{k,\ell}) - m(z_{k,\ell})
		}
		+
		\abs{
			m(z_{k,\ell}) - m(z)
		}.
	\]	
By \eqref{eq:s_lip}, we have $\abs{s(w) - s(z_{k,\ell})	} = \oo\pa{\frac{u}{N\eta}}$ with overwhelming probability,
and since $\abs{m'(z)} \leq 10$ for $z \in \mathcal{D}$, we also have $\abs{m(w) - m(z_{k,\ell})	} = \oo\pa{\frac{u}{N\eta}}$.
Since $\P(\mathcal(B)^c) \to 1$ as $N \to \infty$, this concludes the proof.
\end{proof}

Recall that to prove the local law at the edge, we used the three step strategy. In the bulk, we instead use the continuity method. More precisely, we rely on 
\cite[Lemma A.1]{BouYau2017} which we now quote.

\begin{lemma}
\label{lem:continuity_est}
	Denote $\partial_{ij} = \partial_{h_{ij}}$. Suppose that $F$ is a smooth function of the matrix elements
	$(h_{ij})_{i \leq j}$ satisfying
	\begin{equation}
	\label{eq:def_M}
		\sup_{0 \leq s \leq t, i \leq j, \theta} \E\pa{
			\pa{
				N^{3/2} \absa{h_{ij}(s)}^3 + N^{1/2} \abs{h_{ij}(s)} 
			}
			\partial_{ij}^3 F(\theta H_s)
		} 
		\leq 
		M,
	\end{equation}
	where $(\theta H)_{ij} = \theta_{ij}h_{ij}$, $\theta_{k\ell} =1$ unless $\{k,\ell\} = \{i,j\}$ and
	$0 \leq \theta_{ij} \leq 1$. Then 
	\[
		\absa{
			\E F(H_t) - \E F(H_0)
		}
		=
		\OO\pa{
			t N^{1/2} M
		}.
	\]
\end{lemma}

We now apply Lemma \ref{lem:continuity_est} to
	\[
		F(H_t) = \abs{s_t(z) - s_0(z)}^2,
	\]
so that we can control $\E F(H_t) - \E F(H_0) = \E \abs{s_t(z) - s_0(z)}^2$.

\begin{lemma}
\label{lem:cont_est_my_F}
	Let $F(H_t) = \abs{s_t(z) - s_0(z)}^2$ and $M$ be as in \eqref{eq:def_M}.
	Then for any $z$ in
	$\mathcal{D} = \{ z = E + \ii \eta \text{ s.t. } \abs{E} < 2-\kappa, (\log N)^{1/4}/N \leq \eta \leq e^{(\log N)^{1/4}}/N \}$ and
	$\e > 0$, $M \leq N^{\e}$ with overwhelming probability.
\end{lemma}
\begin{proof}
First, fix a time $u$ and let $\partial_{ij}$ denote $\partial_{h_{ij}(u)}$.
Then \eqref{eq:resolvent_derivatives} to compute $\partial_{ij}^3 \pa{ \pa{s_u(z) - s_0(z)} \overline{\pa{s_u(z) - s_0(z)}}}$ and \eqref{eq:local_law_entrywise_GW} to bound terms give
that for any $\e > 0$,
	\[
		\P\pa{
			\partial_{ij}^3 \pa{
			\abs{s_u(z) - s_0(z)}^2
			}
			\geq
			N^{\e}
		}
		\leq
		N^{-D}.
	\]
To take the supremum over $\theta$, we can first discretize the interval $[0,1]$, and apply the above argument so that the above estimate holds on this entire grid.
We can then extend to all $\theta \in [0,1]$ since we can bound $\partial_{ij}^4 F(\theta H_u)$ uniformly in $\theta$ using the deterministic estimate $\abs{\partial_{k\ell}^rG_{ij}(z)} \leq \eta^{-r-1}$.
Similarly, to take $\sup_{0 \leq s \leq t}$, we can discretize the interval $[0,t]$ and apply the previous arguments on this grid in time.
Then
	\[
		\abs{
			s_{t_1}(z) - s_{t_2}(z)
		} 
		= 
		\absa{
			\frac{1}{N} \sum_{k=1}^N \frac{1}{\lambda_k(t_1)-z} - \frac{1}{N} \sum_{k=1}^N \frac{1}{\lambda_k(t_2)-z}
		}  
		\leq \frac{N^\e}{N\eta^2} \pa{
			 \sup_{1\leq k \leq N} \abs{ \lambda_k(t_1) - \lambda_k(t_2)}
		},
	\]
and by Weyl's inequality, we can make the second term as small as we want by choosing the mesh size of the small enough.
Finally, we can take make the previous steps uniform in $i\leq j$ since there are only order $N^2$ matrix entries, 
and considering the subexponential decay of the matrix entries $h_ij$, we 
finally find
	\[
		\P(M \geq N^\e) \leq N^{-D} 
	\]
for any $D > 0$, as desired.
\end{proof}

We close this subsection with the proof of Proposition \ref{prop:weak_local_law_bulk}.

\begin{proof}[Proof of Proposition \ref{prop:weak_local_law_bulk}]
Lemma \ref{lem:est_at_t} proves the estimate at time $t = N^{-1+\alpha}$. And by
Lemmas \ref{lem:continuity_est} and \ref{lem:cont_est_my_F}, with $t=N^{-1+\alpha}$, we have
	\[
		\E \absa{s_t(z) - s_0(z)}^2 = \OO\pa{ N^{-1/2 + 2\alpha} },
	\]
which, by Markov's inequality, concludes the proof.
\end{proof}

\section{Smoothing}

Theorem \ref{thm:wegner_estimate}, and Proposition \ref{prop:weak_local_law_bulk} and Theorem \ref{thm:local_law} imply Proposition \ref{prop:smoothing} which says
that in order to prove Theorem \ref{thm:log_corr_field}, we can study $L_N(\bl, z)$ in place of $L_N(\bl, E)$, with appropriate $z$ in the upper-half plane.
The statement and proof are entirely analogous to \cite[Proposition 2.1]{BouMod2019} which holds for $E$ strictly in the bulk. 
Below we follow the proof of \cite[Proposition 3.10]{BouModPai2021}, which uses the same argument in the
context of $\beta$-ensembles.

\begin{proposition} 
\label{prop:smoothing}
	Let $\eta(E) = e^{(\log N)^{1/4}} \ell(E)$, and let $z = E + \ii \eta(E)$.
	For any fixed $c,\e > 0$, uniformly in $E$ in $\mathcal{G}_{\e,c}$ as in \eqref{def:G_e,c}, 
		\[
			(\log N)^{-1/2} (L_N(z) - L_N(E)) \to 0 \text{ in probability.}
		\]
\end{proposition}

\begin{proof}
Fix $E \in \mathcal{G}_{\e, c}$, and for brevity, let $\eta = \eta(E)$, $z = E + \ii \eta$.
Writing $\log(z - \lambda) - \log(E - \lambda) = \int_0^\eta \frac{\ii \diff u}{E+\ii u - \lambda}$, we have
	\[
		L_N(z) - L_N(E) = \ii N \int_0^{\eta} (m (E+\ii u) - s(E+\ii u)) \diff u.
	\]
Introduce $\eta' = (\log N)^{-1/4} \ell(E)$ and $z' = E + \ii \eta'$.
By Theorem \ref{thm:local_law} and Proposition \ref{prop:weak_local_law_bulk}, there exists an event $G$ with $\P(G) \to 1$ as $N\to\infty$ on which we have
	\begin{align}
	\label{eq:from_eta'_to_eta}
    		\E \pa{ \1_{G} \absa{ N \int_{\eta'}^{\eta} (m(E+\ii u) - s(E+\ii u)) \diff u } }
		\leq
		\int_{\eta'}^{\eta} \frac{C}{u} \diff u
		\leq 
		C(\log N)^{1/4}.
	\end{align}
Furthermore, by the triangle inequality,
	\begin{multline}
	\label{eq:decompo}
    		\absa{ N \int_0^{\eta'} (m (E+\ii u) - s_N(E+\ii u)) \diff u } \\
    		 \leq N \eta' \absa{m (z') - s(z')}
    		+ N \int_0^{\eta'} \absa{m (E+\ii u) - m(z')} \diff u
   		 + N \int_0^{\eta'} \absa{s (E+\ii u) - s(z')} \diff u. 
	\end{multline}
Taking expectation of the right hand side of \eqref{eq:decompo}, by Theorem \ref{thm:local_law} the first term is bounded.
The second term is bounded too, as follows from the deterministic estimate $\abs{m'(w)} \leq C/\sqrt{\kappa(w)}$, valid for $w$
in a compact set of $\CC$. For the third term, we introduce the event $A = \{ \cN([E-\eta',E+\eta']) =0 \}$, 
on which there are no particles at distance less than $\eta'$ from $E$.
By Proposition \ref{thm:wegner_estimate}, we have $\P(A) \to 1$ uniformly in $E$, and on the event $A$, we have
	\begin{align*}
    		\absa{s (E+\ii u) - s(z')} 
		= 
		\absa{\frac{1}{N} \sum_{k=1}^N \frac{E+\ii u-z'}{(\lambda_k-E-\ii u)(\lambda_k-z')}} 
		\leq
		\frac{1}{N} \sum_{k=1}^N \frac{\eta' \sqrt{2}}{\abs{\lambda_k-z'}^2}
    		= 
		\sqrt{2} \im s(z').
\end{align*}
Moreover, it follows from Theorem \ref{thm:local_law} that $\E[\im s_N(z')] \leq C( (N\eta')^{-1} + \im m(z') ) \leq C (N\eta')^{-1}$.
Therefore, we have 
	\begin{align*}
    		\E \pa{ \1_A 
    				N \int_0^{\eta'} \absa{s_N (E+\ii u) - s_N(z')} \diff u 
 		}
		\leq 
		\sqrt{2} N \eta' \E \left[\im s_N(z')\right]
		\leq 
		C.
\end{align*}
Combined with \eqref{eq:from_eta'_to_eta}, this concludes the proof.
\end{proof}

\section{Coupling of Determinants}

Our goal in this section is to prove Proposition \ref{prop:coupling}, which extends \cite[Proposition 3.2]{BouMod2019} to include $z$ at the edge.
Our argument comes directly from \cite[Section 2]{Bou2020}.
Proposition \ref{prop:flow_est} contains the essential estimate, 
which extends \cite[Proposition 2.11]{Bou2020} to include $z$ at the edge.  \\

First, applying It\^o's formula, we can compute the dynamics for the observable $f_t(z)$, see \cite[Lemma 2.1]{Bou2020}.

\begin{lemma}
\label{lem:ito_ft}
	Let $f_t(z)$ be as in \eqref{eq:def_uk_f}. Then, for any $z$ such that $\im(z) \neq 0$, 
		\begin{equation}
		\label{eq:ft_dynamics}
			\rd f_t = 
			\pa{\frac{z}{2} + s(z)} \pa{\partial_z f_t} \diff t 
			+ \frac{1}{2N}\pa{\frac{2}{\beta} - 1} \partial_{zz} f \diff t
				- e^{-\frac{t}{2}} \sqrt{ \frac{2}{N\beta}}  \sum_{k=1}^N \frac{\mathfrak{u}_k}{(x_k-z)^2}  \rd B_k .
		\end{equation}
\end{lemma}

As in the proof of Proposition 2.11 in \cite{Bou2020}, the proof of Proposition \ref{prop:flow_est} essentially says
that $f_t(z)$ approximately satisfies an advection equation,
	\begin{equation}
	\label{eq:advection_eq}
		\partial_t r(z,t) = \pa{ m(r(z,t)) + \frac{r(z,t)}{2} } \partial_z r(z,t).
	\end{equation}
We will use that the characteristics of this advection equation are explicit, and are given by
	\begin{equation}
	\label{eq:characteristics}
		z_t = \frac{
			e^{t/2}(z+ \sqrt{z^2-4}) + e^{-t/2}(z+\sqrt{z^2-4})
		}{2},
	\end{equation}
see \cite[(1.13)]{Bou2020}.
We quote \cite[Lemma 2.2]{Bou2020}, which describes the behavior of $z_t$ near the edge.
For the statement, recall the definition of $\kappa(z)$ from \eqref{def:l(E)_and_kappa(E)}, introduce
	\[
		a(z) = \text{dist}(z, [-2,2]), \quad b(z) = \text{dist}(z, \R \backslash[-2,2]),
	\]
and define the domains 
	\begin{align}
	\label{eq:def_domains}
		\mathcal{S}(\varphi) = \left\{  z = E + \ii y : E \in [-2,2],\, y = \varphi^2 \ell(E) \right\},
		\quad
		\mathcal{R}(\varphi) = \bigcup_{0 < t < 1} \left\{ z_t \,:\, z \in \mathcal{S}(\varphi) \right\}.
	\end{align}

\begin{lemma}
	\label{lem:z_t-z}
	For $z \in \mathcal{S}(\varphi)$, let $z_t$ solve
	\begin{equation}
	\label{eq:characteristics}
		\partial_t z_t = m(z_t) + \frac{z_t}{2}, \quad z_0 = z.
	\end{equation}
	Then, uniformly in $0 < t < 1$ and $z = E + \ii \eta$ such that $\eta > 0$, $\abs{z-2} \leq 1/10$, we have
		\begin{align*}
			\re\pa{z_t - z} \sim t \frac{a(z)}{\kappa(z)^{1/2}} + t^2,
			\quad
			\im\pa{z_t - z} \sim \frac{b(z)}{\kappa(z)^{1/2}} t.
		\end{align*}
	In particular, if $z \in \mathcal{S}(\varphi)$, then $z_t - z \sim \pa{\frac{\varphi^2}{N \kappa(z)^{1/2}} t + t^2} 
		+ \ii \kappa(z)^{1/2} t$.
\end{lemma}

In order to prove Proposition \ref{prop:flow_est}, we first extend \cite[Proposition 2.6]{Bou2020} to include all $z \in \mathcal{S}(\varphi)$,
namely $z$ at the edge. Note that \cite[Proposition 2.6]{Bou2020} is also an important step in the proof of
Theorem \ref{thm:coupled_distance}. 
We first quote \cite[Lemma 2.4]{Bou2020} which provides the initial estimate for the proof of Proposition \ref{prop:im_f}. 

\begin{lemma}
\label{lem:initial_est}
	On the event $\mathcal{A}$, see \eqref{eq:rigidity_set}, for any $z = E + \ii \eta \in \mathcal{R}(\varphi)$ in the trapezoid
	$\eta > \max(E-2, -E-2)$, we have $\im \tilde{f}_0(z) \leq C \varphi^{1/2}$. Outside the same trapezoid, 
	we have $\im \tilde{f}_0(z) \leq C\varphi^{1/2}\eta/\kappa(z)$.
\end{lemma}

In the proof of Proposition \ref{prop:im_f}, we will give precise meaning to the statement that $f_t$ approximately satisfies
the advection equation \eqref{eq:advection_eq}, following the argument in \cite[Proposition 2.6]{Bou2020}.

\begin{proposition}
\label{prop:im_f}
	For any $D > 0$, there exists $N_0(D)$ such that for any $N \geq N_0$,
		\[
			\P\pa{
				\im \tilde{f}_t(z) \leq \frac{\varphi}{N \im(z) \sqrt{\kappa(z)} },
				\, 0 < t < 1,\, z \in \mathcal{S}(\varphi) 
			} > 1 - N^{-D}.			
		\]
\end{proposition}

\begin{proof}
For any $1 \leq \ell, m \leq N^{10}$, define $t_\ell = \ell N^{-10}$ and $z^{(m)} = E_m + \ii \eta_m$, where
$\int_{-\infty}^{E_m} \rd \rho = (m-1/2)N^{-10}$ and $\eta_m = \varphi^2 \ell(E_m)$. Define the stopping times
with respect to $\mathcal{F}_t = \sigma\pa{B_k(s), 0 \leq s \leq t, 1 \leq k \leq N}$, 
	\begin{align*}
		\tau_{\ell, m} &= \inf\left\{
			0 \leq s \leq t_{\ell} : \im \tilde{f}_s\pa{z_{t_\ell-s}^{(m)}} > \frac{\varphi}{2} 
				\frac{ \kappa\pa{z^{(m)}}^{1/2} }{ \max(\kappa\pa{z^{(m)}}^{1/2}, t_\ell) } 
			\right\} \\
		\tau_0 &= \inf\left\{
			0 \leq t \leq 1 : \exists k \in \llbracket 1, N \rrbracket \text{ s.t. } \abs{x_k(t) - \gamma_k} > \varphi^{1/2} N^{-\frac{2}{3}} (\hat{k})^{-\frac{1}{3}}
			\right\} \\
		\tau &= \min\left\{
			\tau_0, \tau_{\ell, m} : 0 \leq \ell, m \leq N^{10}
			\right\},
	\end{align*}
with the convention $\inf \emptyset = 1$. We first claim it is sufficient to prove that for any $D>0$, there exists $N_1(D)$ such that for any
$N \geq N_1(D)$, we have
		\begin{equation}
		\label{eq:tau_1_tilde_f}
			\P\pa{ \tau = 1 } > 1 - N^{-D}.
		\end{equation}
Indeed, let
	\[
		A_{\ell, m, k} = \left\{
			\sup_{t_\ell \leq u \leq t_{\ell + 1}} \absa{
				\int_{t_\ell}^u \frac{e^{- \frac{s}{2}} \mathfrak{v}_k(s) }{ (z^{(m)}-x_k(s))^2} \rd B_k(s)
			} < N^{-3}
		\right\},
	\]
and for any given $z$, $t$, choose $t_\ell$, $z^{(m)}$ such that $t_\ell \leq t \leq t_{\ell +1}$, $\abs{z-z^{(m)}} < N^{-5}$.
Then using $\abs{\mathfrak{v}_k} < 1$, we have $\abs{\tilde{f}_t(z) - \tilde{f}_t(z^{(m)})} < N^{-2}$, say. Furthermore,
we have the trivial bounds $\abs{s_t(E+\ii \eta)} \leq \eta^{-1}$, $\abs{\partial_z f_t(E+\ii \eta)} \leq N\eta^{-2}$,
$\abs{\partial_{zz} f_t(E+\ii \eta)} \leq N \eta^{-3}$.
Therefore, on the event $\cap_k A_{\ell, m, k}$, by \eqref{lem:ito_ft}, we have $\abs{\tilde{f}_t(z^{(m)}) - f_{t_\ell}(z^{(m)})} < N^{-2}$ as well. 
It follows that
		\[
			\{ \tau = 1 \} \bigcap_{\substack{1 \leq \ell, m \leq N^{10} \\ 1 \leq k \leq N}} A_{\ell, m, k}
			\subset
			\bigcap_{z \in \mathcal{S}(\varphi), 0 < t < 1} \left\{
				\im \tilde{f}_t(z) \leq \varphi \frac{\kappa\pa{z^{(m)}}^{1/2}}{  \max\pa{ \kappa(z^{(m)}}^{1/2}, t )    }
			\right\}
		\]
Moreover by \cite[Appendix B.6 (18)]{ShoGalWel2009}, for any continuous martingale $M$ and any $\lambda, \mu > 0$, we have
		\begin{equation}
		\label{eq:martingale_est}
			\P\pa{
				\sup_{0 \leq u \leq t} \absa{M_u} \geq \lambda, \, \left<M\right>_t \leq \mu
			}
			\leq
			2e^{- \frac{\lambda^2}{2\mu}}.
		\end{equation}
For the martingale $M_u = \int_{t_\ell}^u \frac{e^{- \frac{s}{2}} \mathfrak{v}_k(s) }{ (z^{(m)}-x_k(s))^2} \rd B_k(s)$, we have
the deterministic estimate $\left< M \right>_{t_{\ell + 1}} \leq N^{-10} (\varphi^2/N)^{-4} \leq \varphi^{-8}N^{-6}$, so applying 
\eqref{eq:martingale_est} with $\mu = \varphi^{-8}N^{-6}$, for any $D > 0$, we have
		\[
			\P\pa{ \bigcap_{1 \leq \ell, m \leq N^{10}, \,1 \leq k \leq N} A_{\ell, m, k} } \geq 1 - N^{-D}.
		\]
We now prove \eqref{eq:tau_1_tilde_f}. Abbreviate $t = t_\ell$ and $z = z^{(m)}$ for some $1 \leq \ell, m \leq N^{10}$,
and let $g_u(z) = \tilde{f}_u(z_{t-u})$. 
By Lemmas \ref{lem:z_t-z} and \ref{lem:initial_est}, we have $\im g_0(z) \leq \frac{\varphi}{10} \frac{\kappa(z)^{1/2}}{\max(\kappa(z)^{1/2}, t)}$,
and therefore, we need only to bound the increments of $g$. By It\^o's formula, we have
		\begin{equation}
		\label{eq:ito_g}
			\rd g_{u \wedge \tau}(z) = \e_u\pa{z_{t-u}} \rd \pa{u \wedge \tau}  
				- \frac{e^{-\frac{t}{2}  }}{\sqrt{N}} \sqrt{\frac{2}{\beta}} \sum_{k=1}^N \frac{\mathfrak{v}_k(t)}{  \pa{x_{k}(t) - z_{t-u}}^2   } \rd B_k(t),
		\end{equation}
where 
		\[
			\e_u(z) = \pa{ s_u(z) - m(z) } \partial_z f_u + \frac{1}{N}\pa{ \frac{2}{\beta} - 1 } \pa{ \partial_{zz} f_u  }.
		\]
To bound the contribution of $\e_u(z)$,
by the local law, \eqref{eq:local_law_original}, and Lemma \ref{lem:z_t-z}, since $u \leq t_\ell$, we have
		\begin{multline*}
			\int_0^t \absa{ \pa{s_u (z_{t-u}) - m(z_{t-u}) }  \partial_z \tilde{f}_u(z_{t-u})   } \, \rd(u \wedge \tau) 
			\leq
			\int_0^t \frac{\varphi}{N \im(z_{t-u})} \sum_{k=1}^N \frac{\mathfrak{v}_k(u)}{  \absa{z_{t-u}-x_k(u)}^2 } \, \rd (u \wedge \tau) \\
			\leq
			\int_0^t \frac{\varphi \im \tilde{f}_u\pa{z_{t-u}}}{N \pa{ \im(z_{t-u})  }^2 } \, \rd(u \wedge \tau)
			\leq 
			\int_0^t \frac{\varphi^2 \rd u}{N( \eta + (t-u) \kappa(z)^{1/2} )^2} \frac{\kappa(z)^{1/2}}{\max(\kappa(z)^{1/2}, t)}
			=
			\frac{\kappa(z)^{1/2}}{\max(\kappa(z)^{1/2}, t)},
		\end{multline*}
and for the remaining term, we have
		\[
			\sup_{0 \leq s \leq t} \absa{
				\int_0^s \frac{1}{N} \partial_{zz} \tilde{f}_u(z_{t-u}) \, \rd (u \wedge \tau)
			}
			\leq
			\int_0^t \frac{\im \tilde{f}_u\pa{z_{t-u}}}{N   (\im(z_{t-u}))^2} \, \rd (u \wedge \tau)
			\leq
			\frac{\kappa(z)^{1/2}}{\varphi\max(\kappa(z)^{1/2}, t)}.
		\]
To control the martingale term in \eqref{eq:ito_g}, we bound $\sup_{0 \leq s \leq t} \abs{M_s}$ where
		\[
			M_s = \int_0^s \frac{e^{- \frac{u}{2} }}{ \sqrt{N} } 
				\sum_{k=1}^N \frac{ \mathfrak{v}_k(u)  }{ (z_{t-u} - x_k(u))^2   } \rd B_k(u \wedge \tau).
		\]
To do so, we will prove	
		\begin{equation}
		\label{eq:martingale_term_bound}
			\varphi^{1/10} \int_0^t \frac{1}{N} \sum_{k=1}^N \frac{\mathfrak{v}_k(u)^2}{\abs{z_{t-u}-x_k(u)}^4}
			\rd(u \wedge \tau) 
			\leq 
			\frac{\varphi^{4+ \frac{1}{5} }}{N^2} \int_0^t \rd u \int_A^B
				\frac{\rd \rho(x)}{  \abs{z_{t-u}-x}^4 \max(\kappa(x), u^2) }
			\leq  \frac{ C \varphi^{1/5} \kappa(z)}{\max(\kappa(z), t^2)}.
		\end{equation}
Assuming \eqref{eq:martingale_term_bound}, \eqref{eq:martingale_est} gives 
		\[
			\sup_{0 \leq s \leq t} \abs{M_s} \leq C \varphi^{1/10} \frac{\kappa(z)^{1/2}}{\max(\kappa(z)^{1/2}, t)}
		\]
with overwhelming probability,
and by a union bound, it follows that for any $D>0$, there exists $N_0(D)$ such that for $N \geq N_0(D)$,
		\[
			\P(\tau < 1)
			\leq 
			\P\pa{
				\sup_{  0 \leq \ell, m \leq N^{10}, \, \kappa(z^{(m)}) \geq \varphi^2 N^{-2/3}, \, 0 \leq s \leq t_\ell }
				\im \tilde{f}_{s\wedge \tau} \pa{z^{(m)}_{t_\ell - s \wedge \tau}}
				> \frac{\varphi}{2} \frac{ \kappa(z^{(m)})^{1/2} }{   \max(\kappa(z^{(m)})^{1/2}, t)      }
			} < N^{-D}.
		\]
We now prove \eqref{eq:martingale_term_bound}. First, by \eqref{eq:rigidity_set}, for $u \leq \tau_0$, we have
		\[
			\varphi^{1/10} \int_0^t \frac{1}{N} \sum_{k=1}^N \frac{\mathfrak{v}_k(u)^2}{\abs{z_{t-u}-x_k(u)}^4}
			\rd(u \wedge \tau) 
			\leq
			\varphi^{1/10} \int_0^t \frac{1}{N} \sum_{k=1}^N \frac{\mathfrak{v}_k(u)^2}{\abs{z_{t-u}-\gamma_k}^4}
			\rd(u \wedge \tau). 
		\]
Let $k_j = \lfloor j \varphi^2 \rfloor$ and $I_j = \llbracket k_j, k_{j+1} \rrbracket \cap \llbracket 1, N \rrbracket$, $0 \leq j \leq N/\varphi^2$.
Then
		\[
			\frac{1}{N} \sum_{k=1}^N \frac{\mathfrak{v}_k(u)^2}{\abs{z_{t-u}-\gamma_k(u)}^4}
			\leq
			\frac{1}{N} \sum_{0 \leq j \leq N / \varphi^2} \pa{ \max_{k \in I_j} \mathfrak{v}_k(u) }
				\pa{ \max_{k \in I_j} \frac{1}{ \abs{z_{t-u} - \gamma_k}^4   }  }
				\pa{ \sum_{k \in I_j} \mathfrak{v}_k(u) }.
		\]
For each $0 \leq j \leq N / \varphi^2$, pick $n = n_j$ such that $\abs{z^{(n)} - \gamma_{k_j}} < N^{-9}$. Then we have
$\sum_{k \in I_j} \mathfrak{v}_k(u) \leq \eta_n \im \tilde{f}_u(z^{(n)})$. To estimate $\im \tilde{f}_u(z^{(n)})$, choose
$\ell$ such that $t_\ell \leq u < t_{\ell + 1}$. On the event $\cap_k A_{\ell, m, k}$ and $u \leq \tau$, 
as in the argument following \eqref{eq:tau_1_tilde_f}, we have $\abs{ \tilde{f}_u(z^{(n)}) - \tilde{f}_{t_\ell}(z^{(n)})   } < N^{-2}$.
It follows that
		\begin{equation}
		\label{eq:bound_sum_vk}
			\sum_{k \in I_j} \mathfrak{v}_k(u) \leq \eta_n \im \tilde{f}_u(z^{(n)}) + N^{-2} \leq \frac{\varphi^3}{N\max( \kappa(\gamma_{E_n})^{1/2}, u)}. 
		\end{equation}
Furthermore, from $\gamma_k \sim N^{-2/3} (\hat{k})^{-1/3}$, we have that there is a universal constant $c$ such that for any $j$
and $k_1, k_2 \in I_j$,	
		\[
			c \abs{z_{t-u} - \gamma_{k_2}} \leq \abs{z_{t-u} - \gamma_{k_1}} \leq c^{-1} \abs{z_{t-u} - \gamma_{k_2}}.
		\]
Therefore,
		\begin{equation}
		\label{eq:deterministic_part_4th_power}
			\max_{k \in I_j} \frac{1}{\abs{  z_{t-u} - \gamma_k }^4} \leq C\varphi^{-2} \sum_{k \in I_j} \frac{1}{ \abs{z_{t-u} - \gamma_k}^4   }.
		\end{equation}
Combining \eqref{eq:bound_sum_vk} and \eqref{eq:deterministic_part_4th_power} establishes the first bound in \eqref{eq:martingale_term_bound}.
The second bound is deterministic, and follows from the proof of Lemma A.2 in \cite{Bou2020}. 
\end{proof}

Proposition \ref{prop:im_f} gives us an important a priori estimate for the proof of the following Proposition, which 
we will apply in Proposition \ref{prop:coupling}, establishing the main result of this section.

\begin{proposition}
	\label{prop:flow_est}
	For any $D > 0$, there exists $N_0(D)$ such that for any $N \geq N_0$,
		\[
			\P\pa{
				\absa{
					f_t(z) - f_0\pa{z_t}
				} \leq \frac{\varphi^{30}}{N\im(z) \sqrt{\kappa(z)} },
				\, 0 < t < 1,\, z \in \mathcal{S}(\varphi) 
			} > 1 - N^{-D}.			
		\]
\end{proposition}

\begin{proof}
	This proof is similar to the proof of Proposition \ref{prop:im_f}, which we use as an input.
	For any $1 \leq \ell, m \leq N^{10}$, define $t_\ell = \ell N^{10}$ and $z^{(m,p)} = E_m + \ii \eta_p$ where
		\[
			\int_{-\infty}^{E_m} \rd \rho = (m-1/2)N^{-10}, \text{ and } \eta_p = \frac{\varphi^2}{N \kappa\pa{z^{(m,p)}}^{1/2}} + p N^{-10}.
		\]
	We also define the stopping times with respect to $\mathcal{F}_t = \sigma(B_k(s),\, 0 \leq s \leq t,\, 1 \leq k \leq N)$,
		\begin{align}
			\tau_{\ell, m, p} &= \inf\left\{
					0 \leq s \leq t_\ell \, : \, \absa{f_s\pa{z_{t_\ell - s}^{(m,p)}} 
						- f_0\pa{z_{t_\ell}^{(m,p)}}} > \frac{\varphi^{25}}{N\eta_p \sqrt{\kappa\pa{ z^{(m,p)}  }}}
				\right\} \\
			\tau_1 &= \inf\left\{
					0 \leq t \leq 1 \, : \, \exists k \in \llbracket 1, N \rrbracket, \, \absa{\mathfrak{u}_k^{(\nu)}(t)} 
																	> \frac{\varphi^{10}}{N \max (\hat{k}/N)^{1/3}, t )   }
				\right\}\\
				\label{eq:tau1}
			\tau &= \min\left\{ \tau_0, \tau_1, \tau_{\ell, m, p} \,:\, 0 \leq \ell, m, p \leq N^{10}, \, \absa{E_m} \leq 2 \right\},
		\end{align}
	with the convention $\inf \emptyset = 1$.
	We first show it is sufficient to prove there exists $N_0(D)$ such that 
		\begin{equation}
		\label{eq:suff_to_prove}
			\P(\tau = 1) > 1 - N^{-D} 
		\end{equation}
	for any $N \geq N_0(D)$.
	Let	
		\[
			A_{\ell, m, k} = \left\{
				\sup_{t_\ell \leq u \leq t_{\ell + 1}} \absa{
					\int_{t_\ell}^u \frac{e^{-\frac{s}{2}} \mathfrak{u}_k(s)}{  \pa{ z^{(m)} - x_k(s)   }^2   } \, \rd B_k(s)
				} < N^{-3}
			\right\}.
		\]
	Then,
		\[
			\{ \tau = 1 \} \bigcap_{ \substack{1 \leq \ell, m \leq N^{10} \\ 1 \leq k \leq N} } A_{\ell, m, k}
			\subset
			\bigcap_{ \substack{ z \in \mathcal{S}, 0 < t < 1} } \left\{
				\absa{
					f_t\pa{z} - f_0\pa{z_t}
				} \leq \frac{\varphi^{30}}{N\eta \sqrt{\kappa(z)}}
			\right\}.
		\]
	Indeed, for any given $z$ and $t$, choose $w=z^{(m,p)}$ and $t_\ell$ such that $t_\ell \leq t \leq t_{\ell + 1}$ and $\abs{z- w} < N^{-5}$,
	and write
		\[
			\absa{
				f_t\pa{z} - f_0\pa{z_t}
			}
			\leq \absa{f_t(z) - f_t(w)} + \absa{f_t(w) - f_{t_\ell}(w)} + \absa{f_0(z_t) - f_0(z_{t_\ell})} + \absa{f_0(z_{t_\ell}) - f_0(w_{t_\ell})}
				+ \absa{f_{t_\ell}(w) - f_0(w_{t_\ell})}.
		\]
	On the event $\{ \tau = 1 \}$, 
		\[
			\absa{f_t(z) - f_t(w)} \leq \absa{  \sum_{k=1}^N \frac{ \mathfrak{u}_k(t)   }{x_k(t) - z}   - \sum_{k=1}^N \frac{ \mathfrak{u}_k(t)   }{x_k(t) - w} }
			\leq
			N \pa{\sup_{1 \leq k \leq N} \absa{ \mathfrak{u}_k(t) }} \frac{ \abs{z-w}   }{ (\im z) (\im w) }
			=
			\OO\pa{N^{-2}},
		\] 	
	say. Applying Lemma \ref{lem:z_t-z}, the same argument bounds 
	$\absa{f_0(z_t) - f_0(z_{t_\ell})}$ and $\absa{f_0(z_{t_\ell}) - f_0(w_{t_\ell})}$ similarly.
	Next, we have the crude bounds $\sup_{1\leq k \leq N} \absa{ \mathfrak{u}_k(u) } < 1$,
		\begin{align*}
			\absa{s_u(E+\ii \eta)} \leq \eta^{-1}, 
			\quad
			\absa{\partial_z f_u(E+\ii \eta)} \leq N  \eta^{-2},
			\quad
			\absa{\partial_{zz} f_u(E + \ii \eta)} \leq N \eta^{-3}.
		\end{align*}
	Therefore, by \eqref{lem:ito_ft}, on the event $\cap_k A_{\ell, m, k}$, we have $\absa{f_{t_\ell}(w) - f_0(w_{t_\ell})} = \OO(N^{-2})$ as well.
	Finally, on the event $\{ \tau = 1\}$, we have $\absa{f_{t_\ell}(w) - f_0(w_{t_\ell})} \leq \varphi^{25}/N\eta_p \sqrt{\kappa(z^{(m,p)})}$. 
	To establish the sufficiency of \eqref{eq:suff_to_prove}, it remains to show $\P(\cap_{\ell, m, k}A_{\ell, m, k}) \geq 1-N^{-D}$ for any $D>0$.
	Let $M_u = e^{-\frac{s}{2}}\int_{t_\ell}^u \frac{ \mathfrak{u}_k(s)}{  ( z^{(m,p)} - x_k(s) )^2  } \rd B_k(s)$, and note the deterministic estimate
		\[
			\left<M\right>_{t_{\ell + 1}} 
			\leq \int_{t_\ell}^{t_{\ell+1}}  \frac{ \absa{ \mathfrak{u}_k(s)  }^2   }{   \absa{z^{(m,p)} - x_k(s)}^4     } \, \rd s 
			< \varphi^{-8}N^{-6}.
		\]	
	Choosing $\mu = \varphi^{-8}N^{-6}$ in \eqref{eq:martingale_est}, we have $\P(A_{\ell, m, k}) \geq 1 - e^{-c \varphi^{1/5}}$, 
	from which we conclude it is sufficient to prove \eqref{eq:suff_to_prove} in order to prove the proposition. \\

	Let $t = t_\ell$, $z = z^{(m,p)}$ for some $0 \leq \ell, m, p \leq N^{10}$, $\abs{E_m} \leq 2$, and let $g_u(z) = f_u(z_{t-u})$.
	By It\^{o}'s formula and \eqref{eq:characteristics}, we have
		\[
			\rd g_{u \wedge \tau}(z) = \e_u\pa{z_{t-u}} \rd \pa{u \wedge \tau} 
			- \frac{e^{-u/2}}{\sqrt{N}} \sqrt{\frac{2}{\beta}} \sum_{k=1}^N \frac{\mathfrak{u}_k(u)}{\pa{z_{t-u} - x_k(u)}^2} \rd B_k \pa{u \wedge \tau},
		\]
	where 
		\[
			\e_u(z) = \pa{ s_u(z) - m(z) } \partial_z f_u + \frac{1}{N}\pa{ \frac{2}{\beta} - 1 } \pa{ \partial_{zz} f_u  }.
		\]
	By the local law \eqref{eq:local_law_original}, we have
		\[
			\int_0^t \absa{
				s_u\pa{z_{t-u}} - m\pa{z_{t-u}}}
				\absa{ \partial_z f_u\pa{z_{t-u}} } \, \rd (u \wedge \tau)
			\leq
			\int_0^t \frac{\varphi}{N \im \pa{z_{t-u}}} \sum_{k=1}^N \frac{ \absa{\mathfrak{u}_k(u)}  }{  \absa{  z_{t-u} - x_k(u) }^2   } \, \rd (u \wedge \tau),
		\]
	and applying Proposition \ref{prop:im_f} and Lemma \ref{lem:z_t-z}, we have
		\begin{multline*}
			\int_0^t \frac{\varphi}{N \im \pa{z_{t-u}}} \sum_{k=1}^N \frac{ \abs{\mathfrak{u}_k(u)}  }{  \absa{  z_{t-u} - x_k(u) }^2   } \, \rd (u \wedge \tau)
			\leq
			\int_0^t \frac{\varphi^2 \im \tilde{f} \pa{ z_{t-u} }}{N \pa{\im\pa{z_{t-u}}}^2 } \, \rd (u \wedge \tau) \\
			\leq
			\frac{\varphi^2}{N} \frac{\kappa(z)^{1/2}}{ \max\pa{ \kappa(z)^{1/2}, t  }  } \int_0^t \frac{1}{\pa{\eta + \kappa(z)^{1/2}(t-u)}^2} \, \rd (u \wedge \tau)
			\leq
			\frac{1}{N\eta \kappa(z)^{1/2}}.
		\end{multline*}
	Furthermore, we have
		\[
			\frac{1}{N} \int_0^t \absa{\partial_{zz} f_u\pa{z_{t-u}}} \, \rd u 
			\leq
			\int_0^t \frac{\varphi}{N \im \pa{z_{t-u}}} \sum_{k=1}^N \frac{ \abs{\mathfrak{u}_k(u)}  }{  \absa{  z_{t-u} - x_k(u) }^2   } \, \rd (u \wedge \tau)
			\leq
			\frac{1}{\varphi N\eta \kappa(z)^{1/2}}.
		\]	
	Finally, we bound $\sup_{0 \leq s \leq t} \abs{M_s}$, where
		\[
			M_s = \int_0^s \frac{e^{-u/2}}{\sqrt{N}} \sum_{k=1}^N \frac{ \mathfrak{u}_k(u)  }{ \pa{z_{t-u} - x_k(u)}^2} \, \rd B_k(u \wedge \tau).
		\]
	To do so, we will bound
		\begin{equation}
		\label{eq:bracket}
			\int_0^t \frac{1}{N} \sum_{k=1}^N \frac{ \absa{\mathfrak{u}_k(u)}^{2}  }{ \absa{z_{t-u} - x_k(u)}^4} \, \rd (u \wedge \tau).
		\end{equation}
	and apply \eqref{eq:martingale_est}. On the set $\mathcal{A}$ (see \eqref{eq:rigidity_set}), we have
	$\abs{x_k(u) - \gamma_k} \ll \abs{z_{t-u} - \gamma_k}$, and therefore, applying \eqref{eq:tau1}, we may bound \eqref{eq:bracket} by
		\begin{align}
		\nonumber
			\int_0^t \frac{1}{N} \sum_{k=1}^N \frac{ \mathfrak{u}^2_k(u)  }{ \absa{z_{t-u} - \gamma_k}^4}  \rd (u \wedge \tau)
			&\leq 
			\frac{\varphi^{20}}{N^2} \int_0^t \frac{1}{N}   \sum_{k=1}^N \frac{ 1  }{ \absa{z_{t-u} - \gamma_k}^4}  \frac{1}{ \kappa\pa{\gamma_k} \vee u^2}   
				 \rd (u \wedge \tau) \\
			&\leq 
			\frac{\varphi^{20}}{N^2} \int_0^t  \int_{-2}^2 \frac{ \rd \rho(x)  }{ \absa{z_{t-u} - x}^4 \kappa(x)}  \rd x\rd (u \wedge \tau).
		\label{eq:int_second_term}
		\end{align}
	Note that for $w = E(w) + \ii \eta(w)$, and $E(w), \eta(w) > 0$, since $\int_0^\infty \frac{\rd x}{(w-x)\sqrt{x}} = \pi w/(-w)^{3/2}$, we have
		\[ 
			\int_{-2}^2 \frac{\rd \rho(x)}{\abs{w-x}^4 \kappa(x)} 
			\leq 
			\frac{1}{\eta(w)^3} \im \int_{-2}^2 \frac{\rd x}{(w-x) \kappa(x)^{1/2}}
			\leq
			\frac{C}{\eta(w)^3\kappa(w)^{1/2}}.
		\]
	Therefore, applying Lemma \ref{lem:z_t-z}, we can bound \eqref{eq:int_second_term} by
		\[
			\int_0^t \frac{C}{\eta\pa{z_{t-s}}^3 \kappa\pa{z_{t-s}}^{1/2}} \, \rd s
			\leq
			\int_0^t \frac{C}{\pa{\eta + s \kappa(E)^{1/2}}^3\kappa(z)^{1/2}} \, \rd s \leq \frac{C}{\eta^2 \kappa(z)},
		\]
	which means in total we have
		\[
			\int_0^t \frac{1}{N} \sum_{k=1}^N \frac{ \mathfrak{u}^2_k(u)  }{ \absa{z_{t-u} - x_k(u)}^4}  \rd (u \wedge \tau)
			\leq
			\frac{\varphi^{25}}{(N\eta)^2 \kappa(z)}.
		\]
	Applying \eqref{eq:martingale_est} and combining and the above estimates, we have
		\[
			\P\pa{
			\sup_{0 \leq s \leq t_\ell}
				\absa{
					f_{s \wedge \tau} \pa{z^{(m,p)}} - f_{t_\ell - (s \wedge \tau)} \pa{z^{(m,p)}_{t_\ell - (s\wedge \tau)}}
				}
			\geq
			\frac{\varphi^{25}}{N \eta_p \kappa\pa{z^{(m,p)}}^{1/2}}
			}
			\leq N^{-D}
		\]
	for any $D>0$ and $N$ sufficiently large. Finally, a union bound gives that for any $D>0$ and $N$ sufficiently large,
		\[
			\P(\tau < 1)
			\leq
			\P\pa{
			\sup_{\substack{ 0 \leq \ell, m, p \leq N^{10} \\  0 \leq s \leq t_\ell}}
				\absa{
					f_{s \wedge \tau} \pa{z^{(m,p)}} - f_{t_\ell - (s \wedge \tau)} \pa{z^{(m,p)}_{t_\ell - (s\wedge \tau)}}
				}
			\geq
			\frac{\varphi^{25}}{N \eta_p \kappa\pa{z^{(m,p)}}^{1/2} }
			}
			\leq N^{-D},
		\]
	which completes the proof.
\end{proof}

With Proposition \ref{prop:flow_est} in hand, we come to the main result of this section, Proposition \ref{prop:coupling}.

	\begin{proposition}
	\label{prop:coupling}
	Let $\epsilon > 0$, $\tau = N^{-\epsilon}$, and let $z_\tau$ be as in \eqref{eq:characteristics}. With $\eta(E) = e^{(\log N)^{1/4}} \ell(E)$, let $z = E + \ii \eta(E)$.
	Then for any $\delta > 0$, 
		\[ 
			\lim_{N\to\infty} \P \pa{
				\pa{ L_N(\bl(\tau), z) - L_N(\bmu(\tau), z) } - \pa{ L_N(\bl(0), z_\tau) - L_N(\bmu(0), z_\tau)} > \delta
			} = 0.  
		 \]
	\end{proposition}

\begin{proof}[Proof of Proposition \ref{prop:coupling}]
	Observe that with $x_k^{(\nu)}$ as in \eqref{eq:DBM}, we have
		\[
			\frac{\rd}{\rd \nu} L_N(\bx^{(\nu)}(t), z) = e^{ \frac{t}{2}} f^{(\nu)}_t(z).
		\]
	Therefore
		\[
			\pa{ L_N(\bl(\tau), z) - L_N(\bmu(\tau), z) } - \pa{ L_N(\bl(0), z_\tau) - L_N(\bmu(0), z_\tau)} 
			= 
			e^{\frac{t}{2}} \int_0^1 \pa{ f^{(\nu)}_\tau(z) - f^{(\nu)}_0(z_\tau) }  \rd \nu.
		\]
	Applying Proposition \ref{prop:flow_est}, we can bound the moments, 
	$\E(|\int_0^1 ( f^{(\nu)}_\tau(z) - f^{(\nu)}_0(z_\tau) )  \rd \nu|^{2p}) = \oo(1)$ 
	with probability $1-N^{-D}$ for any $D> 0$ and $N$ sufficiently large. This completes the proof. 
\end{proof}

\section{Proof of Theorem \ref{thm:log_corr_field}}
\label{sec:end_of_proof}
In this section, we prove Theorem \ref{thm:log_corr_field} with $m=1$, for $\re L_N$ in the real generalized Wigner case. 
The generalization to $m > 1$ follows in an identical fashion, as do the extensions to $\im L_N$ and the complex generalized Wigner case.
We first collect the required inputs, beginning with the following four-moment matching theorem.  
\begin{theorem}
\label{thm:4_moment_matching_det}
	Let $F: \mathbb{R} \to \mathbb{R}$ be smooth with compact support, and let $W$ and $V$ be two Wigner matrices 
	such that for $1 \leq i, j \leq N$,
		\begin{numcases}{\mathbb{E}\left(w_{ij}^a\right) = }
		\label{moment matching assumption 3}
			\mathbb{E}\left(v_{ij}^a\right)  & $a \leq 3$ \\
			\label{moment matching assumption 4}
			\mathbb{E}\left(v_{ij}^a\right)+ \OO(\tau) & $a = 4$,
		\end{numcases}
	where $\tau = N^{-\e}$ is as in Proposition \ref{prop:coupling}. Further, let $c^R_N$, $c_N^I$ be any deterministic sequences and define
		\[
			u_N^R(W) = \frac{\re\log \det \left(W -z \right) +c_N^R}{\sqrt{\log N}},
			\quad
			u_N^I(W) = \frac{\im\log \det \left(W -z \right) +c_N^I}{\sqrt{\log N}}
		\]
	where $z=E+\ii \eta$, $E \in [-2,2]$, and $\eta = \eta(E) = e^{(\log N)^{1/4}} \ell(E)$ is as in Proposition \ref{prop:smoothing}. Then
		\begin{equation}
		\label{eqn:enough}
			\lim_{N\to \infty} \mathbb{E} \left(F\left( u_N^R(W)\right) - F\left( u_N^R(V) \right)\right) = 0
			\text{ and }
			\lim_{N\to \infty} \mathbb{E} \left(F\left( u_N^I(W)\right) - F\left( u_N^I(V) \right)\right) = 0.
		\end{equation} 
\end{theorem}

\begin{proof}
	The proof is identical to the proof of \cite[Theorem 4.4]{BouMod2019}.
\end{proof}

\begin{proposition}
	\label{prop:meso_var_bound} 
	Let $z = E + \ii \eta(E)$ with $E \in [-2,2]$ and $\eta(E) = e^{(\log N)^{1/4}} \ell(E)$.
	Let $\tau = N^{-\e}$, $z_\tau$ be as in \eqref{eq:characteristics} with $z_0 = z$,
	and let $\chi(x)$ be a smooth function such that $\chi(x) \leq 1$, $\chi(x) \equiv 1$ for $x\in[-3,3]$ and $\chi(x) = 0$ 
	for $\abs{x} \geq 4$. There exists $C>0$ such that for large enough $N$, and $\bl$ the eigenvalues of a generalized Wigner matrix, we have
	\[ 
		\Var \pa{ \re L_N(\bl(0), z_\tau)\chi(\bl(0)) } \leq C\pa{\epsilon \log N}
		\text{ and }
		\Var \pa{ \im L_N(\bl(0), z_\tau) \chi(\bl(0))} \leq C\pa{\epsilon \log N}.
	\]
	Here and in the rest of this section, $\Var(f(\bl)\chi(\bl)) = \Var(\sum f(\lambda_k) \chi(\lambda_k))$.
\end{proposition}
\begin{proof}
	Proposition \ref{prop:var_GW} gives the expression for the variance of linear statistics for a general class of functions,
	which includes $\re L_N(\bl(0), z_\tau)\chi(\bl(0))$ and $\im L_N(\bl(0), z_\tau)\chi(\bl(0))$, for $0< \e <1/100$, say. To bound this variance, we use 
	Lemma \ref{lem:var_A_log}, which gives a formula for this variance, and \cite[Lemma 4.9]{BouModPai2021}, which evaluates this expression more explicitly. 
\end{proof}

\begin{lemma}
\label{lem:exp}
	Let $\bl$ be the eigenvalues of a generalized Wigner matrix, and let
	\[
		\delta_N^{\rm(GW)}(z) = \frac{1}{4} \pa{ \frac{2}{\beta}-1} \log \pa{\kappa(z) \vee N^{-2/3}}.
	\]
	Then for any $z = E + \ii \eta$, $E \in [-2,2]$ and $0 < \eta < 1$,
	\[
		\E\left[ \re L_N(\bl, z) \right] = \delta_N^{\rm(GW)}(z)+ \OO(1),
		\quad
		\E\left[ \im L_N(\bl, z) \right] = \OO(1).
	\]
	Here $\beta=1$ corresponds to the real case, and $\beta=2$ to the complex case.
\end{lemma}
\begin{proof}
	When $\beta=1$, this follows from Lemmas \ref{lem:E} and \ref{lem:E_simplified}, which we have stated for the real case.
	One can redo these calculations for the complex case, however, for purposes of identifying the pre-factor $(2/\beta-1)$, we
	can compare with \cite[(4.19)]{BouModPai2021} which covers the GUE case. 
\end{proof}

\begin{proof}[Proof of Theorem \ref{thm:log_corr_field}]
Choose $0 < \e_1 < 1/100$, let $\tau = N^{-\e_1}$, and recall we define $z_\tau$ as in \eqref{eq:characteristics} with $z_0 = z = E + \ii \eta$,
$E \in \mathcal{G}_{\e, c}$ and $\eta = e^{(\log N)^{1/4}} \ell(E)$.
Let $\bl(0)$ (resp. $\bmu(0)$) denote the eigenvalues of a generalized Wigner (resp. GOE) matrix.
By \cite[Theorem 1.8]{BouModPai2021} and Proposition \ref{prop:smoothing}, for some explicit deterministic sequence 
$\delta_N(z)$ we have
	\begin{equation}
	\label{eqn:1}
		\frac{ \re L_N(\bmu(\tau), z) - \delta_N(z)}{\sqrt{\log N}} \to \mathscr{N}(0,1),
	\end{equation}
and therefore, writing $m_N(t) = \delta_N^{\rm (GW)}(z_t)- \delta_N(z_t)$, Proposition \ref{prop:coupling} implies that
	\begin{equation}
	\label{eq:gaussian_plus_remainder}
		\frac{\re L_N(\bl(\tau), z) -\delta_N^{\rm(GW)}}{\sqrt{\log N}} 
		+ \frac{  \re L_N(\bmu(0), z_\tau) - \re L_N(\bl(0), z_\tau) +m_N(\tau)
			- m_N(\tau) + m_N(0)
		}{\sqrt{\log N}} \to \mathscr{N}(0,1).
	\end{equation}
Note that by \cite[Theorem 1.8]{BouModPai2021} and Lemma \ref{lem:exp}, $\abs{\delta_N(z_\tau) - \delta_N^{\rm(GW)}(z_\tau)} = \OO(1)$.
Let $\chi(x)$ be a smooth function such that $\chi(x) \leq 1$, $\chi(x) \equiv 1$ for $x\in[-3,3]$, and $\chi(x) = 0$ 
for $\abs{x} \geq 4$.
By Proposition \ref{prop:meso_var_bound} and Lemma \ref{lem:exp}, 
the second term in \eqref{eq:gaussian_plus_remainder}, call it $R$, satisfies
$\E(R\chi)^2 < C \e_1$, for some universal constant $C > 0$. 
And by Corollary \ref{cor:rigidity}, $(1-\chi)R \to 0$ in probability as $N\to\infty$.
Consequently, for any fixed, smooth and compactly supported function $F$, we have
	\begin{align*}
		\E\pa{  F\pa{  \frac{\re L_N(\bl(\tau), z) -\delta_N^{\rm(GW)}(z)}{\sqrt{\log N}}    }} 
		&= 
		\E\pa{F\pa{ \frac{ \re L_N(\bmu(\tau), z) -\delta_N(z)}{\sqrt{\log N}} } }+ \OO\pa{\norm{F}_{\rm Lip} \E(\abs{R})} \\
		&=
		\E\pa{F(\mathscr{N}(0,1))} + \oo(1) + \OO(\e_1^{1/2}).
	\end{align*}
Applying Theorem \ref{thm:4_moment_matching_det}, the above equation implies
	\[
		\E\pa{ F\pa{ \frac{ \re L_N(\bl(0), z) -\delta_N^{\rm(GW)}(z)}{\sqrt{\log N}} } } = \E ( F(\mathscr{N}(0,1)) )+\oo(1)+\OO(\e_1^{1/2}).
	\]
Applying Proposition \ref{prop:smoothing} again, we have
	\[
		\E\pa{ F\pa{ \frac{ \re L_N(\bl(0), E) -\delta_N^{\rm(GW)}(z)}{\sqrt{\log N}} } } = \E ( F(\mathscr{N}(0,1)) )+\oo(1)+\OO(\e_1^{1/2}).
	\]
Since $\e_1$ is arbitrarily small, this concludes the proof.
\end{proof}

\setcounter{equation}{0}
\setcounter{theorem}{0}
\renewcommand{\theequation}{A.\arabic{equation}}
\renewcommand{\thetheorem}{A.\arabic{theorem}}
\appendix
\setcounter{secnumdepth}{0}
\section[Appendix A: Mesoscopic Central Limit Theorem]
{Appendix \mynameis{A}:\ \ \ Mesoscopic Central Limit Theorem
\label{sec:mesoscopic_CLT}}

In this section, we explain how to modify the proof of \cite[Theorem 5.21]{LanLopSos2021} to 
prove Proposition \ref{prop:char_function}, which we apply to the linear statistics $L_N(\bl(0), z_\tau)$ 
of Proposition \ref{prop:coupling}, when we conclude the proof of Theorem \ref{thm:log_corr_field} in Section \ref{sec:end_of_proof}.
The argument for \cite[Theorem 5.21]{LanLopSos2021} 
proves a central limit theorem for linear statistics of functions with support in the bulk, for the more general "Wigner-Type" class. 
Here we extend 
the argument mildly to include functions with support at the edge. The variance we find is the same as in \cite{LanLopSos2021}, where
the authors use a Taylor expansion to recover the $H^{1/2}$ norm, which goes back to at least \cite{LytPas2009}, as the main term. At the edge, this Taylor expansion
is not sufficient for our purpose,
however, in the generalized Wigner setting, the variance simplifies algebraically which allows us to conclude. 

\subsection{Model Assumptions. }

We begin by introducing the Wigner-Type model. Our first goal will be to show that the central limit theorem argument in \cite{LanLopSos2021}
holds up to the edge, with a variance formula that we will analyze further in the following section, for generalized Wigner matrices.

\begin{definition}
	A Wigner-type matrix $H=H(N)$ is a symmetric or Hermitian $N \times N$ matrix whose upper-triangular elements
	$H_{ij} = \overline{H_{ji}}$, $i \leq j$, are independent random variables with mean zero and variances 
	$\sigma_{ij}^2 = \E(\abs{H_{ij}}^2)$ that satisfy $\sigma_{ij}^2 \sim N^{-1}$ for all $i, j \in \llbracket 1, N \rrbracket$.
	In the Hermitian case, we assume $\Var \re(H_{ij}) \sim \Var \im(H_{ij})$ and that
	$\re(H_{ij})$ and $\im(H_{ij})$ are independent. We will assume that for every $p \geq 2$, there exists a constant $C_p$
	such that
	\[
		\E \pa{ \abs{ H_{ij} }^p} \leq C_p.
	\]
\end{definition}

In the course of the proof of our main theorem, we will make use of results from \cite{LanLopSos2021}.
We therefore impose similar assumptions on our Wigner-Type matrices. For their statements, first let 
	\[
		\mathbb{H}_+ = \{ z \in \C \,:\, \im(z) > 0 \}.
	\]
Following \cite{AjaErdKru2019}, consider the quadratic vector equation for $z \in \mathbb{H}_+$ and a vector $\mathfrak{m}(z) \in \C^N$,
	\begin{equation}
	\label{eq:qve}
		- \frac{1}{\mathfrak{m}_i(z)} = z + (S\mathfrak{m}(z))_j,
		\quad
		S_{ij} = \sigma_{ij}^2.
	\end{equation}
By \cite[Theorem 2.1]{AjaErdKru2019}, \eqref{eq:qve} has a unique solution $\mathfrak{m}: \mathbb{H}_+ \to \mathbb{H}_+^N$.
We now introduce our first assumption, as in \cite{AltErdKru2020}.
\begin{enumerate}[(A)]
	\item There is a universal constant $C>0$ such that
		\[
			\abs{\mathfrak{m}_i(z)} \leq C
		\]
		for all $i \in \llbracket 1, N \rrbracket$.
\end{enumerate}

Theorem 2.10 in \cite{AjaErdKru2019} gives sufficient conditions under which assumption (A) is known to hold.
For our second assumption, let
	\begin{equation}
	\label{eq:def_rho}
		\varrho(E) = \lim_{\eta \to 0} \frac{1}{ \pi N} \sum_{i=1}^N \im \mathfrak{m}(E + \ii \eta).
	\end{equation}
Under assumption (A), Theorem 4.1 in \cite{AltErdKru2020} describes the possible behavior of $\rho$. In particular, $\rho$ is a continuous
density of compact support consisting of finitely many intervals. At its edges, $\rho$ vanishes as a square root.
We will assume $\rho$ has only one interval of support, and that it is bounded below away from the edges, 
excluding from our study the case where $\rho$ has small local minima or cusps. 
\begin{enumerate}[(B)]
	\item There exist (sequences of real numbers) $A < B$ and constants $C,c > 0$ such that
		$B- A > c$, $\abs{\alpha} + \abs{\beta} \leq C$ and with $c \leq f(E) \leq C$,
			\[
				\varrho(E) =  f(E) \sqrt{ (E-A)(B-E)_+}.
			\]
	\end{enumerate}

Under assumptions (A) and (B), we now state special cases of the local law, \cite[Theorem 1.7]{AjaErdKru2017} and \cite[Theorem 2.6]{AltErdKru2020}. 
The local law is a fundamental input to the entire argument in this section.

\begin{theorem}[Local Law]
	Let $W$ be a Wigner-Type matrix, define $G(z) = (W-z)^{-1}$, and let
		\begin{equation}
		\label{eq:def_E_gamma_delta_m}
			\mathcal{E}_\gamma^\delta = \{ z = E + \ii \eta \in \C \text{ s.t. } E \in [A-\delta, B +\delta], \, N^{-1 + \gamma} \leq \eta \leq 1\},
			\quad
			\mathfrak{m}(z) = \diag(\mathfrak{m}_1(z), \dots, \mathfrak{m}_N(z)),
		\end{equation}	
	where $\mathfrak{m}_i(z)$, $i = 1, \dots N$ are the solution to \eqref{eq:qve}.
	Then for any $D, \gamma, \e > 0$, any sufficiently small $\delta > 0$, any deterministic vectors $\mathbf{x}, \mathbf{y}$, and any $N$ sufficiently large,  
		\begin{equation}
		\label{eq:isotropic_local_law}
			\P \pa{
					\abs{
						\left<x, (G(z)-\mathfrak{m}(z)), y\right>
					} 
					\geq
					N^{\e} \norm{\mathbf{x}} \norm{\mathbf{y}} \Psi(z),
					\, z \in \mathcal{E}_\gamma^\delta
			} \leq N^{-D}, \text{ where } \Psi(z) = \sqrt{  \frac{\varrho(z) }{N\eta} }+ \frac{1}{N\eta}. 
		\end{equation}
	In particular, 
		\begin{equation}
		\label{eq:local_law_entrywise}
			\P\pa{ \max_{i, j} \, \abs{G_{ij}(z) - \mathfrak{m}_i(z) \delta_{ij}} \geq N^{\e} \Psi(z)} \leq N^{-D}.
		\end{equation}
	Furthermore, for any deterministic vector $u \in \R^N$ with $\norm{u}_\infty \leq 1$,
		\begin{equation}
		\label{eq:avg_local_law}
			\P\pa{
				\frac{1}{N} \sum_{i} u_i(G_{ii}(z) - \mathfrak{m}(z))\geq \frac{N^\e}{N\eta}, \, z \in \mathcal{E}_\gamma^\delta
			}
			\leq
			N^{-D}.
		\end{equation}
\end{theorem}

We close this our discussion on our model assumptions with the following Lemma which we will use repeatedly to control $\mathfrak{m}'$.
\begin{lemma}
	For $i = 1, \dots N$, there exists $C > 0$ such that uniformly for $z \in \mathcal{E}_\gamma^\delta$,
		\[
			\absa{ \mathfrak{m}'_i(z) } \leq C \kappa(z)^{-1/2}.
		\]
		Here $\kappa(z)$ is as in \eqref{def:l(E)_and_kappa(E)}.
\end{lemma}

\begin{proof}
	This follows directly from \cite[Corollary 7.3]{AjaErdKru2019} which bounds $\norm{\partial_z\mathfrak{m}(z)}$ in terms of 
	$\left< \im \mathfrak{m} (z)\right> = \frac{1}{N}\sum_{i} \im\mathfrak{m}_i(z)$
	in an order 1 neighborhood of the support of the spectral measure, and \cite[Corollary A.1 (c)]{AjaErdKru2019} which controls $\left< \im \mathfrak{m}(z) \right>$. 
\end{proof}

\subsection{Preliminary Estimates. }
\label{sec:clt_preliminary_estimates}

Before coming to the proof of Proposition \ref{prop:char_function}, we establish an estimate on the quantity
	\begin{equation}
	\label{def:T}
		T_{xy}(z,w) = \frac{1}{N} \sum_{i} s_{ix} G_{iy}(z) G_{yi}(w).
	\end{equation}
Our estimate here closely follow \cite[Sections 4,5]{LanLopSos2021}.
Indeed, our we first restate \cite[Proposition 5.1]{LanLopSos2021} in terms
of the control parameters from the local law $\Psi(z), \Psi(w)$ to capture the improvement of the local law near the edge.
\begin{proposition}
\label{prop:self_consistent_T}
	Let $z,w \in \mathcal{E}_\gamma^\delta$, as in \eqref{eq:def_E_gamma_delta_m}. Then
	\begin{equation}
	\label{eq:self_consistent_T}
		T_{xy}(z,w) = [S\mathfrak{m}(z)\mathfrak{m}(w)]_{xy} + [S\mathfrak{m}(z)\mathfrak{m}(w)T]_{xy} + \mathcal{E}_{xy}(z,w),
	\end{equation}
	where for any $\e > 0$, $\mathcal{E}_{xy}$ satisfies
	\[
		\E\pa{
			\abs{ \mathcal{E}_{xy}(z,w) }^2
		}
		\leq
		N^\e \pa{
			\Psi(z)^2\Psi(w) + \Psi(z)\Psi(w)^2 + \frac{1}{N} \Psi(z)\Psi(w)
		}.
	\]
\end{proposition}

We would now like to solve \eqref{eq:self_consistent_T} for $T$. To do this, 
we require estimates on the operator $({\rm I} - S \mathfrak{m}(z)\mathfrak{m}(w))^{-1}$. In fact
we will use such estimates repeatedly in the proof of Proposition \ref{prop:char_function}.
We therefore discuss these estimates before completing our estimate on $T$ in Theorem \ref{thm:Txy}. 
Our first step in this direction is to study 
	\begin{equation}
	\label{eq:F}
		F(z,w) = \abs{ \mathfrak{m}(z) \mathfrak{m}(w) }^{1/2} S \abs{ \mathfrak{m}(z) \mathfrak{m}(w) }^{1/2}.
	\end{equation}

\begin{proposition}
\label{prop:F}
	Let $F(z,w)$ be as in \eqref{eq:F}. Then $F$ is a symmetric matrix with positive entries and its largest in magnitude eigenvalue $\lambda_1(F(z,w))$ 
	is positive and simple, with an $\ell^2$ normalized eigenvector $v(z,w)$, that has strictly positive entries, and
	there exists $c_g > 0$ such that
	\begin{equation}
	\label{eq:F_spectral_gap}
		\inf_{j \neq 1} \abs{
			\lambda_1(F(z,w)) - \lambda_j(F(z,w))
		}
		\geq
		c_g.
	\end{equation}
	Furthermore, writing $\left<x\right> = \sum_{a} x_a$ for any vector $x$, we have
	\begin{equation}
	\label{eq:Fl2_general}
		\lambda_1(z,w)
		=
		\norm{F(z,w)}_{\ell^2 \to \ell^2} 
		\leq 
		1 - \frac{1}{2} \pa{
			\abs{\im w} \frac{  \left< v(w,w) \abs{\mathfrak{m}(w)} \right>    }{  \left<  v(w,w) \frac{ \abs{ \im \mathfrak{m}(w)}  }{  \abs{  \mathfrak{m}(w)  }   }  \right>  }
			+ \abs{\im z} \frac{  \left< v(z,z) \abs{\mathfrak{m}(z)} \right>    }{  \left<  v(z,z) \frac{ \abs{ \im \mathfrak{m}(z)}  }{  \abs{  \mathfrak{m}(z)  }   }  \right>  }
		},
	\end{equation}
	and, for any $C > 0$ and $\abs{z}, \abs{w} \leq C$, we have
	\begin{equation}
	\label{eq:size_of_v}
		\sqrt{N} v(z,w) \sim 1,
	\end{equation}
	Consequently, for such $z,w$, there exists $c > 0$ such that
	\begin{equation}
	\label{eq:Fl2}
		\norm{F(z,w)}_{\ell^2 \to \ell^2} 
		\leq
		1-c\pa{\im(w) \kappa^{-1/2}(w) + \im(z) \kappa^{-1/2}(z)}.
	\end{equation}
	\end{proposition}
\begin{proof}
	The estimates \eqref{eq:Fl2_general}, \eqref{eq:F_spectral_gap}, and \eqref{eq:size_of_v} follow 
	from \cite[Proposition 4.5]{LanLopSos2021}.
	The estimate \eqref{eq:Fl2} follows from \eqref{eq:Fl2_general} 
	using $\abs{\mathfrak{m}_i(z)} \sim 1$, \eqref{eq:size_of_v}, and $\abs{ \im \mathfrak{m}_i(z) } \sim \kappa(z)^{1/2}$,
	which follow from \cite{AjaErdKru2017,Erd2019}, see \cite[Proposition 4.4]{LanLopSos2021}. 
	\end{proof}

Proposition \ref{prop:F} gives us the following estimates concerning the norm of $(1 - S\mathfrak{m}(z)\mathfrak{m}(w))^{-1}$.
\begin{corollary}
\label{cor:invert_1-Smm}
	Let $C > 0$ and suppose $\abs{z}, \abs{w} \leq C$. Then
		\begin{equation}
		\label{eq:l2_and_l_infty}
			\norm{
				1 - \mathfrak{m}(z)\mathfrak{m}(w) S
			}_{\ell^2 \to \ell^2} 
			+
			\norm{
				1 - \mathfrak{m}(z)\mathfrak{m}(w) S
			}_{\ell^\infty \to \ell^\infty} 
			\leq 
			\frac{C}{ \abs{\im(w)}  \kappa^{-1/2}(w)  + \abs{\im(z)}  \kappa^{-1/2}(z)} .
		\end{equation}
	We also have
	\begin{equation}
	\label{eq:1-mmS_entries}
		\absa{
			(1 - \mathfrak{m}(z) \mathfrak{m}(w) S)^{-1}_{xy}
		}
		\leq
		\delta_{xy} + \frac{C}{N( \abs{\im(w)} \kappa^{-1/2}(w) + \abs{\im(z)} \kappa^{-1/2}(z))},
	\end{equation}
	and if $z,w$ are in the same half space, we have the improved bound
	\begin{equation}
	\label{lem:bound_1-m^2S_inv}
		\norm{ ({1 - \mathfrak{m}(z)\mathfrak{m}(w) S})^{-1} }_{\ell^\infty \to \ell^\infty} \leq \frac{C}{ \im \mathfrak{m}(z) \im \mathfrak{m}(w)}.
	\end{equation}
	Furthermore, all of the above bounds apply with $(1 - \mathfrak{m}(z)\mathfrak{m}(w)S)^{-1}$ in place of $(1 - S\mathfrak{m}(z)\mathfrak{m}(w))^{-1}$.
\end{corollary}
\begin{proof}
	The proof of \eqref{eq:l2_and_l_infty} follows exactly as in the proof of \cite[Proposition 4.7]{LanLopSos2021}, replacing \cite[Proposition 4.5]{LanLopSos2021}
	with our Proposition \ref{prop:F}. Arguing as in \cite[Proposition 4.8]{LanLopSos2021}, the estimate \eqref{eq:1-mmS_entries} follows from
	our bound for $\norm{ 1 - \mathfrak{m}(z)\mathfrak{m}(w) S}_{\ell^\infty \to \ell^\infty}$. Finally the proof of 
	\eqref{lem:bound_1-m^2S_inv} follows as in the proof of \cite[Lemma 4.9]{LanLopSos2021}.
\end{proof}

The results of Corollary \ref{cor:invert_1-Smm} will be useful in the next subsection where we calculate the characteristic function for
of the relevant linear statistics, and both Proposition \ref{prop:F} and Corollary \ref{cor:invert_1-Smm} will be important inputs
in the following subsection where we estimate the variance of these same linear statistics.
Finally, combining Proposition \ref{prop:self_consistent_T} and Corollary \ref{cor:invert_1-Smm}, we arrive at the following useful estimate for $T_{xy}$.

\begin{theorem}
\label{thm:Txy}
	Let $z,w \in \mathcal{E}_\gamma^\delta$ and define
	$A_{xy}$ by
	\[
		T_{xy}(z,w) = [(1 - S\mathfrak{m}(z)\mathfrak{m}(w))^{-1} S \mathfrak{m}(z) \mathfrak{m}(w) ]_{xy} + A_{xy}.
	\]
	Then for any $\e > 0$ we have
	\[
		\E\pa{ \abs{A_{xy}}^2 } \leq \frac{N^\e 
			\pa{
				\Psi(z)^2\Psi(w) + \Psi(z)\Psi(w)^2 + \frac{1}{N} \Psi(z)\Psi(w)
			}^{2}
		}{(\im(w) \kappa^{-1/2}(w) + \im(z) \kappa^{-1/2}(z))^2}.
	\]
\end{theorem}

\subsection{Calculation of the Characteristic Function. }
\label{sec:char_fcn}

We now turn to the proof of Proposition \ref{prop:char_function}, in which we compute the characteristic function
of the centered linear statistics
	\begin{equation}
	\label{def:e(lambda)}
		e(\lambda) = \exp[\ii \lambda(  \Tr f(W) - \E( f(W) )  )  ],
	\end{equation}
for sufficiently regular $f$. Our argument here follows \cite[Section 5]{LanLopSos2021}, with the 
minor addition of assumption (C), allowing us to include functions $f$ with support at the edge, and 
$\norm{f'}_1$ possibly depending on $N$. We begin by writing $e(\lambda)$ using the Helffer-Sj\"ostrand formula \cite{HelfferSjostrand1989},
	\begin{equation}
	\label{eq:HS_e}
		\Tr f(w) - \E( \Tr f(W) ) 
		=
		\frac{N}{2\pi} \int_{\C}
			\pa{
				\ii y \chi(y) f''(x) + \ii( f(x) + \ii f'(x) y) \chi'(y)
			}
			\pa{
				m_N(z) - \E [ m_N(z) ]
			}
		\rd x \rd y, 
	\end{equation}	
where $\chi(y)$ is a smooth function such that $\chi(y) = 1$ for $\abs{y} < 1$ and $\chi(y) = 0$ for $\abs{y} > 2$. Now fix $\mathfrak{a} > 0$, let
	\[
		\Omega_{\mathfrak{a}} = \left\{
			(x,y) \in \R^2 \,: \abs{y} > N^{-1 + \mathfrak{a}}
		\right\},
	\]
recall $\bar{\partial}_z = \frac{1}{2}(\partial_x + \ii \partial_y)$, and let
	\[
		\tilde{f}(x + \ii y) = (f(x) + \ii y f'(x)) \chi(y).
	\]
Then by \eqref{eq:avg_local_law} (one can extend the estimate $\abs{s(z)-m(z)} \leq N^\e (N\eta)^{-1}$ to
include $0 < \eta < N^{-1+\e}$ using that $y \im s(E+\ii y)$ is monotonely increasing in $y$), for any $\e > 0$, we have 
	\begin{multline}
	\label{eq:HS}
		\Tr f(w) - \E( \Tr f(W) ) 
		=
		\frac{N}{2\pi} \int_{\Omega_{\mathfrak{a}}}
			\pa{
				\ii y \chi(y) f''(x) + \ii( f(x) + \ii f'(x) y) \chi'(y)
			}
			\pa{
				m_N(z) - \E [ m_N(z) ]
			}
		\rd x \rd y \\
		+ 
		\OO(N^{-1+\mathfrak{a} + \e} \norm{f''}_1 + N^\e \norm{f'}_1 + N^\e \norm{f}_1),
	\end{multline}	
with overwhelming probability. Therefore, defining
	\begin{equation}
	\label{eq:def_e_a}
		e_{\mathfrak{a}}(\lambda) 
		=
		\exp\pa{
			\ii \lambda \frac{N}{2\pi} \int_{\Omega_{\mathfrak{a}}}
			\pa{
				\ii y \chi(y) f''(x) + \ii( f(x) + \ii f'(x) y) \chi'(y)
			}
			\pa{
				m_N(z) - \E [ m_N(z) ]
			}
			\diff x \diff y
		},
	\end{equation}
we have
	\begin{equation}
	\label{eq:ea-a}
		\abs{ e_{\mathfrak{a}}(\lambda) - e(\lambda) }
		\leq
		C \abs{\lambda} (N^{-1+\mathfrak{a} + \e} \norm{f''}_1 + N^\e \norm{f'}_1 + N^\e \norm{f}_1),
	\end{equation}
which will allow us to estimate $\E(e(\lambda))$ from $\psi_\mathfrak{a}(\lambda) = \E(e_\mathfrak{a}(\lambda))$. 
To compute $\psi_\mathfrak{a}(\lambda)$, we begin by differentiatiating \eqref{eq:def_e_a} to find
	\begin{equation}
	\label{eq:HS_dpsi}
		\frac{\rd}{\rd \lambda} \psi_{\mathfrak{a}}(\lambda) 
		= 
		\frac{\ii}{2\pi} \int_{\Omega_{\mathfrak{a}}} (\ii y \chi(y) f''(x) + \ii y (f(x) + f'(x)) \chi'(y)) E_{\mathfrak{a}}(z) \rd x \rd y,
	\end{equation}
where $E_{\mathfrak{a}}(z) = \sum_{i=1}^N \E( e_\mathfrak{a} (  G_{ii}(z) - \E(G_{ii}(z)) )  )$. Our first goal is now to estimate 
$E_{\mathfrak{a}}(z)$,
as in Proposition \ref{prop:tilde_vi}. Note that in the remainder of this section, $\e>0$ is an arbitrary small constant whose
exactly value may change from line to line.

\begin{proposition}
\label{prop:tilde_vi}
	Let $v_i = \E( e_\mathfrak{a} (  G_{ii}(z) - \E(G_{ii}(z)) )  )$. Then for any $\e > 0$, with overwhelming probability, we have
	\begin{multline}
	\label{eq:main_formula}
		\tilde{v}_i = [(1- \mathfrak{m}(z)^2 S) v]_i = \psi_{\mathfrak{a}}(\lambda) \biggr(
			- \frac{\ii \lambda}{\pi} \int_{\Omega_\mathfrak{a}} ( \bar{\partial}_w \tilde{f}(w) ) [S\mathfrak{m}'(w) \mathfrak{m}^2(z)]_{ii} \rd w
		\\
		+ 
		\frac{2\ii \lambda}{\pi} \int_{\Omega_\mathfrak{a}} ( \bar{\partial}_w \tilde{f}(w) ) \partial_w [
			\mathfrak{m}(z) (1 - S\mathfrak{m}(z)\mathfrak{m}(w))^{-1} S \mathfrak{m}(z) \mathfrak{m}(w)
		]_{ii} \rd w \\
		+
		\frac{\ii \lambda}{\pi N^2} \sum_{a} s_{ia}^{(4)} \mathfrak{m}_a(z) \mathfrak{m}_i^2(z) \int_{\Omega_\mathfrak{a}} ( \bar{\partial}_w \tilde{f}(w) )
			\partial_w (  \mathfrak{m}_a(w) \mathfrak{m}_i(w)  ) \rd w
		\biggr) \\
		\hspace*{-1cm}
		+\OO\pa{ N^\e (1 + \abs{\lambda})^4 \pa{
					\Psi(z)(N\eta)^{-1} + N^{-1}\Psi(z)(1 + \norm{f''}_1)^{1/2} + N^{-3/2}(1 + \norm{f''}_1) \pa{N^{1-\mathfrak{a}}}^{3/2} + N^{-3/2} \eta^{-5/2}
				} 
			}.
	\end{multline}
\end{proposition}

Before coming to the proof of Proposition \ref{prop:tilde_vi}, we prove Lemmas \ref{lem:HS_bound} and \ref{lem:e_a_bounds},
which we will use repeatedly to bound error terms in the remainder of our argument.
Our proof of Lemma \ref{lem:HS_bound} follows that of \cite[Lemma 4.4]{LanSos2020}, and includes the effect of $\norm{f'}_1$
which may depend on $N$.

\begin{lemma}
\label{lem:HS_bound}
	Let $H$ be a function holomorphic on $\C \backslash \R$, and assume that for some constant $K > 0$
	and $1 \leq s \leq 2$,
	\[
		\abs{H(z)} \leq \frac{K}{\abs{\im z}^s}.
	\]
	Suppose further that $\lim_{\abs{x} \to \infty} f'(x) = 0$.
	Then there is a constant $C > 0$ such that
	\[
		\absa{
			\int_{\Omega_{\mathfrak{a}}}
			\pa{
				\ii y \chi(y) f''(x) + \ii( f(x) + \ii f'(x) y) \chi'(y)
			} H(x + \ii y) \diff x \diff y 
		}
		\leq
		C K \log(N) (1 + \norm{f}_1 + \norm{f'}_1) (1 + \norm{f''}_1)^{s-1}.
	\]
\end{lemma}

\begin{proof}
	The term with integrand proportional to $\chi'(y)$ is bounded by $C(\norm{f}_1 + \norm{f'}_1)$ because $\chi'(y) \neq 0$
	only for $1 \leq y \leq 2$. We therefore aim to show
	\[
		\absa{
			\int_{\Omega_\mathfrak{a}} \ii y \chi(y) f''(x) H(x + \ii y) \rd x \rd y
		}
		\leq
		C K \log(N) (1 + \norm{f'}_1) (1 + \norm{f''}_1)^{s-1}.
	\]
	For any $\eta_1 > 0$, integrating by parts in $x$ and using the Cauchy-Riemann equations, we have
	\[
		\absa{
			\int_{\abs{y} > \eta_1} \ii y \chi(y) f''(x) H(x + \ii y) \rd x \rd y
		}
		=
		\absa{
			\int_{\abs{y} > \eta_1} \ii y \chi(y) f'(x) \partial_z H(x + \ii y) \rd x \rd y
		}.
	\]
	Furthermore, by the Cauchy integral formula, for some $C > 0$, we have
	\[
		\abs{\partial_z H(x+\ii y)} \leq \frac{C}{\abs{\im z}^{s+1}}.
	\]
	Therefore
	\begin{align*}
		\absa{
			\int_{\abs{y} > \Omega_\mathfrak{a}} \ii y \chi(y) f''(x) H(z) \rd x \rd y
		}
		&\leq
		\absa{
			\int_{\abs{y} > \eta_1} \ii y \chi(y) f''(x) H(z) \rd x \rd y
		}
		+
		\absa{
			\int_{\Omega_\mathfrak{a} \cap \{\abs{y} \leq \eta_1\}} \ii y \chi(y) f''(x) H(z) \rd x \rd y
		} \\
		&\leq
		CK \log(N) ( \norm{f'}_1 (\eta_1)^{1-s}  + \norm{f''}_1 (\eta_1)^{2-s}).
	\end{align*}
	Choose $\eta_1 = \norm{f''}_1^{-1}$ to complete the proof. 
\end{proof}
In the following, we will use the notation $\partial_{ia} f(\cdot) = \frac{\rd}{\rd W_{ia}} f(\cdot)$.
Note that since, $\frac{\partial}{\partial W_{k\ell}} G_{ij} = - G_{ik}G_{\ell j}$,
we have
	\begin{equation}
	\label{eq:resolvent_derivatives}
		\partial_{k\ell} G_{ij} = -\frac{1}{1 + \delta_{k\ell}} \pa{
			G_{ik}G_{\ell j} + G_{i\ell}G_{kj}
		}.
	\end{equation}
Our proof of Proposition \ref{prop:tilde_vi} will make repeated use of estimates on the derivatives of $e_\mathfrak{a}(\lambda)$ 
with respect to the matrix entries $W_{ia}$, so we state Lemma \ref{lem:e_a_bounds} for future reference.
\begin{lemma}
\label{lem:e_a_bounds}
	For any $i,a \in \llbracket 1, N \rrbracket$, $\e > 0$, and $f$ such that $\norm{f''}_1 \leq N^{1-2\e}$, we have
	\begin{align}
	\label{eq:e_a_1_derivative}
		\partial_{ia} e_\mathfrak{a}(\lambda) &= - \frac{2 \ii \lambda e_\mathfrak{a}(\lambda) }{\pi (1 + \delta_{ia})} 
			\int_{\Omega_\mathfrak{a}} \bar{\partial}_z \tilde{f}(z) \partial_z G_{ia}, 
			\\
		\partial_{ia}^2 e_\mathfrak{a}(\lambda) 
		&=   \frac{2\ii \lambda e_\mathfrak{a}(\lambda)}{\pi (1 + \delta_{ia})^2} \int_{\Omega_\mathfrak{a}} \bar{\partial}_z \tilde{f}(z) \partial_z(G_{aa}G_{ii} + G_{ia}^2)
			+ e_\mathfrak{a}(\lambda) \pa{
			\frac{2\ii \lambda}{\pi ( 1+\delta_{ia})} \int_{\Omega_a} \bar{\partial}_z \tilde{f}(z) \partial_z G_{ia}
			}^2 \nonumber \\
		&= \frac{2\ii \lambda e_\mathfrak{a}(\lambda) }{\pi (1 + \delta_{ia})^2} \int_{\Omega_\mathfrak{a}} \bar{\partial}_w \tilde{f}(w) \partial_w(\mathfrak{m}_a(w) \mathfrak{m}_i(w))
			+ \OO(
				(1 + \abs{\lambda})^2 N^{-1/2+\e} (1 + \norm{f''}_1)^{1/2}
			),
	\label{eq:e_a_2_derivative} 
	\end{align}
	with overwhelming probability. Furthermore, for any $k \in \N$ and $\e > 0$, 
	\begin{align}
		\label{eq:e_a_kth_derivative}
		\abs{\partial_{ia}^k e_\mathfrak{a}(\lambda)} \leq N^\e(1 + \abs{\lambda)}^k,
	\end{align}
	with overwhelming probability.
\end{lemma}
\begin{proof}
	The formula \eqref{eq:e_a_1_derivative} follows from \eqref{eq:resolvent_derivatives}
	and the identity $G^2(z) = \partial_z G(z)$. 
	The first equality in \eqref{eq:e_a_2_derivative} follows in the same way, and the second follows from 
	\eqref{eq:local_law_entrywise} and Lemma \ref{lem:HS_bound}. A similar calculation yields \eqref{eq:e_a_kth_derivative}.
	For more detail, see \cite[Section 4]{LanSos2020}.
\end{proof}

We now prove Proposition \ref{prop:tilde_vi} in two steps, which we capture in Lemmas \ref{lem:step1} and \ref{lem:step2}.
Our proof of Lemma \ref{lem:step1} is based on the following cumulant expansion, which we quote from \cite[Lemma 3.2]{LeeSch2018}.

\begin{lemma}
\label{lem:cumulant_expansion}
	Fix $\ell \in \N$ and let $F \in C^{\ell + 1}(\R, \C^+)$. Let $Y$ be a centered random variable with finite moments to order 
	$\ell + 2$. Then
	\[
		\E(Y F(Y)) = \sum_{r = 1}^\ell \frac{\kappa^{(r+1)} (Y)}{r!} \E[F^{(r)}(Y)] + \E[\Omega_\ell(YF(Y))],
	\]
	where $\E$ denotes the expectation with respect to $Y$, $\kappa^{(r+1)}$ denotes the $(r+1)$-th cumulant of $Y$
	and $F^{(r)}$ denotes the $r$-th derivative of the function $F$. The error term $\Omega_\ell(YF(Y))$ satisfies
	\[
		\abs{\E[\Omega_\ell(YF(Y))]} \leq C_\ell \E(\abs{Y}^{\ell + 2}) \sup_{\abs{t} \leq Q} \abs{F^{(\ell + 1)}(t)}
			+ C_{\ell} \E( \abs{Y}^{\ell + 2} \1(\abs{Y} > Q)) \sup_{t \in \R} \abs{ F^{(\ell + 1)} (t)    } ,
	\]
	where $Q \geq 0$ is an arbitrary fixed cutoff and $C_\ell$ satisfies $C_\ell \leq (C\ell)^\ell / \ell !$ for some constant $C$.
\end{lemma}

In the following, we write $s_{ij}^{(k)}$ for the $k$-th cumulant of $H_{ij}$, and we 
define $S$ by $s_{ij} = s_{ij}^{(2)} = N \sigma^2_{ij}$. Note that $s_{ij}^{(k)} \sim 1$ for all $k > 1$. 

\begin{lemma}
\label{lem:step1}
For any $\e > 0$ and $f$ with $\norm{f}_1 \leq N^\e$, $\norm{f'}_1 \leq N^\e$ and $\norm{f''} \leq N^{1 - 10\e}$, 
	\begin{multline}
	\label{eq:step_1}
		- \frac{1}{\mathfrak{m}_i(z)} [(1- \mathfrak{m}^2(z) S)v]_i 
		=
		\frac{1}{N} \sum_{a} s_{ia} \E\pa{
			G_{ia} \partial_{ia} e_{\mathfrak{a}}
		} 
		-
		\E[
			e_\mathfrak{a}(T_{ii}(z,z) - \E(T_{ii}(z,z))
		] 
		- 
		\frac{1}{2N^2} \sum_{a} s_{ia}^{(4)} \mathfrak{m}_a \mathfrak{m}_i \E[ \partial_{ia}^2 e_\mathfrak{a}] \\
		+ \OO \pa{
			N^\e \Psi(z) (N\eta)^{-1}+ N^{-1+\e}\Psi(z)(1 + \abs{\lambda})^3(1 + \norm{f''}_1)^{1/2} + (1 + \abs{\lambda})^4 N^{-3/2+\e}
		}.
	\end{multline}
\end{lemma}

\begin{proof}[Proof of Lemma \ref{lem:step1}]
By Lemma \ref{lem:cumulant_expansion}, using $z G_{ii}(z) = \sum_{a} G_{ia}W_{ai}$, \eqref{eq:resolvent_derivatives}, and Lemma \ref{lem:e_a_bounds}
we have
	\begin{multline}
	\label{eq:cumulant_expansion}
		z \E\pa{ e_{\mathfrak{a}} \pa{ G_{ii}(z) - \E(G_{ii}(z))  }   } = \sum_{a} \E\pa{  e_{\mathfrak{a}}(G_{ia}W_{ai} - \E( G_{ia}W_{ai} ) )    } \\
		=
		\frac{1}{N} \sum_{a} s_{ia} \E\pa{ G_{ia}  \partial_{ia} e_{\mathfrak{a}}  }
		- \frac{1}{N} \sum_{a} s_{ia} \E\pa{ e_{\mathfrak{a}}\pa{G_{ii}G_{aa} - \E(G_{ii} G_{aa})}   }
		-\frac{1}{N} \sum_{a} s_{ia} \E\pa{ e_{\mathfrak{a}} \pa{G_{ia}^2 - \E\pa{G_{ia}^2} }} \\
		+ \frac{1}{N} s_{ii} \E\pa{ e_{\mathfrak{a}} \pa{  G_{ii}^2 - \E\pa{G_{ii}^2}  }  }
		+ \frac{1}{2N^{3/2}} \sum_{a} s_{ia}^{(3)} \pa{
			\E[ (\partial_{ia}^2 e_{\mathfrak{a}}) G_{ai} ] + 2 \E(  \partial_{ai} e_{\mathfrak{a}} \partial_{ai} G_{ai}  )
				+ \E[  e_{\mathfrak{a}}(  \partial_{ai}^2 G_{ai} - \E[\partial_{ai}^2 G_{ai}]       )      ]
		} \\
		+ \frac{1}{6N^2} \sum_{a} s_{ai}^{(4)} \pa{
			\E[G_{ai}\partial_{ai}^3 e_{\mathfrak{a}}] + 3 \E[ \partial_{ai} e_\mathfrak{a} \partial_{ai}^2G_{ai}] + 3\E[ \partial_{ai}^2 e_\mathfrak{a} \partial_{ai}G_{ai}]
			+ \E[  e_{\mathfrak{a}}\pa{ \partial_{ai}^3 G_{ai} - \E[\partial_{ai}^3 G_{ai}]      }   ] 
		} \\
		+ \OO\pa{ N^{-3/2+\e}(1 + \abs{\lambda})^4 }.	
	\end{multline}
We leave the first term on the right hand side of \eqref{eq:cumulant_expansion} as is.
For the second term, write $G_{ii} = G_{ii} - \E(G_{ii}) + \E(G_{ii})$ and $G_{aa} = G_{aa} - \E(G_{aa}) + \E(G_{aa})$ to get
	\begin{multline}
	\label{eq:pm_EGii}
		- \frac{1}{N} \sum_{a} s_{ia} \E\pa{ e_{\mathfrak{a}}\pa{G_{ii}G_{aa} - \E(G_{ii} G_{aa})}   } \\
		=
		- \frac{1}{N} \sum_{a} s_{ia} \E(e_\mathfrak{a} (G_{ii} - \E(G_{ii})) \E(G_{aa})
		- \frac{1}{N} \sum_{a} s_{ia} \E(e_\mathfrak{a} (G_{aa} - \E(G_{aa})) \E(G_{ii}) \\
		- \frac{1}{N} \sum_{a} s_{ia} \E[
			e_\mathfrak{a} (
				(G_{ii} - \E(G_{ii}))(G_{aa} - \E(G_{aa}) - \E((G_{ii} - \E(G_{ii}))(G_{aa} - \E(G_{aa}) )
			)
		].
	\end{multline}
By \eqref{eq:avg_local_law} and the quadratic vector equation \eqref{eq:qve}, with overwhelming probability
we have 
	\begin{equation}
	\label{eq:use_qve}
		-\frac{1}{N} \sum_{a} s_{ia} G_{aa} 
		= 
		-\frac{1}{N} \sum_a s_{ia} \mathfrak{m}_a + \OO\pa{N^\e (N\eta)^{-1}}
		=
		(z + \mathfrak{m}_i^{-1}) + \OO\pa{N^\e (N\eta)^{-1}},
	\end{equation}
and using \eqref{eq:local_law_entrywise} to estimate $\E(G_{ii})$ and \eqref{eq:avg_local_law} to estimate $\sum_{a} s_{ia} (G_{aa} - \E(G_{aa}))$, we have
	\begin{equation}
	\label{eq:use_local_law}
		-\frac{1}{N} \sum_{a} s_{ia} \E(e_\mathfrak{a}(G_{aa} - \E(G_{aa}))) \E(G_{ii})
		= 
		-\frac{1}{N} \sum_{a} s_{ia} \E(e_\mathfrak{a}(G_{aa} - \E(G_{aa}))) \mathfrak{m}_i 
		+ \OO\pa{ N^\e\Psi(z)(N\eta)^{-1}}.
	\end{equation}
Combining \eqref{eq:use_qve} and \eqref{eq:use_local_law} and
applying \eqref{eq:local_law_entrywise} again to bound $\E(e_\mathfrak{a}(G_{ii} - \E(G_{ii}))$ and the last term on the right hand side of \eqref{eq:pm_EGii}, we have
	\begin{multline*}
		- \frac{1}{N} \sum_{a} s_{ia} \E\pa{ e_{\mathfrak{a}}\pa{G_{ii}G_{aa} - \E(G_{ii} G_{aa})}   } \\
		=
		\E[ e_{\mathfrak{a}}(G_{ii} - \E(G_{ii}))  ](z + \mathfrak{m}_i^{-1})
		- \frac{1}{N} \sum_{a} \mathfrak{m}_i s_{ia} \E[e_{\mathfrak{a}}(G_{aa}- \E(G_{aa})) ]
		+ \OO\pa{ N^\e \Psi(z) (N\eta)^{-1} }.
	\end{multline*}
Substituting this into \eqref{eq:cumulant_expansion} and rearranging, we have so far
	\begin{multline}
	\label{eq:cumulant_exp_after_rearranging}
		-\frac{1}{\mathfrak{m}_i} [(1- \mathfrak{m}^2S)v]_i 
		=
		\frac{1}{N} \sum_{a} s_{ia} \E\pa{ G_{ia} \partial_{ia} e_{\mathfrak{a}}   }
		-\frac{1}{N} \sum_{a} s_{ia} \E\pa{ e_{\mathfrak{a}} \pa{G_{ia}^2 - \E\pa{G_{ia}^2} }} \\
		+ \frac{1}{N} s_{ii} \E\pa{ e_{\mathfrak{a}} \pa{  G_{ii}^2 - \E\pa{G_{ii}^2}  }  }
		+ \frac{1}{2N^{3/2}} \sum_{a} s_{ia}^{(3)} \pa{
			\E[ (\partial_{ia}^2 e_{\mathfrak{a}}) G_{ai} ] + 2 \E(  \partial_{ai} e_{\mathfrak{a}} \partial_{ai} G_{ai}  )
				+ \E[  e_{\mathfrak{a}}(  \partial_{ai}^2 G_{ai} - \E[\partial_{ai}^2 G_{ai}]       )      ]
		} \\
		+ \frac{1}{6N^2} \sum_{a} s_{ia}^{(4)} \pa{
			\E[G_{ai}\partial_{ai}^3 e_{\mathfrak{a}}] + 3\E[ \partial_{ai} e_\mathfrak{a} \partial_{ai}^2G_{ai}] + 3\E[ \partial_{ai}^2 e_\mathfrak{a} \partial_{ai}G_{ai}]
			+ \E[  e_{\mathfrak{a}}\pa{ \partial_{ai}^3 G_{ai} - \E[\partial_{ai}^3 G_{ai}]      }   ] 
		} \\
		+ \OO\pa{N^{-\e/2}(1 + \abs{\lambda})^4 + N^\e \Psi(z) (N\eta)^{-1}}.
	\end{multline}
Next, by definition we have
	\[
		\frac{1}{N} \sum_{a} s_{ia} \E\pa{
			e_{\mathfrak{a}} \pa{
				G_{ia}^2 - \E(G_{ia}^2)
			}
		}
		=
		\E\pa{ e_\mathfrak{a} (T_{ii}(z,z) - \E(T_{ii}(z,z)))},
	\]
and by \eqref{eq:local_law_entrywise},
	\[
		\frac{1}{N} s_{ii} \E\pa{ e_{\mathfrak{a}} \pa{  G_{ii}^2 - \E\pa{G_{ii}^2}  }  }
		=
		\OO\pa{   N^{-1+\e} \Psi(z)   }.
	\]
We next bound the term proportional to $N^{-3/2}$ in \eqref{eq:cumulant_exp_after_rearranging}. 
By Lemma \ref{lem:e_a_bounds}, we have
	\begin{multline}
	\label{eq:N^3/2_pt1_sub_dia^2}
		\frac{1}{2N^{3/2}} \sum_a s_{ia}^{(3)} \E[G_{ai} \partial_{ia}e_{\mathfrak{a}}]
		=
		\frac{1}{N^{1/2}} \sum_{a \neq i} s_{ia}^{(3)} \partial_w(\mathfrak{m}_a(w) \mathfrak{m}_i(w)) \E[e_\mathfrak{a}(\lambda) G_{ai}(z)] \\
		+ \OO\pa{
			N^{-1+\e} \Psi(z) (1 + \abs{\lambda})^2 (1 + \norm{f''}_1)^{1/2} 
		}.
	\end{multline}
And using \eqref{eq:isotropic_local_law} and by assumption (C) that $\partial_w \mathfrak{m}_a(w) \leq \kappa(w)^{-1/2}$, we have
	\begin{equation}
	\label{eq:N^3/2_pt1}
		\frac{1}{N^{1/2}} \sum_{a \neq i} s_{ia}^{(3)} \partial_w(\mathfrak{m}_a(w) \mathfrak{m}_i(w)) \E[e_\mathfrak{a}(\lambda) G_{ai}(z)]
		=
		\OO\pa{
			N^{\e}\Psi(z)  \im(w)^{-1/2} 
		}.
	\end{equation}
Therefore, Lemma \ref{lem:HS_bound} with $s = 1$ gives
	\begin{equation}
	\label{eq:N^3/2_pt1}
		\frac{1}{N^{3/2}} \sum_{a \neq i} s_{ia}^{(3)} \frac{2\ii \lambda}{\pi} \int_{\Omega_\mathfrak{a}} \bar{\partial}_w \tilde{f}(w)
			\partial_w(\mathfrak{m}_a(w) \mathfrak{m}_i(w)) \E[e_\mathfrak{a}(\lambda) G_{ai}(z)] \rd w
		=
		\OO\pa{
			N^{-1+\e} \Psi(z)
		},
	\end{equation}
and combining \eqref{eq:N^3/2_pt1_sub_dia^2} and \eqref{eq:N^3/2_pt1}, we have in summary
	\[
		\frac{1}{2N^{3/2}} \sum_a s_{ia}^{(3)} \E[G_{ai}\partial_{ia}^2 e_\mathfrak{a}]
		=
		\OO\pa{
			N^{-1+\e} \Psi(z) (1 + \abs{\lambda})^2 (1 + \norm{f''}_1)^{1/2} 
		}.
	\]
To bound the next term proportional to $N^{-3/2}$, use \eqref{eq:resolvent_derivatives},
\eqref{eq:local_law_entrywise}, and \eqref{eq:e_a_kth_derivative} to write	
	\begin{align*}
		\frac{1}{N^{3/2}}\sum_{a=1}^N s_{ia}^{(3)} \E[(\partial_{ai} e_\mathfrak{a}) (\partial_{ai} G_{ai})]
		&=
		-\frac{1}{N^{3/2}}\sum_{a =1}^N s_{ia}^{(3)} \E[(\partial_{ai} e_\mathfrak{a}) (G_{ii}G_{aa} + G_{ia}^2)] \\
		&=
		-\frac{1}{N^{3/2}}\sum_{a \neq i}^N s_{ia}^{(3)} \E[(\partial_{ai} e_\mathfrak{a}) (G_{ii}G_{aa} + G_{ia}^2)] 
			+ \OO\pa{ N^{-2+\e} (1 + \abs{\lambda}) (1 + \norm{f''}_1)^{1/2} } \\
		&=
		\begin{aligned}[t]
		-\frac{1}{N^{3/2}} &\sum_{a\neq i}^N s_{ia}^{(3)} \mathfrak{m}_i(z) \mathfrak{m}_a(z) \E[(\partial_{ai} e_\mathfrak{a})] \\
			 &+ \OO\pa{ N^{-1/2+\e} \Psi^2(z) + N^{-1+\e} \Psi(z) (1 + \abs{\lambda}) (1 + \norm{f''}_1)^{1/2} }.
		\end{aligned}
	\end{align*}
Now substituting \eqref{eq:e_a_1_derivative}, we have 
	\[
		-\frac{1}{N^{3/2}} \sum_{a\neq i} s_{ia}^{(3)} \mathfrak{m}_i(z) \mathfrak{m}_a(z) \E[(\partial_{ai} e_\mathfrak{a})]
		=
		\frac{2 \ii \lambda}{\pi N^{3/2}} 
			\E\pa{
				e_\mathfrak{a} \int_{\Omega_\mathfrak{a}} \bar{\partial}_w \tilde{f}(w) \partial_w \pa{
				\sum_{a \neq i} s_{ia}^{(3)} \mathfrak{m}_i(z) \mathfrak{m}_a(z) G_{ia}(w)
				}
			},
	\]
and \eqref{eq:isotropic_local_law} 
along with and the Cauchy integral formula to bound the derivative of the error term give
	\[
		\partial_w \frac{1}{\sqrt{N}}  \sum_{a \neq i} s_{ia}^{(3)} \mathfrak{m}_a(z) G_{ia}(w) 
		=
		\OO\pa{ N^\e (\im w)^{-1} \Psi(w) }.
	\]
From Lemma \ref{lem:HS_bound}, it follows that
	\[
		-\frac{1}{N^{3/2}} \sum_{a\neq i} s_{ia}^{(3)} \mathfrak{m}_i(z) \mathfrak{m}_a(z) \E[(\partial_{ai} e_\mathfrak{a})]
		=
		\OO\pa{ N^{-3/2+\e}(1 + \abs{\lambda})(1 + \norm{f''}_1)^{1/2} }.
	\]
To bound the last term proportional to $N^{-3/2}$, we use \eqref{eq:resolvent_derivatives} and \eqref{eq:local_law_entrywise},
which give
	\[
		\partial_{ai}^2 G_{ai} = 6 (1 + \delta_{ai})^{-2} G_{ia}\mathfrak{m}_i \mathfrak{m}_a + \OO\pa{N^\e \Psi(z)^2},
	\]
and using \eqref{eq:isotropic_local_law}, we conclude
	\[
		\frac{1}{N^{3/2}} \sum_a s_{ia}^{(3)} \E[
			e_\mathfrak{a} (\partial_{ai}^2 G_{ia} - \E[\partial_{ai}^2 G_{ai}]))
		]
		=
		\OO\pa{ N^{-1/2 + \e} \Psi(z)^2 }.
	\]
Finally, by \eqref{eq:e_a_kth_derivative} and \eqref{eq:resolvent_derivatives}, we have
	\begin{align*}
		&\frac{1}{6N^2} \sum_{a} s_{ia}^{(4)} \pa{
			\E[G_{ai}\partial_{ai}^3 e_{\mathfrak{a}}] + 3\E[ \partial_{ai} e_\mathfrak{a} \partial_{ai}^2G_{ai} ]+ 3 \E[ \partial_{ai}^2 e_\mathfrak{a} \partial_{ai}G_{ai}]
			+ \E[  e_{\mathfrak{a}}\pa{ \partial_{ai}^3 G_{ai} - \E[\partial_{ai}^3 G_{ai}]      }   ] 
		} \\
		=
		& \frac{1}{6N^2} \sum_{a\neq i} s_{ia}^{(4)} \pa{
			\E[G_{ai}\partial_{ai}^3 e_{\mathfrak{a}}] + \dots
			+ \E[  e_{\mathfrak{a}}\pa{ \partial_{ai}^3 G_{ai} - \E[\partial_{ai}^3 G_{ai}]      }   ] 
		}
		+ \OO\pa{ N^{-2+\e} (1 + \abs{\lambda})^{3} }.
	\end{align*}
By \eqref{eq:local_law_entrywise} and \eqref{eq:e_a_kth_derivative} , and since for $a \neq i$
	\[
		\partial_{ai}^3 G_{ai} - \E\pa{\partial_{ai}^3 G_{ai} } = \OO\pa{N^\e \Psi(z)}
	\]
with overwhelming probability, we conclude
	\begin{multline*}
		\frac{1}{6N^2} \sum_{a} s_{ia}^{(4)} \pa{
			\E[G_{ai}\partial_{ai}^3 e_{\mathfrak{a}}] + 3 \partial_{ai} e_\mathfrak{a} \partial_{ai}^2G_{ai} + 3 \partial_{ai}^2 e_\mathfrak{a} \partial_{ai}G_{ai}
			+ \E[  e_{\mathfrak{a}}\pa{ \partial_{ai}^3 G_{ai} - \E[\partial_{ai}^3 G_{ai}]      }   ] 
		} \\
		=
		- \frac{1}{2N^2} \sum_{a} s_{ia}^{(4)} \mathfrak{m}_a \mathfrak{m}_i \E[\partial_{ia}^2 e_\mathfrak{a}] + \OO\pa{N^{-1+\e} \Psi(z) (1 + \abs{\lambda})^2}.
	\end{multline*}
This establishes \eqref{eq:step_1}.
\end{proof}

We proceed with Lemma \ref{lem:step2} in which we further re-write the right hand side of \eqref{eq:step_1}.

\begin{lemma}
\label{lem:step2}
	We have
	\begin{multline}
	\label{eq:pt1}
		\frac{1}{N} \sum_a s_{ia} \E[G_{ia} \partial_{ia} e_\mathfrak{a}]
		=
		\frac{\ii \lambda}{\pi} \frac{s_{ii}}{N} \int_{\Omega_\mathfrak{a}} 
			(\bar{\partial}_w \tilde{f}(w)) \mathfrak{m}_i'(w) \mathfrak{m}_i(z) \psi_\mathfrak{a} \rd w \\
			- \frac{2\ii \lambda}{\pi} \int_{\Omega_\mathfrak{a}} (\bar{\partial}_w \tilde{f}(w)) \psi_\mathfrak{a}
				\partial_w[(1-S\mathfrak{m}(z) \mathfrak{m}(w))^{-1}S\mathfrak{m}(z)\mathfrak{m}(w)]_{ii} \rd w \\
				+ \OO\pa{
					(1+\abs{\lambda}) N^{-3/2+\e} (1 + \norm{f''}_1) (N^{1-\mathfrak{a}})^{3/2} 
					+(1 + \abs{\lambda}) N^{-1+\e} \Psi(z)(1 + \norm{f''}_1)^{1/2}
				},
	\end{multline}
	and
	\begin{multline}
	\label{eq:pt3}
		-\frac{1}{2N^2} \sum_a s_{ia}^{(4)} \mathfrak{m}_a \mathfrak{m}_i \E[ \partial_{ia}^2 e_\mathfrak{a}] \\
		=
		-\frac{\ii \lambda}{\pi N^2} \sum_a s_{ia}^{(4)} \mathfrak{m}_a(z) \mathfrak{m}_i(z) 
			\int_{\Omega_\mathfrak{a}} (\bar{\partial}_w \tilde{f}(w)) \partial_w
				(
					\mathfrak{m}_a(w) \mathfrak{m}_i(w)
				)
			\psi_\mathfrak{a} \rd w
		+
		\OO\pa{
			(1 + \abs{\lambda}) N^{-1/2+\e} (1 + \norm{f''}_1)^{1/2}
		}.
	\end{multline}
	Furthermore, we have the bound 
	\begin{equation}
	\label{eq:pt2}
		\E[e_\mathfrak{a}(T_{ii}(z,z) - \E (T_{ii}(z,z))] = \OO\pa{N^{-3/2+\e} \eta^{-5/2}}.
	\end{equation}
\end{lemma}
\begin{proof}
	By \eqref{eq:e_a_1_derivative}, we have
	\begin{align*}
		\frac{1}{N} \sum_a s_{ia} \E[G_{ia} \partial_{ia} e_\mathfrak{a}] 
		=
		- \frac{1}{N} \sum_{a=1}^N s_{ia} \frac{2\ii \lambda}{\pi} \int_{\Omega_\mathfrak{a}} (\bar{\partial}_w \tilde{f}(w)) 
			&\E[ e_\mathfrak{a} G_{ia}(z) \partial_w G_{ia}(w) ] \rd w \\
			+& \frac{\ii \lambda}{\pi} \frac{s_{ii}}{N} \int_{\Omega_\mathfrak{a}} (\bar{\partial}_w \tilde{f}(w)) 
			\E[ e_\mathfrak{a} G_{ii}(z) \partial_w G_{ii}(w) ] \rd w,
	\end{align*}
	and substituting $G_{ii}(z) = \mathfrak{m}_i(z) + \OO(N^\e \Psi(z))$ and 
	$\partial_{w} G_{ii}(w) = \mathfrak{m}'_i(w) + \OO(N^\e \Psi(w) \im(w)^{-1})$ with overwhelming probability, we have
	\begin{multline*}
		\frac{\ii \lambda}{\pi} \frac{s_{ii}}{N} \int_{\Omega_\mathfrak{a}} (\bar{\partial}_w \tilde{f}(w)) 
			\E[ e_\mathfrak{a} G_{ii}(z) \partial_w G_{ii}(w) ] \rd w
		=
		\frac{\ii \lambda}{\pi} \frac{s_{ii}}{N} \int_{\Omega_\mathfrak{a}} (\bar{\partial}_w \tilde{f}(w)) 
			\mathfrak{m}_i'(w) \mathfrak{m}(z) \psi_\mathfrak{a}(\lambda) \rd w
			\\ +
			\OO\pa{
				N^{-1+\e} \Psi(z) (1 + \norm{f''}_1)^{1/2}
			}.
	\end{multline*}
	To complete the proof of \eqref{eq:pt1}, note that
	\[
		\frac{1}{N} \sum_{a=1}^N s_{ia} \frac{2\ii \lambda}{\pi} \int_{\Omega_\mathfrak{a}} (\bar{\partial}_w \tilde{f}(w)) 
			\E[ e_\mathfrak{a} G_{ia}(z) \partial_w G_{ia}(w) ] \rd w
		=
		\frac{2\ii \lambda}{\pi} \int_{\Omega_\mathfrak{a}} (\bar{\partial}_w \tilde{f}(w)) 
			\E[ e_\mathfrak{a} \partial_w T_{ii}(z,w) ] \rd w,
	\]
	and applying Theorem \ref{thm:Txy} with $\im(z), \im(w) \geq N^{-1 + \mathfrak{a}}$, we have
	\[
		\absa{ 
			\partial_w (T_{ii}(z,w) - [(1 - S \mathfrak{m}(z) \mathfrak{m}(w))^{-1}S \mathfrak{m}(z)\mathfrak{m}(w)]_{ii} ) 
		}
		\leq
		\frac{N^{-3/2+\e}}{\im(w)^2} \pa{N^{1-\mathfrak{a}}}^{3/2}.
	\]
	Therefore Lemma \ref{lem:HS_bound} with $s=2$ gives
	\begin{multline*}
		\frac{2\ii \lambda}{\pi} \int_{\Omega_\mathfrak{a}} (\bar{\partial}_w \tilde{f}(w)) 
			\E[ e_\mathfrak{a} \partial_w T_{ii}(z,w) ] \rd w
		=
		\frac{2\ii \lambda}{\pi} \int_{\Omega_\mathfrak{a}} (\bar{\partial}_w \tilde{f}(w)) 
			\psi_\mathfrak{a}(\lambda) \partial_w[(1 - S \mathfrak{m}(z) \mathfrak{m}(w))^{-1}S \mathfrak{m}(z)\mathfrak{m}(w)]_{ii} \rd w \\
		+ \OO\pa{
			(1 + \abs{\lambda}) N^{-3/2+\e} (N^{1-\mathfrak{a}})^{3/2} (1 + \norm{f''}_1)
		},
	\end{multline*}
	which proves \eqref{eq:pt1}.
	The bound \eqref{eq:pt3} follows from \eqref{eq:e_a_2_derivative}, and the bound \eqref{eq:pt2} follows from Theorem \ref{thm:Txy}.
\end{proof}

With Lemmas \ref{lem:step1} and \ref{lem:step2}, the proof of Proposition \ref{prop:tilde_vi} follows easily. 

\begin{proof}[Proof of Proposition \ref{prop:tilde_vi}]
	Substituting the results of Lemma \ref{lem:step2} in \eqref{eq:step_1} yields the result.
\end{proof}

\begin{proposition}
\label{prop:char_function}
	Let $f$ be a function with compact support such that
	for some $\e>0$, $\norm{f}_1, \norm{f'}_1 \leq N^\e$ and for some
	$\delta > 0$ 
	\[
		\norm{f''}_1 \leq N^{1/5 - \delta}.
	\]
	Let $c_1 < \delta$ and choose $\mathfrak{a} > 0$ satisfying $4/5 \leq \mathfrak{a} \leq 4/5 + \delta - c_1$. 
	Then for any $\lambda$ such that
	\begin{equation}
	\label{eq:lambda_good_set}
		\lambda^2 V(f) = \oo(\log N) \text{ and } \abs{\lambda} \leq N^{\frac{\delta}{100}},
	\end{equation}
	for some $c > 0$, we have
	\[
		\E[
			\exp[
				\ii \lambda( \Tr(f) - \E(\Tr(f)) )
			]
		]
		=
		\exp[-\lambda^2 V(f)/2] + \OO(N^{-c\delta} + N^{\e-c_1}),
	\]
	where
	\begin{multline}
	\label{eq:V_hat}
		V(f) = \frac{1}{\pi^2} \int_{\Omega_\mathfrak{a} \times \Omega_\mathfrak{a}} ( \bar{\partial}_w \tilde{f}(w)) 	( \bar{\partial}_z \tilde{f}(z))
		\biggr(
		2 \sum_{i,j} (1 - \mathfrak{m}^2(z) S)^{-1}_{ij} \partial_w [
			\mathfrak{m}(z) (1 - S\mathfrak{m}(z)\mathfrak{m}(w))^{-1} S \mathfrak{m}(z) \mathfrak{m}(w)
		]_{jj} 
		\\
		- 
		\sum_{i, j} (1 - \mathfrak{m}^2(z) S)^{-1}_{ij}  [S\mathfrak{m}'(w) \mathfrak{m}^2(z)]_{jj} \\
		+
		\frac{1}{N^2} \sum_{i,j,a} s_{ja}^{(4)} (1 - \mathfrak{m}^2(z) S)^{-1}_{ij} 
			\mathfrak{m}_a(z) \mathfrak{m}_j^2(z) 
			\partial_w (  \mathfrak{m}_a(w) \mathfrak{m}_j(w)  ) 
		\biggr) \rd w
	\end{multline}
\end{proposition}

\begin{proof}
	Using $v_\ell = [(1 - \mathfrak{m}(z)^2S)^{-1} \tilde{v}]_\ell$, by \eqref{eq:HS_dpsi}, we have
	\[
		\frac{\rd}{\rd \lambda} \psi_{\mathfrak{a}}(\lambda) 
		= 
		\frac{\ii}{2\pi} \int_{\Omega_{\mathfrak{a}}} (\ii y \chi(y) f''(x) + \ii y (f(x) + f'(x)) \chi'(y)) \sum_{i=1}^N [(1 - \mathfrak{m}(z)^2S)^{-1} \tilde{v}]_i \rd x \rd y.
	\]
	Substituting \eqref{eq:main_formula}, and using \eqref{lem:bound_1-m^2S_inv} 
	and Lemma \ref{lem:HS_bound} with $s =1$ (we bound the excess powers of $y^{-1}$ by $N^{1-\mathfrak{a}}$)
	to bound the error term, for a small $c > 0$, we have
	\[
		\frac{\rd}{\rd \lambda} \psi_\mathfrak{a}(\lambda) 
		=
		-\lambda \psi_\mathfrak{a} V(f) + \OO\pa{
			(1 + \abs{\lambda}^4)N^{-1/2+c\e} \pa{1 + \norm{f''}_1} (N^{1-\mathfrak{a}})^{3/2}
		},
	\]
	and applying our assumptions on $\norm{f''}_1$, we have
	\[
		\frac{\rd}{\rd \lambda} \psi_\mathfrak{a}(\lambda) 
		=
		-\lambda \psi_\mathfrak{a} V(f) + \OO\pa{
			(1 + \abs{\lambda}^4)N^{c\e - \delta}
		}.
	\]
	Let $g(\lambda) = e^{\lambda^2 V(f)} \psi_\mathfrak{a}^2(\lambda)$.
	Then $g'(\lambda) = e^{\lambda^2 V(f)} \psi_\mathfrak{a}^2(\lambda) \OO((1 + \abs{\lambda}^4)N^{c\e - \delta})$,
	and so $g(\lambda) = 1 + e^{\lambda^2 V(f)} \psi_\mathfrak{a}(\lambda) \OO(N^{c(\e - \delta)})$. This gives us
	\[
		\psi^2_\mathfrak{a}(\lambda) = e^{-\lambda^2 V(f)} \pa{
			1 + \OO\pa{ e^{\lambda^2 V(f)} N^{c(\e - \delta)}}
		},
	\]
	and on the set \eqref{eq:lambda_good_set}, using continuity in $\lambda$, we can take the square root of both sides.
	Finally, we conclude via \eqref{eq:ea-a}.
\end{proof}

 \subsection{Calculation of the Variance. }
 \label{sec:variance}

We now want to calculate
	\begin{multline*}
		V(f) = \frac{1}{\pi^2} \int_{\Omega_\mathfrak{a} \times \Omega_\mathfrak{a}} ( \bar{\partial}_w \tilde{f}(w)) 	( \bar{\partial}_z \tilde{f}(z))
		\biggr(
		2 \sum_{i,j} (1 - \mathfrak{m}^2(z) S)^{-1}_{ij} \partial_w [
			\mathfrak{m}(z) (1 - S\mathfrak{m}(z)\mathfrak{m}(w))^{-1} S \mathfrak{m}(z) \mathfrak{m}(w)
		]_{jj} 
		\\
		- 
		\sum_{i, j} (1 - \mathfrak{m}^2(z) S)^{-1}_{ij}  [S\mathfrak{m}'(w) \mathfrak{m}^2(z)]_{jj} \\
		+
		\frac{1}{N^2} \sum_{i,j,a} s_{ja}^{(4)} (1 - \mathfrak{m}^2(z) S)^{-1}_{ij} 
			\mathfrak{m}_a(z) \mathfrak{m}_j^2(z) 
			\partial_w (  \mathfrak{m}_a(w) \mathfrak{m}_j(w)  ) 
		\biggr) \rd w.
	\end{multline*}

First, we have the following algebraic fact, which we quote from \cite[Proposition 5.12]{LanLopSos2021}.

\begin{proposition}
\label{prop:variance_after_algebra}
	\begin{multline*}
		V(f) = \frac{1}{\pi^2} \int_{\Omega_\mathfrak{a} \times \Omega_\mathfrak{a}} ( \bar{\partial}_w \tilde{f}(w)) 	( \bar{\partial}_z \tilde{f}(z))
		\biggr(
		2 g(z,w) - \Tr( \mathfrak{m}'(z) S \mathfrak{m}'(w) )
		+
		\frac{1}{N^2} \sum_{j,a} s_{ja}^{(4)}  
			\partial_z \partial_w (  \mathfrak{m}_a(z) \mathfrak{m}_j(z) \mathfrak{m}_a(w) \mathfrak{m}_j(w)  ) 
		\biggr) ,
	\end{multline*}
	where
	\begin{multline*}
		g(z,w) = \Tr (
			\mathfrak{m}'(z)\mathfrak{m}(z)^{-1} (1 - S \mathfrak{m}(z) \mathfrak{m}(w))^{-1} S \mathfrak{m}(z) \mathfrak{m}'(w) 
				(1- S \mathfrak{m}(z) \mathfrak{m}(w))^{-1}
		) \\
		=
		\partial_w \Tr(
			\mathfrak{m}'(z)\mathfrak{m}(z)^{-1}(1-S\mathfrak{m}(z)\mathfrak{m}(w))^{-1}
		) 
		=
		\partial_w h(z,w).
	\end{multline*}
	The function $g(z,w)$ is symmetric in $z,w$.
\end{proposition}

We first address the contribution of the terms coming from the terms other than $g(z,w)$. We specialize to the generalized Wigner case,
because we will do so when we study the remaining term. As a reminder, we mention that this means $\mathfrak{m}_i(z) = m_{sc}(z)$ for 
all $i = 1, \dots, N$. 

\begin{lemma}
	In the generalized Wigner case, we have
	\begin{align}
		\nonumber
		V_1(f) &= \frac{1}{\pi^2} \int_{\Omega_\mathfrak{a} \times \Omega_\mathfrak{a}} ( \bar{\partial}_w \tilde{f}(w)) 	( \bar{\partial}_z \tilde{f}(z))
		\pa{
		- \Tr( \mathfrak{m}'(z) S \mathfrak{m}'(w) )
		+
		\frac{1}{N^2} \sum_{j,a} s_{ja}^{(4)}  
			\partial_z \partial_w (  \mathfrak{m}_a(z) \mathfrak{m}_j(z) \mathfrak{m}_a(w) \mathfrak{m}_j(w)  ) 
		} \\
		&=
		\frac{\Tr(S)}{4\pi^2} \pa{
			\int f(x) \frac{x}{\sqrt{4-x^2}} \diff x
		}^2
		+
		\frac{1}{N^2} \sum_{j,a} 
		\frac{s_{ja}^{(4)}}{\pi^2}  \pa{
			\int_{-2}^2 f(x) \frac{2-x^2}{\sqrt{4-x^2}} \diff x
		}^2.
		\label{eq:var_order_1_terms}
	\end{align}

\end{lemma}
\begin{proof}
	First, by \cite[(4.63)]{LanSos2020}, we have
	\begin{align*}
		-\frac{1}{\pi^2} \int_{\Omega_\mathfrak{a} \times \Omega_\mathfrak{a}} ( \bar{\partial}_w \tilde{f}(w)) 	( \bar{\partial}_z \tilde{f}(z))
		\Tr( \mathfrak{m}'(z) S \mathfrak{m}'(w) )
		&=
		-\frac{\Tr(S)}{\pi^2} \int_{\Omega_\mathfrak{a} \times \Omega_\mathfrak{a}} ( \bar{\partial}_w \tilde{f}(w)) 	( \bar{\partial}_z \tilde{f}(z))
			m'_{sc}(z)m'_{sc}(w) \\
		&= \frac{\Tr(S)}{4\pi^2} \pa{
			\int f(x) \frac{x}{\sqrt{4-x^2}} \diff x
		}^2
	\end{align*}
	And by \cite[(4.64)]{LanSos2020}, we have
	\begin{align*}
		\frac{1}{\pi^2} \int_{\Omega_\mathfrak{a} \times \Omega_\mathfrak{a}} ( \bar{\partial}_w \tilde{f}(w)) 	( \bar{\partial}_z \tilde{f}(z))
		\pa{
		\frac{1}{N^2} \sum_{j,a} s_{ja}^{(4)}  
			\partial_z \partial_w (  m_{sc}^2(z) m_{sc}^2(w)) 
		} 
		&=
		\frac{1}{N^2} \sum_{j,a}  
		\frac{s_{ja}^{(4)} }{\pi^2} \pa{
			\int_{-2}^2 f(x) \frac{2-x^2}{\sqrt{4-x^2}} \diff x
		}^2,
	\end{align*}
	which concludes the proof.
\end{proof}

To conclude this section, we study the contribution to the variance of the term proportional to $g(z,w)$. 
For $f$ with support in the bulk, \cite[Proposion 5.17]{LanLopSos2021} studies this term by Taylor expanding the main terms.
In our case, the error we make by Taylor expanding deteriorates so that the resulting formula holds in a window insufficiently small for our purpose. 
Therefore, in this step, we begin as in \cite[Proposion 5.17]{LanLopSos2021}, but
then specialize to the case of generalized Wigner matrices, allowing us to give a tractable expression for $V(f)$. 
We will later show that this term gives the main contribution. For the first part of our analysis, recall the operator $F$, see \eqref{eq:F}. 
Following Proposition \ref{prop:F} define $A(z,w)$ via the orthogonal decomposition of $F$,
	\begin{equation}
	\label{eq:decompose_F}
		F(z,w) = \lambda_1 (z,w) v(z,w) v^T(z,w) + A(z,w).
	\end{equation}

\begin{proposition}
\label{prop:var_GW}
	Let $f$ satisfy $\log(N) N^{-1 + \mathfrak{a}} \norm{f''}_1 \leq 1$ with $\mathfrak{a}$ as in Proposition \ref{prop:char_function}.
	Then in the generalized Wigner case, we have
		\[
			\frac{2}{\pi^2} \int_{\Omega_\mathfrak{a} \times \Omega_\mathfrak{a}} ( \bar{\partial}_w \tilde{f}(w)) 	( \bar{\partial}_z \tilde{f}(z))
			g(z,w) \diff z \diff w 
			=
			\frac{1}{\pi^2}\iint_{[-2,2]^2} \pa{ \frac{f(x)-f(y)}{x-y}    } f'(y) \frac{ \sqrt{4-y^2}  }{ \sqrt{4-x^2} } \diff y \diff x
				+ \e(f),
		\]	
	where
		\begin{equation}
		\label{eq:var_A}
			\e(f) = \sum_{i=1}^N A_{ii}\pa{ \frac{1}{\pi^2} \iint_{[-2,2]^2} (f(x)-f(y)) f'(y) m_{sc}(x) \frac{\sqrt{4-y^2}  }{\sqrt{ 4-x^2 }} \diff x \diff y} +
			\iint_{\R^2} (f(x)-f(y))f'(y) h_A(x,y)\diff x \diff y,
		\end{equation}
	for some bounded function $h_A$ which depends on the matrix $A$.
\end{proposition}
\begin{proof}
	For any $0 < \e < N^{-1+\mathfrak{a}}$, define the domains
	\[
		D_\e^{\pm} = \{
			x + \ii \eta \, : \, x \in I, \, \e \leq \pm \eta \leq 10, 
		\},
		\quad
		D_\e = D_\e^+ \cup D_\e^-.
	\]
	We first show that
	\begin{equation}
	\label{eq:omega_a_to_D_e}
		\frac{2}{\pi^2} \int_{\Omega_\mathfrak{a}^2} ( \bar{\partial}_z \tilde{f}(z) ) ( \bar{\partial}_w \tilde{f}(w) ) g(z,w) \rd z \rd w
		=
		\frac{2}{\pi^2} \int_{D_\e^2} ( \bar{\partial}_z \tilde{f}(z) ) ( \bar{\partial}_w \tilde{f}(w) ) g(z,w) \rd z \rd w
			+ \OO\pa{N^{-1 + \mathfrak{a}} \log N \norm{f''}_1}.
	\end{equation}
	Indeed, by \cite[Lemma 5.13]{LanLopSos2021}, for $\abs{z}, \abs{w} \leq C$, we have
	\[
		\abs{g(z,w)} \leq C \norm{\mathfrak{m}'(z)}_\infty \norm{\mathfrak{m}'(w)}_\infty 
			\norm{ 1 - S \mathfrak{m}(z)\mathfrak{m}(w)  }_{\ell^\infty \to \ell^\infty}^2,
	\]
	and recall from \eqref{eq:l2_and_l_infty} that
	\[
		\norm{
				1 - \mathfrak{m}(z)\mathfrak{m}(w) S
			}_{\ell^\infty \to \ell^\infty} 
			\leq 
			\frac{C}{ \abs{\im(w)}  \kappa^{-1/2}(w)  + \abs{\im(z)}  \kappa^{-1/2}(z)}.
	\]
	This gives
	\[
		\abs{g(z,w)} \leq \frac{C}{
			\pa{
				\im(w)\kappa^{-1/4}(w)\kappa^{1/4}(z) + \im(z)\kappa^{-1/4}(z)\kappa^{1/4}(w)
			}^2
		}.
	\]
	Considering separately the cases $\kappa(z), \kappa(w) \leq N^{-2/3}$, $\kappa(z) \leq N^{-2/3}, \kappa(w) \geq N^{-2/3}$ and
	$\kappa(z), \kappa(w) \geq N^{-2/3}$, we have
	\[
		\absa{
			\int_{0 < y_1 < N^{-1+\mathfrak{a}}} \int f''(x_1) y_1 \pa{f(x_2) + \ii f'(x_2)}\chi'(y_2) g(x_1+\ii y_1, x_2 + \ii y_2) \rd x_1 \rd x_2 \rd y_2 \rd y_1
		}
		=
		\OO\pa{
			N^{-1+\mathfrak{a}}\norm{f''}_\infty \norm{f'}_\infty 
		}.
	\]
	Again considering separately the cases $\kappa(z), \kappa(w) \leq N^{-2/3}$, $\kappa(z) \leq N^{-2/3}, \kappa(w) \geq N^{-2/3}$ and
	$\kappa(z), \kappa(w) \geq N^{-2/3}$, we have
	\[
		\absa{
			\int_{\e < \abs{y_1}, \abs{y_2} < N^{-1 + \mathfrak{a}}} 
			f''(x_1) y_1 f''(x_2) y_2 g(x_1 + \ii y_1, x_2 + \ii y_2)
			\rd x_1 \rd x_2 \rd y_1 \rd y_2
		}
		=
		\OO\pa{
			\log N \pa{N^{-1+\mathfrak{a}}   \norm{f''}_\infty  }^2  
		}.
	\]
	And finally, first applying Lemma \ref{lem:HS_bound} with $s=2$ to the integral in $z = x_1 + \ii y_1$, we have
	\[
		\absa{
			\int_{\e < \abs{y_2} < N^{-1+\mathfrak{a}}} f''(x_1) y_1 f''(x_2) y_2 g(x_1+\ii y_1, x_2 + \ii y_2) \rd x_1 \rd y_1 \rd x_2 \rd y_2
		}
		=
		\OO\pa{
			\log N \pa{  N^{-1+\mathfrak{a}} \norm{f''}_\infty  }^2
		},
	\]
	which establishes \eqref{eq:omega_a_to_D_e}.
	Now, re-write $h(z,w)$ from Proposition \ref{prop:variance_after_algebra} using the diagonal matrix $U(z,w)$ as
	\[
		h(z,w) = \Tr\pa{
			\mathfrak{m}'(z) \mathfrak{m}(z)^{-1} U^* (U^* - F)^{-1}
		},
		\quad
		U_{ii}(z,w) = \frac{ \mathfrak{m}_i(z)\mathfrak{m}_i(w)  }{  \absa{ \mathfrak{m}_i(z)\mathfrak{m}_i(w)}   },
	\]
	and applying the Sherman-Morrison formula, we have
	\[
		(U^* - F)^{-1} = (U^* - A)^{-1} - \frac{\lambda_1 (U^*-A)^{-1}vv^T(U^*-A)^{-1}}{ 1 - \lambda_1 v^T (U^*-A)^{-1}v }.
	\]
	By the resolvent expansion, we have 
	\[
		\Tr\pa{
			\mathfrak{m}'(z)\mathfrak{m}^{-1}(z) U^*(U^*-A)^{-1}
		}
		=
		\Tr \pa{
			\mathfrak{m}'(z) \mathfrak{m}(z)^{-1}
		}
		+
		\Tr \pa{
			\mathfrak{m}'(z) \mathfrak{m}(z)^{-1} AU
		}
		- \Tr \pa{
			\mathfrak{m}'(z) \mathfrak{m}(z)^{-1} A(U^*-A)^{-1} AU
		},
	\]
	which means that with
	\[
		\Gamma(z,w)
		=
		\Tr \pa{
			\frac{ \mathfrak{m}'(z) \mathfrak{m}(z)^{-1} U^*   \lambda_1 (U^*-A)^{-1}vv^T(U^*-A)^{-1}}{ 1 - \lambda_1 v^T (U^*-A)^{-1}v }
		}
		=
		\frac{
			\lambda_1 v^T(U^*-A)^{-1}m'(z)m^{-1}(z)U^*(U^*-A)^{-1}v
		}{
			1 - \lambda_1 v^T(U^*-A)^{-1}v
		},
	\]
	and
	\begin{equation}
	\label{eq:P(z,w)}
		P(z,w) 
		=
		\Tr \pa{
			\mathfrak{m}'(z) \mathfrak{m}(z)^{-1} AU
		}
		- \Tr \pa{
			\mathfrak{m}'(z) \mathfrak{m}(z)^{-1} A(U^*-A)^{-1} AU
		} \\
		- \Gamma(z,w),
	\end{equation}
	we have $h(z,w) = \Tr(\mathfrak{m}'(z) \mathfrak{m}^{-1}(z)) + P(z,w)$.
	Now, integrating by parts and using that $\Gamma(z,w)$ is analytic in $w$ and $z$, we have
	\begin{align*}
		\frac{2}{\pi^2} \int_{D_\e^+ \times D_\e^-} ( \bar{\partial}_z \tilde{f}(z) ) ( \bar{\partial}_w \tilde{f}(w) ) \partial_w g(z,w) \rd z \rd w 
		&=
		\frac{2}{\pi^2} \int_{D_\e^+ \times D_\e^-} ( \bar{\partial}_z \tilde{f}(z) ) ( \bar{\partial}_w \tilde{f}(w) ) \partial_w P(z,w) \rd z \rd w  \\
		&=
		-\frac{2}{\pi^2} \int_{D_\e^+ \times D_\e^-} ( \bar{\partial}_z \tilde{f}(z) ) ( \bar{\partial}_w \tilde{f}'(w) ) P(z,w) \rd z \rd w \\
		&=
		-\frac{2}{\pi^2} \int_{D_\e^+ \times D_\e^-} ( \bar{\partial}_z \bar{\partial}_w)[ \tilde{f}(z) \tilde{f}'(w)  P(z,w) ]\rd z \rd w \\
		&= -\frac{2}{\pi^2} \int_{D_\e^+ \times D_\e^-} ( \bar{\partial}_z \bar{\partial}_w)[ (\tilde{f}(z) -  \tilde{f}(w)) \tilde{f}'(w)  P(z,w) ]\rd z \rd w.
	\end{align*}
	Applying Green's formula,
	\[
		\int_{\Omega} \bar{\partial}_z F(z) \rd x \rd y = - \frac{\ii}{2} \int_{\partial \Omega} F(z) \rd z,
	\]
	we find
	\[
		\frac{2}{\pi^2} \int_{D_\e^+ \times D_\e^-} ( \bar{\partial}_z \tilde{f}(z) ) ( \bar{\partial}_w \tilde{f}(w) ) \partial_w g(z,w) \rd z \rd w 
		=
		-\frac{1}{2\pi^2} \int_{\partial\pa{D_\e^+ \times D_\e^-}}  (\tilde{f}(z) -  \tilde{f}(w)) \tilde{f}'(w)  P(z,w) \rd z \rd w.
	\]
	Letting $\e \to 0$, this gives
	\begin{multline}
	\label{eq:var_f}
		\lim_{\e \to 0} \pa{\frac{2}{\pi^2} \int_{D_\e^+ \times D_\e^-} ( \bar{\partial}_z \tilde{f}(z) ) ( \bar{\partial}_w \tilde{f}(w) ) \partial_w \Gamma(z,w) \rd z \rd w} \\
		= 
		\frac{1}{4} \int_{ \R^2 } (f(y) - f(x)) f'(y) [\Gamma(x+\ii 0,y + \ii 0) - \Gamma(x+\ii 0, y-\ii 0) - \Gamma(x-\ii 0, y+\ii 0) + \Gamma(x+\ii 0, y+\ii 0)] \rd x \rd y.
	\end{multline}
We now evaluate $\Gamma(x+\ii 0,y + \ii 0) - \Gamma(x+\ii 0, y-\ii 0)$. By the resolvent expansion (recall that by Proposition \ref{prop:F}
we have $\norm{A}_{\ell^2 \to \ell^2} < 1-c_g$), we have
	\[
		(U^*-A)^{-1} = U + X, \quad X = (U^*-A)^{-1}AU.
	\]
Therefore, suppressing the dependence on $z,w$ from our notation, writing $M$ for $\mathfrak{m}'(z)\mathfrak{m}^{-1}(z)$, 
$q_1(z,w) = \lambda v^T MUv$, and $q_2(z,w) = \lambda v^T Uv$,
we define $r_1(z,w)$ and $r_2(z,w)$ by
	\[
		\Gamma(z,w) = \frac{
			\lambda v^T MUv + \lambda v^T MX v + \lambda v^TXM v + \lambda v^TMU^* X v
		}{
			1 - \lambda v^TUv - \lambda v^T X v
		}
		=
		\frac{q_1(z,w) + r_1(z,w)}{1 - q_2(z,w) - r_2(z,w)}.
	\]
Now in the generalized Wigner case, we have $\mathfrak{m}_i = m_{sc}(z)$ for $i = 1, \dots, N$.
Therefore, $U$ is a constant multiple of the identity and since by definition 
$Av = 0$, we have $Xv = 0$. Furthermore, $XMv = 0$ because $M$ is a also a constant multiple of the identity. This gives
	\[
		\Gamma(z,w) = \frac{q_1(z,w)}{1-q_2(z,w)}, 
	\]
and 
	\[
		\Gamma(z,w) - \Gamma(z,\bar{w}) 
		=
		\frac{
			q_1(z,w) - q_1(z, \bar{w}) + q_1(z,\bar{w})q_2(z,w) - q_1(z,w)q_2(z,\bar{w})
		}
		{
			1 - q_2(z,w) - q_2(z,\bar{w}) + q_2(z,w)q_2(z,\bar{w})
		}.
	\]
Again leveraging the fact that we are working in the generalized Wigner case, we have the explicit formulas
	\[
		v(z,w) = \pa{
			\frac{1}{\sqrt{N}}, \dots, \frac{1}{\sqrt{N}}
		},
		\quad
		\lambda_1 = \abs{m_{sc}(z) m_{sc}(w)}.
	\]
This gives
	\[
		q_1(z,w) = \frac{m_{sc}(z)m_{sc}(w)}{2m_{sc}(z) + z}, 
		\quad
		q_2(z,w) = m_{sc}(z)m_{sc}(w), 
	\]
and therefore
	\begin{align*}
		\Gamma(z,w) - \Gamma(z,\bar{w}) 
		&=
		\frac{1}{2m(z) + z}
		\pa{ \frac{m_{sc}(z)m_{sc}(w) - m_{sc}(z)m_{sc}(\bar{w}) }{1 - m_{sc}(z)m_{sc}(w) - m_{sc}(z)m_{sc}(\bar{w}) + m_{sc}(z)^2\abs{m_{sc}(w)}^2} } \\
		&=
		\frac{1}{2m_{sc}(z) + z}
		\pa{ \frac{2m_{sc}(z)\im m_{sc}(w)}{1 + m_{sc}(z)^2\abs{m_{sc}(w)}^2 - 2m_{sc}(z)\re m(w)  } }.
	\end{align*}
Using $\abs{m_{sc}(E)} = 1$ for $E \in (-2,2)$, and $x\pm \ii 0 + 2m_{sc}(x\pm \ii 0) = \pm \ii \sqrt{4-x^2} \1_{\abs{x} < 2}$, this gives
	\begin{align*}
		\Gamma(x+\ii0, y+\ii 0) - \Gamma(x+\ii 0, y-\ii 0)
		&= 
		\frac{1}{\im m_{sc}(x+\ii 0)} \pa{
			\frac{2m_{sc}(x+\ii 0) \im m_{sc}(y + \ii 0)}{1 + m_{sc}(x+\ii 0)^2 - 2 m_{sc}(x+\ii 0) \re m_{sc}(y + \ii 0)}
		} \\
		&=
		\frac{2}{\im m_{sc}(x+\ii 0)} \pa{
			\frac{m_{sc}(x+\ii 0) \im m_{sc}(y + \ii 0)}{ m_{sc}(x+\ii 0) (y-x)}
		} \\
		&=
		\frac{\im m_{sc}(y+\ii 0)}{\im m_{sc}(x+\ii 0)} \frac{2}{y-x}.
	\end{align*}
Applying the same analysis to $\Gamma(x-\ii 0, y+\ii 0) - \Gamma(x-\ii 0, y + \ii 0)$, and substituting into \eqref{eq:var_f}, we have
	\[
		\lim_{\e \to 0} \pa{ \frac{2}{\pi^2} \int_{D_\e^+ \times D_\e^-} ( \bar{\partial}_z \tilde{f}(z) ) ( \bar{\partial}_w \tilde{f}(w) ) \partial_w \Gamma(z,w) \rd z \rd w }
		= 
		\frac{1}{\pi^2}\int_{-2}^2 \int_{-2}^2 \pa{ \frac{f(y)-f(x)}{y-x} } f'(y) \frac{\sqrt{4-y^2} }{\sqrt{4-x^2}} \rd x \rd y.
	\]
It remains to evaluate
	\[
		-\frac{1}{4\pi^2} \int_{ \R } (f(x) - f(y)) f'(y) [\Delta(x+\ii 0,y + \ii 0) - \Delta(x+\ii 0, y-\ii 0) - \Delta(x-\ii 0, y+\ii 0) + \Delta(x+\ii 0, y+\ii 0)] \rd x \rd y
	\]
where
	\[
		\Delta(z,w) = \Tr \pa{
			\mathfrak{m}'(z) \mathfrak{m}(z)^{-1} AU
		}
		- \Tr \pa{
			\mathfrak{m}'(z) \mathfrak{m}(z)^{-1} A(U^*-A)^{-1} AU
		}.
	\]
Since 
	\[
		\lim_{\eta \to 0} \Tr \pa{
			\mathfrak{m}'(z) \mathfrak{m}(z)^{-1} AU
		}
		=
		-2m_{sc}(x)
		\frac{\sqrt{4-y^2}  }{\sqrt{ 4-x^2 }} \1_{\abs{y} < 2}  \sum_{i=1}^N A_{ii},
	\]
the terms coming from $\Tr \pa{\mathfrak{m}'(z) \mathfrak{m}(z)^{-1} AU}$ contribute
	\[
		\sum_{i=1}^N A_{ii} \pa{ \frac{1}{\pi^2} \iint_{[-2,2]^2} (f(x)-f(y)) f'(y) m_{sc}(x) \frac{\sqrt{4-y^2}  }{\sqrt{ 4-x^2 }} \diff x \diff y}.
	\]
For the remaining term, note that by Proposition \ref{prop:F}, $h_A(z,w)=\Tr \pa{\mathfrak{m}'(z) \mathfrak{m}(z)^{-1} A(U^*-A)^{-1} AU} = \OO(1)$.
\end{proof}

\begin{lemma}
\label{lem:var_A_log}
	Choose $\gamma > 0$ and let $f(x) = \log\pa{(x-E)+\ii N^{-\gamma}}\chi(x)$, where
	$\chi(x)$ is a smooth function such that $\chi(x) \leq 1$, $\chi(x) \equiv 1$ for $x\in[-3,3]$ and $\chi(x) = 0$ 
	for $\abs{x} \geq 4$. Then with $\e(f)$ as in \eqref{eq:var_A}, we have
	$\e(\re f) = \OO(\gamma\log N)$, $\e(\im f) = \OO(\gamma\log N)$.
\end{lemma}
\begin{proof}
We bound the term depending on $h_A(x,y)$, for the real part of $f$. One can bound the remaining terms by a similar argument.
	\begin{align*}
		\iint_{[-4,4]^2} \frac{\absa{(f(x)-f(y)) f'(y)}}{\sqrt{\absa{4-x^2}}} \diff x \diff y
		&\leq 
		C\int_{E-N^{-\gamma}}^{E+N^{-\gamma}} \abs{f'(y)} \diff y 
			+ CN^{\gamma}\int_{-4}^{E-N^{-\gamma}} \int_{E-N^{-\gamma}}^{E+N^{-\gamma}} \frac{1}{(y-E)^2} \frac{(x-y)}{\sqrt{\absa{4-x^2}}} \diff x \diff y \\
		&\leq
		\OO(1) + C\int_{-4}^{E-N^{-\gamma}} \frac{1}{(y-E)^2} (E-y)\diff y 
		=
		\OO(\e \log N).
	\end{align*}
\end{proof}

 \subsection{Calculation of the Expectation. }
 \label{sec:expectation}

By making minor modifications to the proof of \cite[Lemma 5.22]{LanLopSos2021}, we prove the following Lemma that gives us a formula
for the expectation of linear statistics for sufficiently regular functions. 
We do not attempt to optimize the error terms in this section, which are sufficient for the proof of Theorem \ref{thm:log_corr_field}.
Our calculations below examine the real generalized Wigner case. The relevant modifications for the complex case do not introduce
any meaningful changes to the proof technique.
\begin{lemma}
\label{lem:E}
Let $\e > 0$, suppose
$f$ is such that $\norm{f''} \leq N^{1/2-3\e/2}$, and
let $W$ be a Wigner matrix. Then with $m = m_{sc}$, we have
	\begin{multline}
	\label{eq:expectation}
		\E(\Tr f(W)) = N \int f(x) \diff \varrho(x)
			+ \frac{1}{\pi} \int_{\Omega_\mathfrak{a}} \partial_{\bar{z}} \tilde{f}(z)
		\pa{
			-\Tr(S\mathfrak{m}\mathfrak{m'}) + \Tr\pa{ (1-S\mathfrak{m}^2)^{-1} S \mathfrak{m'}\mathfrak{m} }
				+ \frac{1}{N^2} \sum_{j,a} s_{aj}^{(4)} \mathfrak{m}'_j \mathfrak{m}_j \mathfrak{m}_a
		} \diff z \\
		+ \OO\pa{N^{-1+\mathfrak{a}+\e} \norm{f''}_1 + N^{-1/2+\e}( 1 + \norm{f''}_1   )^{3/4} + N^{-\mathfrak{a}/2}(1+\norm{f''}_1)}.
	\end{multline}
We take $\mathfrak{a} = 1/2 -\e/2$ to ensure that the error terms are $\oo(1)$ as $N \to \infty$.
\end{lemma}
\begin{proof}
	In the following we write $\e > 0$ for an arbitrary small constant that may vary from line to line.
	By the Helffer-Sj\"ostrand formula, with overwhelming probability, we have
	\[
		\Tr f(W) - N \int f(x) \diff\varrho(x) 
		=
		\frac{1}{2\pi} \int_{\Omega_\mathfrak{a}} \pa{ \partial_{\bar{z}} \tilde{f}(z) } \sum_{i=1}^N \pa{G_{ii}(z) - \mathfrak{m}_i(z)}
		+\OO\pa{ N^{-1+ \e + \mathfrak{a}} \norm{f''}_1}.
	\]
	By the cumulant expansion, Lemma \ref{lem:cumulant_expansion}, we have
	\[ 1+
		z \E[G_{ii}] = -\frac{1}{N} \sum_{a} s_{ia} \E[G_{ia}^2 + G_{ii}G_{aa}]
			+ \frac{s_{ii}}{N} \E[G_{ii}^2]
			+ \frac{1}{2N^{3/2}} \sum_{a} s_{ia}^{(3)} \E[\partial_{ia}^2 G_{ia}] 
			+ \frac{1}{6N^2} \sum_{a}s_{ia}^{(4)} \E[\partial_{ia}^3 G_{ia}] + \OO\pa{N^{-\frac{3}{2}+\e}}.
	\]
	We now bound the error terms on the right hand side. By \eqref{eq:local_law_entrywise}
	for $i \neq a$, with overwhelming probability, we have
	\[
		\partial_{ia}^2 G_{ia} = 6 G_{ia}\mathfrak{m}_i \mathfrak{m}_a + \OO\pa{N^\e \Psi^2(z)}.
	\]
	Therefore \eqref{eq:isotropic_local_law} gives
	\[
		\frac{1}{N^{3/2}} \sum_{a} s_{ia}^{(3)} \E[ \partial_{ia}^2 G_{ia}] = \OO\pa{ N^{-1/2+\e} \Psi^2(z)}.
	\]
	Similarly, for $i \neq a$, we have $\partial_{ia}^3 G_{ia}^3 = -6\mathfrak{m}_i^2 \mathfrak{m}_a^2 + \OO(N^\e \Psi(z))$ with overwhelming probability,
	and therefore
	\[
		\frac{1}{6N^2} \sum_{a} s_{ia}^{(4)} \E[ \partial_{ia}^3 G_{ia}] = -\frac{1}{N^2} \sum_{a} s_{ia}^{(4)} \mathfrak{m}_i^2 \mathfrak{m}_a^2
			+ \OO(N^{-1+\e} \Psi(z)).
	\]
	Continuing to bound terms, we have
	\[
		\frac{s_{ii}}{N} \E[G_{ii}^2] = \frac{s_{ii} \mathfrak{m}_i^2}{N} \OO\pa{N^{-1+\e} \Psi(z)}, 
	\]
	and
	\[
		\frac{1}{N} \sum_{a} s_{ia}\E[G_{ii}G_{aa}] = 
		\frac{1}{N} \sum_{a} s_{ia}\pa{ G_{ii} - \mathfrak{m}_i} \mathfrak{m}_a + \frac{1}{N} \sum_{a} s_{ia}\mathfrak{m}_i\pa{ G_{aa} - \mathfrak{m}_a}
			+ \frac{1}{N} \sum_{a} s_{ia} \mathfrak{m}_i \mathfrak{m}_a + \OO\pa{ N^\e (N\eta)^{-1} \Psi(z)  }.
	\]
	Now by definition we have $\frac{1}{N} \sum_a s_{ia} \E[G_{ia}^2] = \E[T_{ii}(z,z)]$, so applying \eqref{eq:qve} and re-arranging terms, we have 
	\[
		-\frac{1}{\mathfrak{m}_i} \E[G_{ii} - \mathfrak{m}_i] = 
		-\frac{1}{N} \sum_a \mathfrak{m}_i s_{ia} \E[G_{aa} - \mathfrak{m}_a] + \frac{s_{ii}}{N} \mathfrak{m}_i^2
			- \frac{1}{N^2} \sum_a s_{ia}^{(4)} \mathfrak{m}_i^2 \mathfrak{m}_a^2 - \E[T_{ii}(z,z)] + \OO\pa{N^\e (N\eta)^{-1} \Psi(z)},
	\]
	and so, 
	\begin{multline}
	\label{eq:E_Gii-mi}
		\E[G_{ii}] - \mathfrak{m}_i = - \sum_j (1 - S\mathfrak{m}^2)^{-1}_{ij} (S\mathfrak{m}^3)_{jj} +
			\sum_j (1-\mathfrak{m}^2S)^{-1}_{ij}[\mathfrak{m}(1-S\mathfrak{m}^2)^{-1} S\mathfrak{m}^2]_{jj} \\
			+ \frac{1}{N^2} \sum_{j,a} (1- \mathfrak{m}^2S)^{-1}_{ij} s_{ja}^{(4)} \mathfrak{m}_j^3 \mathfrak{m_a}^2 +
			\OO\pa{  \norm{(1-\mathfrak{m}^2S)^{-1}}_{\ell^\infty \to \ell^\infty} \pa{  N^\e (N\eta)^{-1} \Psi(z) + N^\e   
				\E\pa{ \abs{A_{xy}}^2 }^{1/2}  }},
	\end{multline}
	where we recall from Theorem \ref{thm:Txy} that
	\[
		\E\pa{ \abs{A_{xy}}^2 } \leq \frac{N^\e 
			\pa{
				\Psi(z)^2\Psi(w) + \Psi(z)\Psi(w)^2 + \frac{1}{N} \Psi(z)\Psi(w)
			}
		}{(\im(w) \kappa^{-1/2}(w) + \im(z) \kappa^{-1/2}(z))^2}.
	\]
	To simplify the main terms in \eqref{eq:E_Gii-mi},  
	note that $\sum_{i} (1-\mathfrak{m}^2S)^{-1}_{ij} = \mathfrak{m}'_j \mathfrak{m}_j^{-2}$, see \cite[(5.103)]{LanLopSos2021}.
	For the error terms, \eqref{lem:bound_1-m^2S_inv} 
	gives $N^{-1/2+\e} (1 + \norm{f''}_1)^{3/4}$, say, for the first term,
	and $N^{-\mathfrak{a}/2}(1+\norm{f''}_1)$, for the second one.
\end{proof}

In the generalized Wigner case, we can make some simplifications.

\begin{lemma}
\label{lem:E_simplified}
	Let
	\begin{align}
	\nonumber
		\delta_{{\rm GW}}(f) &= \frac{1}{\pi} \int_{\Omega_\mathfrak{a}} \partial_{\bar{z}} \tilde{f}(z)
		\pa{
			\Tr\pa{ (1-S\mathfrak{m}^2}^{-1} S \mathfrak{m'}\mathfrak{m}
		} \diff z \\
		&= -\frac{1}{2\pi} \int_{-2}^2 \frac{f(x)}{\sqrt{4-x^2}} \diff x + \frac{f(2)+ f(-2)}{4} + \OO(1).
	\label{eq:shift_GW}
	\end{align}
	In the generalized Wigner case, we have
	\begin{align}
		\nonumber
		&\frac{1}{\pi} \int_{\Omega_\mathfrak{a}} \partial_{\bar{z}} \tilde{f}(z)
		\pa{
			-\Tr(S\mathfrak{m}\mathfrak{m'}) + \Tr\pa{ (1-S\mathfrak{m}^2)^{-1} S \mathfrak{m'}\mathfrak{m} }
				+ \frac{1}{N^2} \sum_{j,a} s_{aj}^{(4)} \mathfrak{m}'_j \mathfrak{m}_j \mathfrak{m}_a
		} \diff z \\
		= 
		& \sum_{i=1}^N s_{ii} \pa{ \frac{1}{2\pi} \int_{-2}^2 f(x) \frac{2-x^2}{\sqrt{4-x^2}} \diff x}
		+ \delta_{{\rm GW}}(f)
		+ \frac{1}{2N^2} \sum_{j,a} s_{aj}^{(4)} \int_{-2}^2 f(x) \pa{\frac{x^4-4x^2+2}{\sqrt{4-x^2}}} \diff x.
	\end{align}
\end{lemma}
\begin{proof}
	For the terms other than $\delta_{\rm GW}(f)$, we have the explicit identities, 
	\[
		-\frac{1}{\pi} \int_{\C} \partial_{\bar{z}} \tilde{f}(z) m'_{sc}(z)m_{sc}(z) \diff z 
		=
		\frac{1}{2\pi} \int_{-2}^2 f(x) \pa{ \frac{2-x^2}{ \sqrt{4-x^2} } }\diff x
	\]
	and
	\[
		\frac{1}{\pi} \int_{\C} \partial_{\bar{z}} \tilde{f}(z) m_{sc}^3(z) m'_{sc}(z)
		=
		\frac{1}{2\pi} \int_{-2}^2 f(x) \pa{ \frac{x^2-4x^2+2}{ \sqrt{4-x^2} } }\diff x,
	\]
	as in \cite[(4.68) and (4.69)]{LanSos2020}.
	To evaluate $\delta_{\rm GW}(f)$, recall that in the generalized Wigner case, we have 
	$v(z,w) = \frac{1}{\sqrt{N}}(1, \dots, 1)$ and $\lambda_1 = \abs{m_{sc}(z) m_{sc}(w)}$,
	\[
		F(z,w) = \lambda_1(z,w) v(z,w)v(z,w)^T + A(z,w),
	\]
	where the above decomposition with $v \perp A$ defines $A$. Taking $w=z$ and applying
	the Sherman-Morrison formula, see \cite[(9.41)]{LanLopSos2021} for more detail, we have
	\[
		\Tr\pa{
			\pa{
				1- S\mathfrak{m}^2
			}^{-1} S \mathfrak{m}'\mathfrak{m}
		}
		=
		\Tr\pa{
			(U^*-A)^{-1} \abs{\mathfrak{m}} S \mathfrak{m}'\mathfrak{m}\mathfrak{m}^{-1}
		}
		+
		\frac{
			\lambda_1v^T(U^*-A)^{-1}\abs{\mathfrak{m}}S\mathfrak{m}'\abs{\mathfrak{m}}\mathfrak{m}^{-1}(U^*-A)^{-1}v
		}{
			1 - \lambda_1 v^T(U^*-A)^{-1}v
		}.
	\]
	We now simplify the above arguing as in Proposition \ref{prop:var_GW}. Recall that
	$(U^*-A)^{-1} = U + (U^*-A)^{-1}AU$. Using that $v\perp A$, we have
	\begin{equation}
	\label{eq:delta_GW_decomp}
		\Tr\pa{
			\pa{
				1- S\mathfrak{m}^2
			}^{-1} S \mathfrak{m}'\mathfrak{m}
		}
		=
		\Tr\pa{
			(U+(U^*-A)^{-1}AU) \abs{\mathfrak{m}} S \mathfrak{m}'\mathfrak{m}\mathfrak{m}^{-1}
		}
		+
		\frac{
			\lambda_1v^TU\abs{\mathfrak{m}}S\mathfrak{m}'\abs{\mathfrak{m}}\mathfrak{m}^{-1}Uv
		}{
			1 - \lambda_1 v^TUv
		}.
	\end{equation}
	Using that in the generalized Wigner case, $\mathfrak{m}_i = m_{sc}$ for all $i \in \llbracket 1, N \rrbracket$,
	we have
	\[
		\Tr\pa{
			U\abs{\mathfrak{m}} S \mathfrak{m}'\mathfrak{m}\mathfrak{m}^{-1}
		}
		+
		\frac{
			\lambda_1v^TU\abs{\mathfrak{m}}S\mathfrak{m}'\abs{\mathfrak{m}}\mathfrak{m}^{-1}Uv
		}{
			1 - \lambda_1 v^TUv
		}
		=
		-\frac{m'_{sc}(z)}{2m_{sc}(z) + z},
	\]
	and furthermore, 
	\[
		-\frac{1}{\pi} \int_\C \bar{\partial}_z \tilde{f}(z) \frac{m'_{sc}(z)}{2m_{sc}(z) + z} 
		=
		-\frac{1}{2\pi} \int_{-2}^2 \frac{f(x)}{\sqrt{4-x^2}} \diff x + \frac{f(2)+ f(-2)}{4},
	\]
	see \cite[(4.70)]{LanSos2020}. Finally, 
	the spectral gap for $A$, \eqref{eq:F_spectral_gap}, shows that
	the remaining term in \eqref{eq:delta_GW_decomp} contributes $\OO(1)$.
\end{proof}



\addcontentsline{toc}{section}{References}
\bibliographystyle{abbrv}
\bibliography{biblio}

\end{document}